\documentclass[preprint,number,12pt]{elsarticle}
\usepackage[english]{babel} 
\usepackage[latin1]{inputenc}
\usepackage[T1]{fontenc}
\usepackage{amsmath}
\usepackage{amssymb}
\usepackage{amsthm}
\usepackage{nicefrac}
\usepackage{schreiber_bell_polynomials}
\begin{document}
\begin{frontmatter}
\title{Inverse relations and reciprocity laws involving\\partial Bell polynomials and related extensions}
\author{Alfred Schreiber}
\address{Department of Mathematics and Mathematical Education,\\Europa-Universit{\"a}t Flensburg,\\Auf dem Campus 1, D-24943 Flensburg, Germany}
\ead{info@alfred-schreiber.de}
\begin{abstract}
The objective of this paper is, in the main, twofold: Firstly, to develop an algebraic setting for dealing with Bell polynomials and related extensions. Secondly, based on the author's previous work on multivariate Stirling polynomials (2015), to present a number of new results related to different types of inverse relationships, among these (1) the use of multivariable Lah polynomials for characterizing self-orthogonal families of polynomials that can be represented by Bell polynomials, (2) the introduction of `generalized Lagrange inversion polynomials' that invert functions characterized in a specific way by sequences of constants, (3) a general reciprocity theorem according to which, in particular, the partial Bell polynomials $B_{n,k}$ and their orthogonal companions $A_{n,k}$ belong to one single class of Stirling polynomials: $A_{n,k}=(-1)^{n-k}B_{-k,-n}$. Moreover, of some numerical statements (such as Stirling inversion, Schl\"omilch-Schl\"afli formulas) generalized polynomial versions are established. A number of well-known theorems (Jabotinsky, Mullin-Rota, Melzak, Comtet) are given new proofs. 
\end{abstract}
\end{frontmatter}
\section{Introduction}
The history of the Bell polynomials essentially originates with the problem of expanding a composite function $f\comp\varphi$ into a Taylor series. Fa\`{a} di Bruno's famous formula \cite{crai2005, john2002} for the higher derivatives of $f\comp\varphi $ provides a solution that takes on an elegant form using the (partial) Bell polynomials $B_{n,k}$:
\begin{equation}\label{fdb_intro}
	(f\comp\varphi)^{(n)}=\sum_{k=0}^n(f^{(k)}\comp\varphi)\cdot B_{n,k}(\varphi',\varphi'',\ldots,\varphi^{(n-k+1)}).
\end{equation}
A wider scope of interesting properties and applications has been opened up, among other things, through the investigations of Bell \cite{bell1934}, Riordan \cite{rior1958,rior1968}, and Comtet \cite{comt1974}. And even in the past two decades, research on Bell polynomials has still received considerable attention, as evidenced by a wealth of relevant literature, e.\,g. \cite{abbo2005,bigw2012,chhs2006,feng2020,figr2005,miho2008a,miho2010,schr2015,wtwa2009}. 

In this paper a unifying framework is developed for dealing with Bell polynomials and related extensions on an exclusively algebraic basis, which makes it easier to bring general concepts into play and to avoid \emph{ad-hoc} calculations as far as possible. 

Above all, the polynomial substitution introduced in this context proves to be an efficient tool. It can already be roughly read from the Fa\`{a} di Bruno formula that there is a correspondence between functions and polynomials. This basic observation is reflected in rules, according to which the composition of functions corresponds to a specific form of substituting a family of polynomials into a multivariable polynomial. 

Another key concept that is directly related to this is inversion. The following variant was a starting point for the author's previous study \cite{schr2015}. If we replace $\varphi$ in \eqref{fdb_intro} by its compositional inverse $\inv{\varphi}$, the question naturally arises whether $B_{n,k}(\inv{\varphi}',\inv{\varphi}'',\ldots,\inv{\varphi}^{(n-k+1)})$ can be represented by a polynomial expression in $\varphi',\varphi'',\ldots,\varphi^{(n-k+1)}$. This problem actually has a solution in the form of a Laurent polynomial $A_{n,k}$ such that
\begin{equation}\label{inversion_intro}
	B_{n,k}(\inv{\varphi}',\inv{\varphi}'',\ldots,\inv{\varphi}^{(n-k+1)})=A_{n,k}(\varphi',\varphi'',\ldots,\varphi^{(n-k+1)})\comp\inv{\varphi}.
\end{equation}

\sloppy
\noindent
Todorov \cite{todo1981} presumably was the first to establish at least a semi-ex\-plicit (determinantal) expression for $A_{n,k}$. A few years earlier Comtet \cite{comt1974} had represented the subfamily $A_{n,1}$ by means of $B_{n-1+k,k}(0,X_2,\ldots,X_n)$, \mbox{$0\leq k\leq n-1$} (see \eq{comtet_formula} below). In \cite{schr2015}, the Comtet formula was extended to the entire family $A_{n,k}$, which eventually also enabled an explicit representation in the form of a partition polynomial (cf. \theorem{mainresult} and \eq{stirling_partitionsum} below). A number of fundamental properties have also been proven, among which the orthogonality relation deserves special emphasis:
\begin{equation}\label{orthogonality_intro}
	\sum_{j=k}^n A_{n,j}B_{j,k}=\kronecker{n}{k}\qquad(1\leq k\leq n).
\end{equation}
\eq{orthogonality_intro} is the perfect analogue of the corresponding relation between the Stirling numbers, for in addition to the well-known fact that $B_{n,k}(1,\ldots,1)$ is equal to the Stirling number of the second kind $s_2(n,k)$, it indeed turns out that $A_{n,k}(1,\ldots,1)$ is equal to the \emph{signed} Stirling number of the first kind $s_1(n,k)$. In other interesting cases, too, identities that are satisfied by the Stirling numbers can be raised to the level of the corresponding polynomials. To name just a few examples, we mention a theorem by Khelifa and Cherruault \cite{khch2000} (Section~5.4), the introduction of multivariate Lah polynomials (Section~5.5), and the reciprocity theorem $A_{n,k}=(-1)^{n-k}B_{-k,-n}$ (Section~8). The latter shows that the partial Bell polynomials and their orthogonal companions ultimately belong to \emph{one} kind of multivariate \emph{Stirling polynomials}.\\[-1.5ex]

\fussy
The content is organized as follows:
\\[-1.5ex]

\textit{Section~2}.---\,In \cite{schr2015} it was sufficient to axiomatically describe the required algebra of functions. However, in order to continue and deepen these investigations it proves advantageous using a suitable standard model instead, here the algebra of formal power series over a field of characteristic zero. Section~2 summarizes the notations used and some of the results required from \cite{schr2015}. 

\enlargethispage{1ex}
\textit{Section~3}.---\,For any function $\varphi$ with Taylor coefficients $\varphi_0,\varphi_1,\varphi_2,\ldots$, a higher-order derivative operator $\Omega_n(\cdot\,|\,\varphi)$ is introduced that assigns to every function term $f$ (whether containing $\varphi$ or not) a polynomial $\Omega_n(f\,|\,\varphi)$ by replacing in $f^{(n)}(0)$ each occurence of $\varphi_j$ with $X_j$ ($j=0,1,2,\ldots$). Based on rules for evaluating $\Omega_n(f\,|\,\varphi)$ for composite function terms, we obtain several instances and classes of polynomials that are closely related to the $B_{n,k}$ and will play a crucial role in subsequent sections.

\textit{Section~4}.---\,The effect is studied the composition of functions has on polynomials, which depend in a specific way on those functions. Two composition rules will be proved in this context. The first one is the polynomial counterpart of a functional identity $h=f\comp g$. The second rule reformulates a theorem of Jabotinsky and is given here a new proof. It appears as an indispensable tool in many of the proofs to follow.

\textit{Section~5}.---\,This section deals with the class of B-representable polynomials that can be written in the form $Q_{n,k}=B_{n,k}(H_1,\ldots,H_{n-k+1})$ for all $n\geq k\geq 1$. Section~5.1 contains some criteria (necessary, sufficient) and the simple fact that a regular B-representable family of polynomials $Q_{n,k}$ has a unique orthogonal companion $\ortho{Q}_{n,k}$. In Section~5.2 it is shown that identities which are valid for Stirling numbers\,---\,such as the Stirling inversion or the Schl{\"o}milch-Schl{\"a}fli formulas\,---\,can be extended to regular B-representable polynomials. In the remaining subsections, special B-representable polynomial families and their orthogonal companions are examined. 

\textit{Section~6}.---\,As is well-known there is a close connection between the complete Bell polynomials and binomial sequences. Therefore it appears reasonable to apply to this area the results on Bell polynomials obtained up to then. Most identities known from the literature are direct consequences from our previously established statements. Special interest deserves the binomial sequence related to trees that has been studied by Knuth and Pittel \cite{knpi1989} and is given here a new explicit representation. Contributions are also made to the theorem of Mullin and Rota. We give a new proof, supplemented by a `both-or-none statement', which generalizes a lemma of Yang \cite{yang2008}.

\textit{Section~7}.---\,Lagrange's classical formula for the inversion of a power series has been generalized in numerous forms (see, e.\,g., Gessel \cite[Theorem~2.1.1]{gess2016}). In this section we are, loosely speaking, concerned with the problem of constructing multivariate polynomials, which convert a sequence of constants that characterizes a given function $f$ into the corresponding sequence of constants that characterizes the inverse function $\inv{f}$. Theorem~\ref{generalized_lagrange_inversion} explicitly describes the intricate form of these `generalized Lagrange inversion polynomials'. Some special cases are discussed in detail, in particular a corollary ($=$ Theorem~\ref{comtet_thm_F}) that proves to be equivalent to Comtet's Theorem~F \cite[p.\,151]{comt1974}. 

\textit{Section~8}.---\,This final section is about certain laws of reciprocity. The formulation of such laws requires that the domain of the indices of the polynomials involved can be extended to the integers. In order to achieve this, a new and straightforward procedure is proposed (based on the results from Section~3). The above-mentioned reciprocity law can thus be generalized to any regular B-representable families of polynomials $Q_{n,k}$. The main result (Theorem~\ref{reciprocity_B_rep}) then states that $\ortho{Q}_{n,k}=(-1)^{n-k}Q_{-k,-n}$ for all $n,k\in\integers$. Finally, two reciprocities for the potential polynomials are derived. The first implies the well-known Schur-Jabotinsky theorem (cited in \cite{gess2016}); the second extends Comtet's Theorem~C \cite[p.\,142]{comt1974} and is shown to be essentially a general version of a binomial transformation attributed to Melzak.
%
\section{Basic notions and preliminaries}
\subsection{Algebra of functions}
Throughout this chapter $\const$ is supposed to be a fixed commutative field of characteristic zero (so that $\rationals\subseteq\const$). We denote by $\funcs:=\const\lbrack\lbrack x\rbrack\rbrack$ the algebra of formal power series in $x$ with coefficients in $\const$. As customary, addition and scalar multiplication is defined coordinatewise, and product is defined by the Cauchy convolution. The elements $f\in\funcs$ will be called \textit{functions}. Besides their variable-free notation, we equally write $f(x)$ if the `argument' $x$ is to be referred in any way. So, the coefficient of $x^n$ in $f(x)=\sum_{n\geq0}c_n x^n$ is written $c_n=[x^n]f(x)$. We denote by $\funcs_0$ the set of all $f\in\funcs$ whose leading coefficient, $[x^0]f(x)$, is zero; furthermore we write $\funcs_1$ for the complement of $\funcs_0$ in $\funcs$. Then, $\funcs_1$ is the subring of the units of $\funcs$, that is, its elements are precisely those functions $f\in\funcs$ having a multiplicative inverse, from now on denoted by $f^{-1}$ or by $\nicefrac{1}{f}$ and called \emph{reciprocal of} $f$. 
\label{Translation rule for coefficients}
Another subalgebra of $\funcs$ is the ring of polynomials, $\poly:=\const[x]$. When dealing with multivariate Stirling polynomials, it proves appropriate to include certain elements of $\const(x)$, in particular $\nicefrac{1}{x}$. Therefore, we also admit as functions Laurent polynomials, i.\,e., elements of $\laurentpoly:=\const[x^{-1},x]$, and only once (in Section~8, \remark{schur_jabotinsky_thm}) even formal Laurent series over $\const$.

Next we introduce, as a third binary operation, the \emph{composition} $\comp$ of functions by the following rule of substitution:
\begin{equation}\label{def_composition}
	(f\comp g)(x):=f(g(x))=\sum_{n\geq0}([x^n]f(x))g(x)^n.
\end{equation}
This, however, defines but a partial operation. For almost all\footnote{\label{polynomial_case}The few times (cf. \remark{basic_CR_polynomialcase}) that we shall have to treat Laurent polynomials as composite functions $f\comp g$, we will limit ourselves to the case $f\in\poly$, $g\in\laurentpoly$.}
 of our purposes, two cases will suffice, in which (\ref{def_composition}) makes good sense:

\begin{align*}
f\in\funcs\text{ and }g\in\funcs_0\quad\text{(called 0-\emph{case}),}\\
f\in\laurentpoly\text{ and }g\in\funcs_1\quad\text{(called 1-\emph{case}).}
\end{align*}

\noindent
In either case we get a well-defined composite function $f\comp g$. The identity element is $\id=\id(x):=x$, satisfying $f\comp\id=\id\comp f=f$. Moreover, we have $f\comp(g\comp h)=(f\comp g)\comp h$, $(f+g)\comp h=(f\comp h)+(g\comp h)$, and $(f\cdot g)\comp h=(f\comp h)\cdot(g\comp h)$, whenever the terms involved are meaningful. By $\invfuncs$ we denote the set of $g\in\funcs_0$, for which a compositional inverse $f\in\funcs_0$ exists satisfying $f\comp g=g\comp f=\id$. We write $\inv{g}$ for the (unique) inverse of $g$. The set of invertible functions in $\funcs$ forms a non-abelian group with composition. Recall that $\invfuncs=\left\{g\in\funcs_0\,\mid\,[x]g(x)\neq0\right\}$. 

\begin{rem}\label{invertible_product}
We note the following simple, but useful statement: \emph{The product of two functions in $\funcs$ is invertible, if and only if either is invertible while the other is a unit.} From this follows immediately
\begin{equation}\label{invertible_function}
f\in\invfuncs\iff \id^{-1}\cdot f\in\funcs_1.
\end{equation}
\end{rem}
\label{Invertible product}
\noindent
Among the few special functions we shall be concerned with further on, consider the exponential function in $\funcs$ defined as usual by $\exp(x):=1+x+x^2/2!+x^3/3!+\dotsb$ (occasionally written in traditional notation $e^x$). Since $\exp$ belongs to $\funcs_1$, it has no compositional inverse, and $\log$ cannot be represented in $\funcs$. Alternatively, we define $\expm:=\exp-1$, which has Mercator's series for $\log(1+x)$ as its inverse in $\invfuncs$, namely
\[
	\logm(x):=\inv{\expm}(x)=\sum_{n\geq1}(-1)^n\frac{x^n}{n}.
\]
\label{Exponential function}
\label{Exponential and logarithm modified}

\vspace*{-2ex}
\noindent
Finally, we need our algebra of functions to be endowed with a \emph{derivation}, that is, an additive operator $D$ satisfying the Leibniz product rule $D(fg)=fD(g)+D(f)g$. A derivation is said to be \emph{normalized} when it takes $x$ to 1. A normalized derivation $D$ on $\poly$ agrees with the ordinary derivation known from calculus: $D(a_0+a_1x+a_2x^2+\dotsb +a_n x^n)=a_1+2 a_2 x+\dotsb +n a_n x^{n-1}$. Recall now that $D$ can be extended in a unique way to a derivation on the rational functions (quotient field of $\poly$) as well as to a derivation on $\funcs$ (see, e.\,g., \cite[p.\,31]{kolc1973}, and for some more details \cite{bend1980}). We still denote these extensions by $D$ (only some few times writing $f'$ instead of $D(f)$). As a consequence, the chain rule $D(f\comp g)=(D(f)\comp g)\cdot D(g)$ is satisfied in both the 0-case and the 1-case as well.

Iterating $D$ leads, in the usual way, to derivatives of higher order $D^n(f)$, $n\geq0$. For example, applied to $\expm$, $\logm$, and $\id^{-1}$ one has for all $n\geq1$:
{\allowdisplaybreaks
\begin{align}
	\label{higherderiv_exp}
	D^n(\expm)&=D^n(\exp)=\exp=1+\expm,\\
	\label{higherderiv_log}
	D^n(\logm)&=(-1)^{n-1}(n-1)!(1+\id)^{-n},\\
	\label{higherderiv_rec}
	D^n(\id^{-1})&=(-1)^n n!\id^{-(n+1)}.
\end{align}
}
Here, as in all similar cases, the term $(1+\id)^{-n}$ in (\ref{higherderiv_log}) is to be understood as the $n$th power of $(1+\id)^{-1}$ which, of course, is the function (geometric series) $1-\id+\id^2-\id^3+\dotsb\in\funcs_1$. Generally, every $f\in\funcs$ can be written in form of a Taylor series $f(x)=\sum_{n\geq0}D^n(f)(0)x^n/n!$.

We conclude this subsection with some remarks concerning the operator $\theta D$, which is a derivation whenever $\theta\in\funcs$. Comtet \cite{comt1973} called $\theta D$ \emph{Lie derivation} (with respect to the function $\theta$) and defined it by $\theta D(f)(x):=\theta(x)f'(x)$ (see also \cite{ mima1998}). Since long the special case $\theta=\id$ as well as the nice expansions resulting from repeatedly applying $xD$ to a function had been studied in the literature (see, e.\,g., \cite{boya2012,jord1939, masc2016,tosc1949}). However, leaving $\theta$ unspecified, one obtains far less satisfactory results (an issue we shall come back to in Section~5.6). As is shown in \cite{schr2015}, the situation takes a happy turn when, following Todorov \cite{todo1981}, one chooses $\theta$ to be of the form $D(\varphi)^{-1}$. Therefore, given any function $\varphi$ with $D(\varphi)(0)\neq0$, we define $\liederiv{\varphi}$ by
\begin{equation}\label{todorovs_choice}
	\liederiv{\varphi}(f):=D(\varphi)^{-1}D(f).
\end{equation}
The operator $D_{\varphi}$ may be called \emph{derivation with respect to} $\varphi$, because it acts on a function as if $\varphi$ were its independent variable, for instance:
 $\liederiv{\varphi}(1+3\varphi-\varphi^5)$ $=3-5\varphi^4$, or $\liederiv{\varphi}(e^{\varphi})=e^{\varphi}$. 

\begin{rem}
A peculiar detail: Just the classical operator $xD$ cannot be represented as a $D_{\varphi}$, since $\varphi=\log$ is not a function in our sense. Alternatively, one could use $D_{\log}$ as an abbreviation for \mbox{$D_{\logm}-D$} thus achieving the analogous effect a genuine operator $D_{\log}$ would have on, say, meromorphic functions.
\end{rem}

For an arbitrary $\varphi\in\invfuncs$ and $n\geq0$ it can easily be shown \cite[Proposition 2.2: \emph{Pourchet's formula}]{schr2015} that
\begin{equation}\label{pourchet_formula}
		\hliederiv{\varphi}{n}(f)=D^n(f\comp \inv{\varphi})\comp\varphi.
\end{equation}
Specializing $f=\id$ in (\ref{pourchet_formula}), we get the $n$th Taylor coefficient of the inverse function $\inv{\varphi}$ by just interchanging $\id$ and $\varphi$ in $\hliederiv{\id}{n}(\varphi)(0)$ ($=$ the $n$th Taylor coefficient of $\varphi$), i.\,e., we have $D^n(\inv{\varphi})(0)=\hliederiv{\varphi}{n}(\id)(0)$. Compared to the classical term $D^{n-1}((\id/\varphi)^n)(0)$, which Lagrange proposed to obtain the inverse of a power series, the iterative expression $\hliederiv{\varphi}{n}(\id)(0)$ proves simpler and, in most cases, requires significantly less computational amount.
\label{Computing inverse functions}

The next subsection will exhibit the important part the operator $D_{\varphi}$ plays in setting up a suitable framework for dealing with multivariate Bell polynomials and their related extensions.

\subsection{Multivariate polynomials}
We consider polynomials over the field $\const$ as well as some kinds of special Laurent polynomials from $\const[X_0^{-1},X_1^{-1},X_0,X_1,\ldots,X_n]$ which, for brevity, will be referred as polynomials, too. The partial derivation on $\const[X_1,\ldots,X_n]$ with respect to $X_j$ can be regarded as the (unique) normalized derivation on $\const'[X_j]$ where $\const'=\const[X_1,\ldots,X_{j-1},X_{j+1},\ldots,X_n]$; it extends in a unique way to the rational functions $\const(X_1,\ldots,X_n)$ and is denoted by $\pderop{X_j}$.

Given a polynomial $P$ in $X_1,X_2,\ldots$ and any sequence of polynomials $(Q_n)$, we denote by $P\comp Q_{\num}$ the result obtained by replacing in $P$ each indeterminate $X_j$ by $Q_j$ (the $\num$-sign marks the indexed place that corresponds to the indeterminate's index). Besides $P\comp Q_{\num}$, we also use the traditional notation $P(Q_1,Q_2,\ldots)$. When the sequence $Q_1,Q_2,\ldots$ is constant, the substitution will be called \emph{unification}, and $\num$ be dropped. In the (default) case $Q_1=Q_2=\cdots=1$ we plainly write $P\comp 1$, which gives the sum of the coefficients of $P$. 

A reasonable sense can also be attached to expressions of the form \mbox{$Q_{\num}\comp R_{\num}$}. Suppose $Q_n\in\const[X_1,\ldots,X_{q(n)}]$, with an integer $q(n)\geq 1$. We then take \mbox{$Q_{\num}\comp R_{\num}$} as an abbreviation for $Q_{\num}(R_1,\ldots,R_{q(\num)})$. The composition of polynomials obeys the \emph{associative law}:
\begin{equation}\label{composition_associative_law}
	(P\comp Q_{\num})\comp R_{\num}=P\comp (Q_{\num}\comp R_{\num}).
\end{equation}
Given any function $\varphi$, consider the mapping $P\mapsto P^{\varphi}:=P\comp D^{\num}(\varphi)$, which assigns to each polynomial $P$ the function obtained from $P$ by replacing $X_j$ by $D^{j}(\varphi)$ for each $j$. The following substitution rule is obvious:
\begin{equation}\label{substitution_rule}
	P(Q_1,\ldots,Q_n)^{\varphi}=P(Q_1^{\varphi},\ldots,Q_n^{\varphi}).
\end{equation}
\sloppy
\begin{rem}\label{identical_polynomials}
Let $P,Q$ be polynomials in $X_1,X_2,\ldots$ satisfying $P^{\varphi}(0)=Q^{\varphi}(0)$ for every $\varphi\in\funcs$. We then have $P=Q$. --- Since the constants in any given sequence $c_1,c_2,\ldots\in\const$ may be regarded as Taylor coefficients $c_j=D^{j}(\varphi)(0)$ of some function $\varphi$, one merely has to recall that distinct polynomials over an infinite integral domain cannot give rise to the same polynomial function.
\end{rem}

\fussy
In the sequel double-indexed families of polynomials play a major role. We use the notation $(U_{n,k})$ with $n,k\geq 0$ to mean the infinite family, and moreover $(U_{i,j})_{0\leq i,j\leq n}$ to denote its initial part in form of a quadratic matrix. The family (and the matrix as well) is called \emph{(lower) triangular} if $U_{n,k}=0$ for $n<k$. We say that polynomial families $(U_{n,k})$, $(V_{n,k})$ are \emph{orthogonal companions} (of each other) when they satisfy the \emph{orthogonality relation}
\begin{equation}\label{orthogonality_relation}
	\sum_{j=0}^{n}U_{n,j}V_{j,k}=\kronecker{n}{k}
\end{equation}
for all $n,k\geq0$ ($\kronecker{n}{k}$ Kronecker's symbol). In this case we equally write $U_{n,k}=\ortho{V}_{n,k}$ or $V_{n,k}=\ortho{U}_{n,k}$. The families involved will be called \emph{regular} inasmuch they necessarily have non-singular matrices.\smallskip

Next, we summarize without proofs some relevant material from the author's paper on multivariate Stirling polynomials the following sections are based upon. For details we refer the reader to \cite{schr2015}.

Unless otherwise stated we assume $\varphi$ to be an arbitrary function from $\invfuncs$, and $i,j,k,n,\ldots$ to be non-negative integers.

\begin{prop}\label{msp_1}\label{Todorov's determinant}
$\mspace{-5mu}$There exist polynomials $A_{n,k}\in\const[X_1^{-1}\mspace{-5mu},X_2,\ldots,X_{n-k+1}]$ such that
\[
	\hliederiv{\varphi}{n}=\sum_{k=0}^{n}A_{n,k}^{\varphi}\cdot D^{k}.
\]
The family $(A_{n,k})$ is triangular, regular, and uniquely determined by the differential recurrence
\[
	A_{n+1,k}=\frac{1}{X_1}\left(A_{n,k-1}+\sum_{j=1}^{n-k+1}X_{j+1}\pderiv{A_{n,k}}{X_j}\right),\quad A_{n,0}=\kronecker{n}{0}.
\]
\end{prop}

\noindent
The expansion of $\hliederiv{\varphi}{n}$ into a linear combination of the $D^0,D^1,\ldots,D^n$ can also be done in the reverse direction.

\begin{prop}\label{msp_2}
There exist polynomials $B_{n,k}\in\const[X_1,X_2,\ldots,X_{n-k+1}]$ such that
\[
	D^{n}=\sum_{k=0}^{n}B_{n,k}^{\varphi}\cdot \hliederiv{\varphi}{k}.
\]
The family $(B_{n,k})$ is triangular, regular, and uniquely determined by the differential recurrence
\[
	B_{n+1,k}=X_1\left(B_{n,k-1}+\sum_{j=1}^{n-k+1}X_{j+1}\pderiv{B_{n,k}}{X_j}\right),\quad B_{n,0}=\kronecker{n}{0}.
\]
\end{prop}

\noindent
Both polynomial families are closely connected, as becomes evident by the identity
\begin{equation}\label{connection_A_B}
	A_{n,k}^{\varphi}=B_{n,k}^{\inv{\varphi}}\comp\varphi.
\end{equation}
\sloppy
Taking $\inv{\varphi}$ for $\varphi$, it can be equivalently expressed in the form $B_{n,k}^{\varphi}=$ \mbox{$A_{n,k}^{\inv{\varphi}}\comp\varphi$}.

\fussy
To start with \proposition{msp_2}: $(B_{n,k})$ are the (\emph{partial}) \emph{Bell polynomials} Riordan \cite{rior1958} named in honor of  E.\,T. Bell, whose paper \cite{bell1934} is an extensive study of what Bell himself called `exponential polynomials', obviously motivated by the eigenvalue identity $D^n(e^{\varphi})=(B_{n,0}+B_{n,1}+\dotsb+B_{n,n})^{\varphi}\cdot e^{\varphi}$. The latter results as the special case $f=\exp$ of the famed \emph{Fa\`{a} di Bruno} (FdB) \emph{formula} 
\begin{equation}\label{fdb_formula}
	D^n(f\comp\varphi)=\sum_{k=0}^n (D^k(f)\comp\varphi)\cdot B_{n,k}^{\varphi},
\end{equation}
which appears as a byproduct in the proof of \proposition{msp_2}. Also well-known is the following `diophantine' representation 
\begin{equation}\label{bell_partitionsum}
	B_{n,k}=\sum_{\ptsind{n}{k}}\frac{n!}{r_1!r_2!\dotsm(1!)^{r_1}(2!)^{r_2}\dotsm}X_1^{r_1}X_2^{r_2}\dotsm X_{n-k+1}^{r_{n-k+1}},
\end{equation}
the sum to be taken over all elements $(r_1,\ldots,r_{n-k+1})$ of the set $\ptsind{n}{k}$ of $(n,k)$-partition types, i.\,e., sequences of integers $r_1,r_2,r_3,\ldots\geq 0$ such that $r_1+r_2+r_3+\dotsb=k$ and $r_1+2r_2+3r_3+\dotsm=n$. It follows that $B_{n,k}$ is homogeneous of degree $k$ and isobaric of degree $n$. 
\label{Multiplication rule}

Let us now turn to \proposition{msp_1} and to the most fundamental properties of the family $(A_{n,k})$. The first thing to notice here is the fact that $(A_{n,k})$ is the orthogonal companion of the Bell polynomials: $\ortho{B}_{n,k}=A_{n,k}$. 
\begin{rem}\label{inverse_function}
Let $\varphi(x)=c_1 x+c_2 x^2/2!+c_3 x^3/3!+\dotsb$, $c_1\neq0$. Since $B_{n,1}=X_n$, we obtain from \eq{connection_A_B} $A_{n,1}^{\varphi}(0)=(A_{n,1}^{\varphi}\comp\inv{\varphi})(0)=B_{n,1}^{\inv{\varphi}}(0)=D^n(\inv{\varphi})(0)$, that is, $A_{n,1}(c_1,\ldots,c_n)$ is the $n$th Taylor coefficient of the inverse function $\inv{\varphi}$.
\end{rem}

\sloppy
Thus we are led to the following compositional identities, which can be viewed as polynomial counterparts of \eq{connection_A_B}:
\begin{align}\label{A_B_representable}
	A_{n,k}&=B_{n,k}\comp A_{\num,1}=B_{n,k}(A_{1,1},\ldots,A_{n-k+1,1})\\
	\label{B_A_representable}
	B_{n,k}&=A_{n,k}\comp A_{\num,1}=A_{n,k}(A_{1,1},\ldots,A_{n-k+1,1}).
\end{align}
\fussy
Set $\widetilde{B}_{n,k}:=B_{n,k}(0,X_2,\ldots,X_{n-k+1})$, called \emph{associate} (partial) \emph{Bell polynomials} (the coefficients of which count only partitions with no singleton blocks). Then the main result\footnote{In \cite{schr2015} this statement was proved by an inductive argument. It has turned out later that an independent proof can be provided with the help of the reciprocity law (see \eq{reciprocity_A_B} and \remark{extended_bellpoly} below).} of \cite{schr2015} is as follows:
\begin{thm}\label{mainresult}
	Let $k,n$ be any integers with $1\leq k\leq n$. We have
	\[
	A_{n,k}=\sum_{j=k-1}^{n-1}(-1)^{n-1-j}\binom{2n-2-j}{k-1}X_1^{j-2n+1}\widetilde{B}_{2n-1-k-j,n-1-j}.
	\]
\end{thm}

\noindent
This identity strongly generalizes Comtet's famous formula for the coefficients of an inverted power series \cite[Theorem~E, p.\,150/151]{comt1974}, here obtained by taking $k=1$. \theorem{mainresult} can be used to derive some general polynomial identities of Schl\"omilch-Schl\"afli type (see Section 5.2) as well as the following explicit representation, which corresponds to that of $B_{n,k}$ in \eq{bell_partitionsum}:
\begin{equation}\label{stirling_partitionsum}
	A_{n,k}\mspace{-2mu}=\mspace{-2mu}X_1^{-(2n-1)}\mspace{-30mu}\sum_{\ptsind{2n-1-k}{n-1}}\mspace{-8mu}\frac{(-1)^{n-1-r_1}(2n-2-r_1)!}{(k-1)!r_2!r_3!\dotsm(2!)^{r_2}(3!)^{r_3}\dotsm}X_1^{r_1}X_2^{r_2}\dotsm
\end{equation}
As a consequence, $A_{n,k}$ is homogeneous of degree $n-1$ and isobaric of degree $2n-1-k$. --- Unification with both $A_{n,k}$ and $B_{n,k}$ yields
\begin{align}
	\label{unification_A}
	A_{n,k}\comp 1&=s_1(n,k)~~\text{(signed Stirling numbers of the first kind)},\\
	\label{unification_B}
	B_{n,k}\comp 1&=s_2(n,k)~~\text{(Stirling numbers of the second kind)}.
\end{align}
This may justify $A_{n,k}$ and $B_{n,k}$ being called \emph{multivariate Stirling polynomials} (MSP)\emph{ of the first and second kind}, respectively. While \eq{unification_B} has been well-known since long, (\ref{unification_A}) does add a new facet to what is already known about how the \emph{unsigned} Stirling numbers $c(n,k):=|s_1(n,k)|$ could be obtained by unification from certain polynomials (see Sections 5.2 and 5.5).
\smallskip
\label{Some specializations}

We finally state a property of the Bell polynomials that will be needed later.
\begin{lem}[Identity Lemma]\label{identity_lemma}
	Let $\varphi,\psi\in\invfuncs$ be any invertible functions such that $B^{\varphi}_{n,k}(0)=B^{\psi}_{n,k}(0)$ holds for all $n,k$ with $1\leq k\leq n$. Then $\varphi=\psi$.
\end{lem}
\begin{proof}
Since this statement does not occur in \cite{schr2015},  we will at least sketch a proof (by induction). Let $\varphi_1,\varphi_2,\varphi_3,\ldots$ and $\psi_1,\psi_2,\psi_3,\ldots$ be the Taylor coefficients of $\varphi$ and $\psi$, respectively, and let $n\geq1$ be an arbitrary integer. It suffices to show that $B_{n,n-k}(\varphi_1,\ldots,\varphi_{k+1})=B_{n,n-k}(\psi_1,\ldots,\psi_{k+1})$ implies $\varphi_j=\psi_j$ ($j=1,\ldots,k+1$) for $k=0,1,\ldots,n-1$. --- For $k=0$ we have $\varphi_1^n=B_{n,n}^{\varphi}(0)=B_{n,n}^{\psi}(0)=\psi_1^n$. Taking into account that $\varphi_1,\psi_1\neq0$ and that the equation must hold for an arbitrary $n$, this implies $\varphi_1=\psi_1$. Now the induction step $k\to k+1$ can be carried out using the formula
\begin{equation}\label{finalinstances_B}
	B_{n,n-k}=\sum_{i=1}^{k}X_{i+1}\cdot\mspace{-7mu}\sum_{j=0}^{n-k-1}\binom{n-j-1}{i}X_1^jB_{n-j-i-1,n-j-k-1}.
\end{equation}
This may be left to the reader. One derives \eq{finalinstances_B} from the recurrence in \proposition{msp_2} with the help of the identity
\begin{equation}\label{Bells_identity}
	\pderiv{B_{n,k}}{X_j}=\binom{n}{j}B_{n-j,k-1}\quad~~(1\leq j\leq n-k+1).
\end{equation}
\sloppy
See \cite[Corollary 4.4]{schr2015} (while \cite[Equation (5.1), p.\,266]{bell1934} is a version of \eq{Bells_identity} for the complete Bell polynomials $B_n$). --- It should be noticed here that the Taylor coefficients with index $\geq2$ occur only in linear form. We illustrate this for the induction step in the case $k=1$. Assuming $B_{n,n-1}(\varphi_1,\varphi_2)=B_{n,n-1}(\psi_1,\psi_2)$ we would have by \eq{finalinstances_B} $\binom{n}{2}\varphi_1^{n-2}\varphi_2=\binom{n}{2}\psi_1^{n-2}\psi_2$, which yields $\varphi_2=\psi_2$.
\end{proof}
%
\section{Polynomials from Taylor coefficients}
\fussy
Throughout this section $\varphi\in\funcs$ takes on the role of a fixed unspecified placeholder for an arbitrary function. The mapping $P\mapsto P^{\varphi}$ considered above that makes every polynomial into a function, could easily be reversed by replacing each $D^{j}(\varphi)$ in $P^{\varphi}$ by $X_j$. This can, more generally, be done in a slightly modified way for any function terms, which designate a function in $\funcs$. To achieve this, we will introduce a higher-order derivative operator $\Omega_n(\cdot\,|\,\varphi)$ assigning to each such term $f$ a polynomial $\Omega_n(f\,|\,\varphi)$ with the property
\begin{equation}\label{omega_property}
 \Omega_n(f\,|\,\varphi)^{\varphi}(0)=D^n(f)(0).
\end{equation}
Since there seems to be no real risk of misunderstandings, we will not distinguish in our notation between functions and function terms. It may suffice here, on a more informal basis and without switching explicitly to the meta-level of logical syntax, to remind that function terms are built up inductively from simpler ones. Thus, saying that $\varphi$ occurs in $f$ (or: $f$ contains $\varphi$) is to be understood by the usual method of recursion on terms in the sense that $\varphi$ occurs in $\varphi$ and, if $\varphi$ occurs in $f$ or in $g$, then $\varphi$ also occurs in composite terms like $f+g$, $f\cdot g$, $f^{-1}$, $f\comp g$, and $\inv{f}$. 
\begin{dfn}\label{def_omega}
Let $f$ be any function term. Then $\Omega_n(f\,|\,\varphi)$ is obtained by replacing for each $j\geq 0$ the occurences of $D^j(\varphi)(0)$ in $D^n(f)(0)$ with the indeterminate $X_j$. (Thus $\Omega_n(f\,|\,\varphi)$ results as a polynomial over $\const$ that satisfies \eq{omega_property}).
\end{dfn}
Some simple cases are immediate: 
\begin{equation}\label{function_without_phi}
	\Omega_n(f\,|\,\varphi)=D^n(f)(0),\text{~whenever $\varphi$ does not occur in $f$.}
\end{equation}
For example, one has $\Omega_n(c\,|\,\varphi)=\kronecker{n}{0}\cdot c$ ($c\in\const$), $\Omega_n(\id^k\,|\,\varphi)=\kronecker{n}{k}\cdot k!$, $\Omega_n(\expm\,|\,\varphi)=1-\kronecker{n}{0}$, $\Omega_n(\logm\,|\,\varphi)=s_1(n,1)$. It is also obvious that
\begin{equation}\label{omega_phi}
	\Omega_n(\varphi\,|\,\varphi)=X_n. 
\end{equation}
Since $D^n$ is additive and satisfies the Leibniz rule for higher-order derivatives, we have 
{\allowdisplaybreaks
\begin{align}
	\Omega_n(f+g\,|\,\varphi)&=\Omega_n(f\,|\,\varphi)+\Omega_n(g\,|\,\varphi)\label{omega_sum},\\
	\Omega_n(f\cdot g\,|\,\varphi)&=\sum_{k=0}^{n}\binom{n}{k}\Omega_{n-k}(f\,|\,\varphi)\Omega_{k}(g\,|\,\varphi).\label{omega_product}
\end{align}
In particular, $\Omega_n(c\cdot f\,|\,\varphi)=c\cdot\Omega_n(f\,|\,\varphi)$ holds for every $c\in\const$. Thus $\Omega_n(\cdot\,|\,\varphi)$ operates $\const$-linearly on $\funcs$.

For the composite term $f\comp g$ we have to treat the 0-case separately from the 1-case. In the former case we have $f\in\funcs$ and $g\in\funcs_0$. According to the FdB formula (\ref{fdb_formula}) and observing that $g(0)=0$, we obtain
{\allowdisplaybreaks
\begin{align}
\Omega_n(f\comp g\,|\,\varphi)&=\sum_{k=0}^{n}\Omega_{k}(f\,|\,\varphi)\cdot(B_{n,k}\comp\Omega_{\num}(g\,|\,\varphi)).\label{omega_composition_0case}
\intertext{%
In the 1-case we assume $g$ to be a function in $\funcs_1$, that is, we have $g(0)\neq0$ and $f\in\laurentpoly$. Consequently, $f$ cannot contain $\varphi$ and \eq{omega_composition_0case} has to be modified as follows:
}
 \Omega_n(f\comp g\,|\,\varphi)&=\sum_{k=0}^{n}D^{k}(f)(\Omega_0(g\,|\,\varphi))\cdot(B_{n,k}\comp\Omega_{\num}(g\,|\,\varphi)).\label{omega_composition_1case}
\end{align}
}
\begin{rem}
As mentioned above, \definition{def_omega} could be given a logically more rigorous form as an induction on terms; then, equations \eqref{function_without_phi} to \eqref{omega_composition_1case} would become part of the definition. If one wants also the operator $D$ to be involved in the formation of function terms, then the plausible equation $\Omega_n(D(f)\,|\,\varphi)=\Omega_{n+1}(f\,|\,\varphi)$ must be added.
\end{rem}
Next we turn to the unary operations taking $f$ to $f^{-1}$ (reciprocation) and to $\inv{f}$ (inversion). Either cases turn out to be derivable from the properties of $\Omega_n$ so far established, inasmuch both $f^{-1}$ and $\inv{f}$ can to a sufficient degree be explicitly expressed as functions in $\funcs$.\\

\noindent
\emph{Reciprocation}.\,---\,For any $f\in\funcs_1$ we have $f(0)\neq 0$, hence
{\allowdisplaybreaks
\begin{align*}
	\Omega_n(f^{-1}\,|\,\varphi)& = \Omega_n(\id^{-1}\comp f\,|\,\varphi)\\
	                &\hspace*{-3.4em} = \sum_{k=0}^{n}D^k(\id^{-1})(\Omega_0(f\,|\,\varphi))\cdot (B_{n,k}\comp\Omega_{\num}(f\,|\,\varphi))&\hspace*{-9.3em}\text{($\id^{-1}\in\laurentpoly$, \eq{omega_composition_1case})}\\
									&\hspace*{-3.4em} = \sum_{k=0}^{n}(-1)^k k!\,\Omega_0(f\,|\,\varphi)^{-(k+1)}B_{n,k}(\Omega_1(f\,|\,\varphi),\ldots,\Omega_{n-k+1}(f\,|\,\varphi)).
\end{align*}
}
This gives rise to
\begin{dfn}\label{reciprocal_poly}
	Let $\widehat{R}_n\in\const[X_0^{-1},X_1,\ldots,X_n]$, $n\geq 0$, be the (Laurent) polynomials
	\[\widehat{R}_n:=\sum_{k=0}^{n}(-1)^k k!\,X_0^{-(k+1)}B_{n,k},	\]
	henceforth called \emph{reciprocal polynomials}. We shall also use the special case $R_n:=\widehat{R}_n(1,X_1,\ldots,X_n)$.
\end{dfn}
\label{Table: Reciprocal polynomials}
\label{Reciprocal polynomials}
Together with the foregoing, we thus have obtained
\begin{prop}\label{reciprocal_function}
	$\Omega_n(f^{-1}\,|\,\varphi)=\widehat{R}_n(\Omega_0(f\,|\,\varphi),\ldots,\Omega_n(f\,|\,\varphi))$~~for all \mbox{$f\in\funcs_1$} and $n\geq 0$.
\end{prop}

\sloppy
\begin{rem}
We would have been arrived at the same result by applying \eq{omega_composition_0case} instead of \eq{omega_composition_1case} to the reciprocal function written as $f^{-1}=f(0)^{-1}((1+\id)^{-1}\comp f(0)^{-1}(f-f(0)))$.
\end{rem}
\medskip

\fussy
\noindent
\emph{Inversion}.\,---\,The inverse $\inv{f}$ of a function $f\in\invfuncs$ is given by its Taylor coefficients $D^n(\inv{f})(0)$, $n\geq 1$. We will use two options to represent these coefficients somewhat more explicitly. First, according to \remark{inverse_function} we have $D^n(\inv{f})(0)=A_{n,1}(D(f)(0),D^2(f)(0),\ldots,D^n(f)(0))$, which immediately yields
\begin{prop}\label{inverse_stirling}
	$\Omega_n(\inv{f}\,|\,\varphi)=A_{n,1}(\Omega_1(f\,|\,\varphi),\ldots,\Omega_n(f\,|\,\varphi))$~~for all $f\in\invfuncs$ and $n\geq 1$.
\end{prop}

\noindent\sloppy
Our second option is to make use of the Lagrange formula $D^n(\inv{f})(0)=D^{n-1}((\id/f)^n)(0)$. It follows from \eq{invertible_function} that $\id/f=(\id^{-1}\cdot f)^{-1}\in\funcs_1$, whence $\Omega_0(\id/f\,|\,\varphi)=(\id/f)(0)\neq 0$. Therefore, when passing over from $D^n(\inv{f})(0)$ to $\Omega_n(\inv{f}\,|\,\varphi)$ we again have to apply \eq{omega_composition_1case}:
{\allowdisplaybreaks
\begin{align*}
	\Omega_n(\inv{f}\,|\,\varphi)& = \Omega_{n-1}(\id^n\comp\tfrac{\id}{f}\,|\,\varphi)\\
	                 & = \sum_{k=0}^{n-1}D^k(\id^n)(\Omega_0(\tfrac{\id}{f}\,|\,\varphi))\cdot(B_{n-1,k}\comp\Omega_{\num}(\tfrac{\id}{f})\,|\,\varphi)\\
									 & = \sum_{k=0}^{n-1}\powerfall{n}{k}\Omega_0(\tfrac{\id}{f}\,|\,\varphi)^{n-k}\,B_{n-1,k}(\Omega_1(\tfrac{\id}{f}\,|\,\varphi),\ldots,\Omega_{n-k}(\tfrac{\id}{f}\,|\,\varphi)),
\end{align*}
}

\fussy\noindent
where $\powerfall{n}{k}$ means the falling power $n(n-1)\cdots(n-k+1)$, $k\geq 1$, and $\powerfall{n}{0}=1$.\,---\,This shows that $\Omega_n(\inv{f}\,|\,\varphi)$ can be represented by another class of polynomials.
\begin{prop}\label{inverse_lagrange}
	$\Omega_n(\inv{f}\,|\,\varphi)=\widehat{T}_n(\Omega_{0}(\frac{\id}{f}\,|\,\varphi),\ldots,\Omega_{n-1}(\frac{\id}{f}\,|\,\varphi))$ for $n\geq 1$, where $\widehat{T}_n\in\const[X_0,\ldots,X_{n-1}]$ is defined by
\[\widehat{T}_n:=\sum_{k=0}^{n-1}\powerfall{n}{k}X_0^{n-k}B_{n-1,k}\]
and called $n$th tree polynomial~\emph{(cf. part (iii) of \remark{power_unification})}. For later use we set $T_n:=\widehat{T}_n(1,X_1,\ldots,X_{n-1})$.
\end{prop}
\label{Table: Tree polynomials}
\begin{cor}[\proposition{inverse_stirling} and \ref{inverse_lagrange}]\label{corollary_stirling}
For every $n\geq 1$
	\begin{align*}	
	\text{\emph{(i)~~~}}&A_{n,1}=\widehat{T}_n(\widehat{R}_0(\tfrac{X_1}{1}),\widehat{R}_1(		\tfrac{X_1}{1},\tfrac{X_2}{2}),\ldots,\widehat{R}_{n-1}(\tfrac{X_1}{1},\tfrac{X_2}{2},\ldots,\tfrac{X_n}{n}))\\	
	\text{\emph{(ii)~~~}}&A_{n,1}(X_0,2X_1,3X_2,\ldots)=\widehat{T}_n(\widehat{R}_0,\widehat{R}_1\ldots,\widehat{R}_{n-1})\\
	\text{\emph{(iii)~~~}}&A_{n,1}(1,2X_1,3X_2,\ldots)=T_n(R_1,R_2,\ldots,R_{n-1})
	\end{align*}
\end{cor}
\begin{proof}
(i). We assume $f(x)=\sum_{n\geq 1}f_n x^n/n!$ to be an invertible function, in which $\varphi$ does not occur. Thus we have $\Omega_n(f\,|\,\varphi)=f_n$ and $(\id^{-1}\cdot f)(x)=\sum_{n\geq 0}\frac{f_{n+1}}{n+1}\frac{x^n}{n!}$, whence by \proposition{reciprocal_function}
\begin{equation}\label{omega_xdivf}
	\Omega_n(\tfrac{\id}{f}\,|\,\varphi)=\widehat{R}_n(\tfrac{f_1}{1},\tfrac{f_2}{2},\ldots,\tfrac{f_{n+1}}{n+1}).
\end{equation}
Let $H_n(X_1,\ldots,X_n)$ temporarily stand for the polynomial on the right-hand side of (i). Then, combining \proposition{inverse_stirling} with \ref{inverse_lagrange} and \eq{omega_xdivf} we get
\begin{equation*}
	A_{n,1}^{f}(0)=A_{n,1}(f_1,\ldots,f_n)=H_n(f_1,\ldots,f_n)=H_n^{f}(0).
\end{equation*}
The argument from \remark{identical_polynomials} now yields the assertion (i): $A_{n,1}=H_n$.

(ii) and (iii) are immediate consequences of (i). 
\end{proof}

Later on, part (i) of \corollary{corollary_stirling} will be extended to the entire family $(A_{n,k})$ (see \corollary{family_A} below).
\subsection{Fa\`{a} di Bruno polynomials}
Let $f\in\funcs$ be any function, which does not contain $\varphi$. For $n\geq 0$ we define the $n$\emph{th FdB polynomial of} $f$ (for short, $f$-\emph{polynomial}) by
\begin{equation}\label{fdbpoly_0case}
	\Phi_n(f):=\Omega_n(f\comp\varphi\,|\,\varphi)=\sum_{k=0}^n D^k(f)(0)B_{n,k}.
\end{equation}
The sum on the right-hand side results from \eq{omega_composition_0case}, where $\varphi\in\funcs_0$ is assumed. This definition makes also sense in the case $f\in\laurentpoly$, $\varphi\in\funcs_1$. It is clear then by \eq{omega_composition_1case} that the Taylor coefficient $D^k(f)(0)$ in \eq{fdbpoly_0case} has to be replaced with the Laurent polynomial $D^k(f)(X_0)$. 

This definition is still ambiguous for $f\in\poly$, insofar $f\comp\varphi$ can be evaluated with both $\varphi\in\funcs_0$ and $\varphi\in\funcs_1$. To distinguish the two cases, we shall indicate the choice $\varphi\in\funcs_1$ by writing $\widehat{\Phi}_n(f)$ instead of $\Phi_n(f)$, if necessary. 

We check at once that $\Phi_n$ is a $\const$-linear operator obeying the product rule \eqref{omega_product}. The case of a composite function will be treated in Section~4.
\begin{exms}\label{examples_fdbpoly}
Well-known examples of  FdB polynomials are

\sloppy	
(i): the \emph{exponential (or: complete Bell) polynomials} $B_n:=\Phi_n(\exp)=\sum_{k=0}^n D^k(\exp)(0)B_{n,k}$ $=\sum_{k=0}^n B_{n,k}$, $n\geq 0$. We have $\Phi_n(\expm)=B_n$ for $n\geq 1$, however $\Phi_0(\expm)=0\neq B_0=1$.
	
\fussy
(ii): the \emph{logarithmic polynomials} $L_n:=\Phi_n(\logm)=\sum_{k=1}^n D^k(\logm)(0)B_{n,k}$ $\overset{\eqref{higherderiv_log}}{=}\sum_{k=1}^n(-1)^{k-1}(k-1)!B_{n,k}$ (cf. Comtet \cite[p.\,140]{comt1974}).
\label{Table: Logarithmic polynomials}

(iii): the \emph{potential polynomials} $\widehat{P}_{n,k}:=\widehat{\Phi}_n(\id^k)$ \cite[p.\,141]{comt1974}. Expanding \mbox{$\Omega_n(\id^k\comp\varphi\,|\,\varphi)$} according to \eq{omega_composition_1case} gives 
\begin{equation}\label{potential_polynomial_1}
	\widehat{P}_{n,k}=\sum_{j=0}^n \powerfall{k}{j} X_0^{k-j}B_{n,j}.
\end{equation}
\label{Table: Potential polynomials}

\noindent\sloppy
The reciprocal polynomials and the tree polynomials form subfamilies of ($\widehat{P}_{n,k}$), since we have $\widehat{R}_n=\widehat{P}_{n,-1}$ and $\widehat{T}_n=\widehat{P}_{n-1,n}$.

(iv): The fact that powers and falling powers are connected by $(x)_k=\sum_{j=0}^k s_1(k,j)x^j$ suggests introducing \emph{factorial polynomials} by a definition analogous to that of (iii) above: $\widehat{F}_{n,k}:=\widehat{\Phi}_n((\id)_k)$; additionally set $F_{n,k}:=\widehat{F}_{n,k}(1,X_1,X_2,\ldots)$. Since $\widehat{\Phi}_n$ is $\const$-linear, we immediately obtain
\begin{align}
	\label{factorial_potentialpoly}
	\widehat{F}_{n,k}&=\sum_{j=0}^k s_1(k,j)\widehat{P}_{n,j},\\
\intertext{hence by Stirling inversion (see e.\,g. \cite[Proposition 1.4.1b]{stan1986})}
	\label{potential_factorialpoly}
	\widehat{P}_{n,k}&=\sum_{j=0}^k s_2(k,j)\widehat{F}_{n,j}.\\
\intertext{In a more explicit form the factorial polynomials can be written}
	\widehat{F}_{n,k}&=\sum_{r=0}^n\bigg(\sum_{j=r}^k s_1(k,j)(j)_r X_{0}^{j-r}\bigg)B_{n,r}.\notag
\end{align}
\label{Factorial polynomials and their coefficient sums}
\end{exms}

\fussy\
Though perhaps not apparent at first glance, the potential polynomials are indeed closely related to the Bell polynomials; so it is worth examining them in more detail. We start with restricting $k$ to non-negative values (thereby choosing $\varphi\in\funcs_0$). This leads to a special instance of $\widehat{P}_{n,k}$:
\begin{align}\label{power_0case}
	\Phi_n(\id^k)&=\sum_{j=1}^n\Omega_j(\id^k\,|\,\varphi)B_{n,j}=\sum_{j=1}^n\kronecker{j}{k}\cdot k!B_{n,j}\notag\\
	             &=k!B_{n,k}=\widehat{P}_{n,k}(0,X_1,\ldots,X_n).
\end{align}

\noindent
Now consider the function (geometric series) $\gamma:=(1-\id)^{-1}=1+\id+\id^2+\ldots$, which belongs to $\funcs_1$. For the corresponding FdB polynomials we get $\Phi_n(\gamma)=\sum_{k=0}^n k!B_{n,k}$, which may be called \emph{geometric polynomials} (in view of the fact that their univariate version $\Phi_n(\gamma)\comp x$ is well-known in the literature under the same name; see, e.\,g., \cite{karg2017}). 

Looking at \eq{power_0case} it would be justified calling $\widehat{P}_{n,k}(0,X_1,\ldots,X_n)$ partial geometric polynomials. In general, let $f\in\funcs$ be any function and $f_k$ denote its $k$th Taylor coefficient $D^k(f)(0)$, $k\geq 0$. Then, according to what is customary with the exponential (Bell) polynomials, we shall call $\Phi_n(\tfrac{f_k}{k!}\id^k)$ \emph{partial} $f$-\emph{polynomials}. In just this manner we immediately obtain from \eq{power_0case} $B_{n,k}$ as partial $\exp$-polynomials: $B_{n,k}=\Phi_n(\expm_k)$, where $\expm_k:=\id^k/k!$. The partial Bell polynomials itself thus prove to be FdB polynomials.
   
Another instance of $\widehat{P}_{n,k}$ is of interest, too. Taking $(1+\id)^k$ in place of $\id^k$ we obtain
\begin{equation}\label{power_1case}
	\Phi_n((1+\id)^k)=\sum_{j=0}^n j!\binom{k}{j}B_{n,j}=\widehat{P}_{n,k}(1,X_1,\ldots,X_n)=:P_{n,k}.	
\end{equation}
From this follows by virtue of binomial inversion (cf. \cite[Chapter III]{aign1979}) the well-known formula
\begin{equation}\label{bertrand_formula}
	B_{n,k}=\frac{1}{k!}\sum_{j=0}^k (-1)^{k-j}\binom{k}{j}P_{n,j}.
\end{equation}
See \cite[p.\,156]{comt1974} and \cite{todo1981}, where \eq{bertrand_formula} is ascribed to J. Bertrand (1864).
\begin{rem}[Unifications]\label{power_unification}
~
\sloppy
\begin{enumerate}[(i)]
	\item Observing \eq{unification_B} one has $\widehat{P}_{n,k}\comp 1=P_{n,k}\comp 1=\sum_{j=0}^n\powerfall{k}{j}(B_{n,j}\comp 1)=\sum_{j=0}^n\powerfall{k}{j} s_2(n,j)=k^n$.
	\item Unification on both sides of \eq{bertrand_formula} immediately yields the well-known formula $s_2(n,k)=\frac{1}{k!}\sum_{j=1}^k (-1)^{k-j}\binom{k}{j}j^n$ (see \cite[Remark 4.1]{schr2015}, and from a historical perspective \cite[Theorem~1]{boya2012}).
	\item For all $n\geq1$ we have $\widehat{T}_n\comp 1=T_n\comp 1=n^{n-1}=$ the number of labeled, rooted trees on the vertex set $\left\{1,\ldots,n\right\}$.
	\item Set $p(n):=\Phi_n(\gamma)\comp 1$. Then $p(n)=\sum_{k=0}^n k! s_2(n,k)=$ the number of \emph{preferred arrangements}, that is, the total number of ordered partitions of $\left\{1,\ldots,n\right\}$ (see \cite{schr2018} for a Dobi\'{n}ski type series representing these numbers).
	\item Note that the numbers $[n]_k:=\widehat{F}_{n,k}\comp 1=F_{n,k}\comp 1$ do not agree with the falling powers $(n)_k$. Instead, we obtain from \eqref{factorial_potentialpoly} and \eqref{potential_factorialpoly}	$[n]_k=\sum_{j=0}^k s_1(k,j)j^n$ and $k^n=\sum_{j=0}^k s_2(k,j)[n]_j$. In \cite[p.\,253]{schr2014} it is shown that
\[
[n]_k=\sum_{j=1}^{\min(k,n)}(-1)^{k-j}j!s_2(n,j)(c(k-1,j-1)-c(k-1,j)).		
\]
\end{enumerate}
\end{rem}

\fussy
\begin{prop}[Convolution identities]\label{convolution_identities}
Let $n,r,s$ be any integers, $n\geq0$.
\begin{align*}
\text{\emph{(i)}~~~}&\widehat{P}_{n,r+s}=\sum_{k=0}^n\binom{n}{k}\widehat{P}_{n-k,r}\widehat{P}_{k,s}\qquad(\text{also valid for~}P\text{~in place of~}\widehat{P}),\notag\\ 
\text{\emph{(ii)}~~~}&\binom{r+s}{r}B_{n,r+s}=\sum_{k=0}^n\binom{n}{k}B_{n-k,r}B_{k,s}\hspace*{6.6em}(r,s\geq 0).
\end{align*}
\end{prop}
\begin{proof}
(i). Evaluate $\widehat{P}_{n,r+s}=\widehat{\Phi}_n(\id^r\cdot\id^s)$ by means of the Leibniz rule \eqref{omega_product}. The same identity for $P_{n,r+s}$ is then obtained for $X_0=1$. --- (ii) Replace $X_0$ by 0 in (i) and apply \eq{power_0case}.
\end{proof}
\begin{rem}
According to Birmajer, Gil and Weiner \cite{bigw2012} the convolution formula (ii) seems to be established for the first time by Cvijovi\'{c} \cite{cvij2011}.
\end{rem}

\begin{prop}\label{comtet_logpoly}
	$L_n=\sum_{j=1}^n (-1)^{j-1}\frac{1}{j}\binom{n}{j}P_{n,j}\qquad(n\geq 1)$.
\end{prop}
\begin{proof}
In the linear combination representing $L_n$ (see \examples{examples_fdbpoly}\,(ii)) we replace $B_{n,k}$ by the right-hand side of \eq{bertrand_formula}. This gives
\begin{equation*}
	L_n=\sum_{k=1}^{n}\sum_{j=0}^{k}(-1)^{j-1}\frac{1}{k}\binom{k}{j}P_{n,j}=\sum_{j=1}^{n}\bigg((-1)^{j-1}P_{n,j}\underbrace{\sum_{k=j}^{n}\frac{1}{k}\binom{k}{j}}_{(\ast)}\bigg).
\end{equation*}
One easily verifies that (\textasteriskcentered) is equal to $\frac{1}{j}\binom{n}{j}$. 
\end{proof}
\begin{rem}
The statement of \proposition{comtet_logpoly} is to be found in Comtet \cite[p.\,156]{comt1974}, however flawed by missing the binomial factor.
\end{rem}

Now we will give a more general version of the statements in the Propositions \ref{inverse_stirling} and \ref{inverse_lagrange}.

\begin{thm}\label{inverse_power}
Let $f\in\invfuncs$ and $n,k\in\integers$ with $1\leq k\leq n$. Then we have
\begin{align*}
\text{\emph{(i)}~~~}\Omega_n(\inv{f}^k\,|\,\varphi)&=k!A_{n,k}(\Omega_1(f\,|\,\varphi),\ldots,\Omega_{n-k+1}(f\,|\,\varphi))
,\\
\text{\emph{(ii)}~~~}\Omega_n(\inv{f}^k\,|\,\varphi)&=k!\binom{n-1}{k-1}\widehat{P}_{n-k,n}\comp\widehat{R}_{\num}(\tfrac{\Omega_1(f\,|\,\varphi)}{1},\tfrac{\Omega_2(f\,|\,\varphi)}{2},\ldots).
\end{align*}
\end{thm}

\begin{proof}
(i). Observing $\inv{f}(0)=0$ and applying \eq{omega_composition_0case} to $\inv{f}^k=\id^k\comp \inv{f}$ we see that the left hand-side of (i) becomes $\sum_{j=0}^n\Omega_j(\id^k\,|\,\varphi)\cdot(B_{n,j}\comp\Omega_{\num}(\inv{f}\,|\,\varphi))$. Clearly $\Omega_j(\id^k\,|\,\varphi)=D^j(\id^k)(0)=\kronecker{j}{k}k!$. We use \proposition{inverse_stirling} to evaluate $\Omega_{\num}(\inv{f}\,|\,\varphi)$ and get the asserted by \eq{A_B_representable}.

\sloppy
(ii). We proceed in a way similar to that in the argument leading to \proposition{inverse_lagrange}. Let $g$ denote the function $\id/f\in\funcs_1$. We start from the general Lagrange inversion formula $n[x^n]\inv{f}(x)^k=k[x^{n-k}]g(x)^n$ \cite[Theorem~5.4.2]{stan1999}, which may also be written as $\tfrac{1}{n!}D^n(\inv{f}^k)(0)=\tfrac{k}{n(n-k)!}D^{n-k}(g^n)(0)$. Transforming now the Taylor coefficients herein into $\Omega_n$-terms according to \definition{def_omega}, we obtain by \eq{omega_composition_1case} (1-case) and by \eq{potential_polynomial_1}
\begin{align*}
	\Omega_n(\inv{f}^k\,|\,\varphi)&=\tfrac{k(n-1)!}{(n-k)!}\Omega_{n-k}(\id^n\comp g\,|\,\varphi)\\
	&=k!\binom{n-1}{k-1}\sum_{j=0}^{n-k}D^j(\id^n)(\Omega_0(g\,|\,\varphi))\cdot(B_{n,j}\comp\Omega_{\num}(g\,|\,\varphi))\\
	&=k!\binom{n-1}{k-1}\widehat{P}_{n-k,n}(\Omega_{0}(g\,|\,\varphi),\ldots,\Omega_{n-k}(g\,|\,\varphi)).
\end{align*}

\sloppy
By \proposition{reciprocal_function} we have $\Omega_j(g\,|\,\varphi)=\widehat{R}_j(\Omega_0(\tfrac{f}{\id}\,|\,\varphi),\ldots,\Omega_j(\tfrac{f}{\id}\,|\,\varphi))$. Now observe (as in the proof of \corollary{corollary_stirling}) that $f/\id$ has Taylor coefficients $D^{r+1}(f)(0)/(r+1)$, whence $\Omega_r(\tfrac{f}{\id}\,|\,\varphi)=\tfrac{1}{r+1}\Omega_{r+1}(f\,|\,\varphi)$ for $r\geq 0$. This completes the proof. 
\end{proof}
\fussy
\begin{cor}\label{family_A}
For all $n\geq k\geq 1$
\[
A_{n,k}=\binom{n-1}{k-1}\widehat{P}_{n-k,n}(\widehat{R}_0(\tfrac{X_1}{1}),\widehat{R}_1(\tfrac{X_1}{1},\tfrac{X_2}{2}),\ldots,\widehat{R}_{n-1}(\tfrac{X_1}{1},\ldots,\tfrac{X_n}{n})).
\]
\end{cor}
\begin{proof}
Take $f=\varphi$ and recall that $\Omega_j(\varphi\,|\,\varphi)=X_j$. 
\end{proof}
\begin{rem}\label{s1_numbers_new}
Unification on both sides of the remarkable identity above enables us to represent the Stirling numbers of the first kind by means of the potential polynomials. Define $\rho_s:=\widehat{R}_s(1,\frac{1}{2},\ldots,\frac{1}{s+1})\in\rationals$, $s\geq 0$; then $\rho_0=1$ and by \eq{unification_A}
\[
	s_1(n,k)=\binom{n-1}{k-1}P_{n-k,n}(\rho_1,\ldots,\rho_{n-1})\qquad(1\leq k\leq n).
\]
We note here (without proof) that $\rho_s=\sum_{j=0}^s(-1)^j\binom{s+j}{s}^{-1}\widetilde{s_2}(s+j,j)$ and $\rho_{2j+1}=0$ ($j\geq 1$), where $\widetilde{s_2}(s+j,j):=\widetilde{B}_{s+j,j}\comp 1$ (associated Stirling numbers of the second kind; see, e.\,g., \cite[p.\,76]{rior1958}, \cite[p.\,222]{comt1974}).
\end{rem}

We complete the picture by also establishing statements for $\Omega_n(f^k\,|\,\varphi)$ and $B_{n,k}$ that correspond to those of \theorem{inverse_power} and \corollary{family_A}, respectively.
\begin{thm}\label{power_of_function}
Let $f\in\invfuncs$ and $n,k\in\integers$ with $1\leq k\leq n$. Then we have
\begin{align*}
\text{\emph{(i)}~~~}\Omega_n(f^k\,|\,\varphi)&=k!B_{n,k}(\Omega_1(f\,|\,\varphi),\ldots,\Omega_{n-k+1}(f\,|\,\varphi))
,\\
\text{\emph{(ii)}~~~}\Omega_n(f^k\,|\,\varphi)&=k!\binom{n}{k}\widehat{P}_{n-k,k}(\tfrac{\Omega_1(f\,|\,\varphi)}{1},\tfrac{\Omega_2(f\,|\,\varphi)}{2},\ldots).
\end{align*}
\end{thm}
\begin{proof}
The proof runs in much the same way as it does for \theorem{inverse_power}. (i)~Apply \eq{omega_composition_0case} to $\Omega_n(\id^k\comp f\,|\,\varphi)$.\,---\,(ii)~Choose $g\in\funcs_1$  such that $f=\id\cdot g$; hence $f^k=\id^k\cdot g^k$ and obviously $\Omega_n(f\,|\,\varphi)=n\Omega_{n-1}(g\,|\,\varphi)$. Then, the Leibniz rule \eqref{omega_product} and a straightforward calculation give $\Omega_n(f^k\,|\,\varphi)=k!\binom{n}{k}\widehat{P}_{n-k,k}\comp\Omega_{\num}(g\,|\,\varphi)$. From this the assertion follows.
\end{proof}
\begin{cor}\label{family_B}
For all $n\geq k\geq 1$
\[
B_{n,k}=\binom{n}{k}\widehat{P}_{n-k,k}(\tfrac{X_1}{1},\ldots,\tfrac{X_{n-k+1}}{n-k+1}).
\]
\end{cor}
\begin{proof}
Take $f=\varphi$ and recall that $\Omega_j(\varphi\,|\,\varphi)=X_j$. 
\end{proof}
%
%
\section{Composition rules}
In this section we are going to investigate the effect the composition of functions has on polynomials, which depend in a specific way on those functions. The two main results (\theorem{first_CR} and \theorem{second_CR}) will prove to be  efficient tools for dealing with the polynomial families of interest here. 

Let $f,g$ be any functions such that $h=f\comp g$ is well-defined as a function. We write $F_n=\Phi_n(f)$, $G_n=\Phi_n(g)$, $H_n=\Phi_n(h)$ for the corresponding FdB polynomials according to \eq{fdbpoly_0case}. 
\begin{prop}\label{basic_CR}
	$\Phi_n(f\comp g)=\Phi_n(f)\comp\Phi_{\num}(g)$.
\end{prop}
\begin{proof}
Recall that $\varphi\in\funcs_0$ does not occur in $f$ or in $g$. In the 0-case ($f\in\funcs$, $g\in\funcs_0$) we obtain by \eq{omega_composition_0case}
\begin{align*}
	\Phi_n(f\comp g)&=\Omega_n(f\comp(g\comp\varphi)\,|\,\varphi)\\
	   &=\sum_{k=0}^n\Omega_k(f\,|\,\varphi)\cdot(B_{n,k}\comp\Omega_{\num}(g\comp\varphi\,|\,\varphi))\tag{*}\\
		 &=\bigg(\sum_{k=0}^n D^k(f)(0)B_{n,k}\bigg)\comp\Phi_{\num}(g)=\Phi_n(f)\comp\Phi_{\num}(g).
\intertext{%
In the 1-case ($f\in\laurentpoly$, $g\in\funcs_1$) we have to apply \eq{omega_composition_1case} so that $\Omega_k(f\,|\,\varphi)$ in line (*) becomes $D^k(f)(\Omega_0(g\comp\varphi\,|\,\varphi))=D^k(f)(\Phi_0(g))$, and hence
}
	\Phi_n(f\comp g)&=\bigg(\sum_{k=0}^n D^k(f)(X_0)B_{n,k}\bigg)\comp\Phi_{\num}(g)=\Phi_n(f)\comp\Phi_{\num}(g).\qedhere
\end{align*}
\end{proof}
\begin{rem}\label{fdbpoly_index0}
The statement of \proposition{basic_CR} may alternatively be written as $H_n=F_n(G_1,\ldots,G_n)$ (0-case) or as $H_n=F_n(G_0,G_1,\ldots,G_n)$ (1-case). Note that the 0-case includes $H_0=F_0=f(0)$, whereas in the 1-case we have $H_0=F_0(G_0)$ with $G_0=g(0)\neq 0$.
\end{rem}
\begin{rem}\label{basic_CR_polynomialcase}
Suppose $f\in\poly$ and $g\in\laurentpoly$. Then we have for every $n\geq 0$	$\widehat{\Phi}_n(f\comp g)=\widehat{\Phi}_n(f)\comp\widehat{\Phi}_{\num}(g)$ (by applying the same argument as in the 1-case in the proof of \proposition{basic_CR}).
\end{rem}

From \remark{basic_CR_polynomialcase} we obtain useful \emph{multiplication rules} for the potential polynomials and the partial Bell polynomials.
\begin{cor}[\remark{basic_CR_polynomialcase}]\label{multiplication_rule}
For all $n\geq 0$ and $r,s\in\integers$ we have
\begin{align*}
	\text{\emph{(i)}\quad}&\widehat{P}_{n,r s}=\widehat{P}_{n,r}\comp\widehat{P}_{\num,s}=\widehat{P}_{n,r}(\widehat{P}_{0,s},\ldots,\widehat{P}_{n,s}),\\
	\text{\emph{(ii)}\quad}&\widehat{P}_{n,-r}=\widehat{P}_{n,r}\comp\widehat{R}_{\num}=\widehat{P}_{n,r}(\widehat{R}_{0},\ldots,\widehat{R}_{n}),\\
	\text{\emph{(iii)}\quad}&\widehat{R}_n\comp\widehat{R}_{\num}=\widehat{R}_n(\widehat{R}_0,\ldots,\widehat{R}_{n})=X_n,\\
	\text{\emph{(iv)}\quad}&B_{n,r s}=\frac{r!(s!)^r}{(r s)!}B_{n,r}(B_{1,s},\ldots,B_{n-r+1,s})\quad(r,s\geq 0). 
\end{align*}
\end{cor}
\begin{proof}
(i) $\widehat{P}_{n,r s}=\widehat{\Phi}_n(\id^{r\cdot s})=\widehat{\Phi}_n(\id^{r}\comp\id^{s})=\widehat{\Phi}_n(\id^{r})\comp\widehat{\Phi}_{\num}(\id^{s})=\widehat{P}_{n,r}\comp\widehat{P}_{\num,s}$.\,---\,(ii)~In (i) set $s=-1$.\,---\,(iii)~Take $r=-1$ in (ii) and observe $\widehat{P}_{n,-1}=\widehat{R}_n$, $\widehat{P}_{n,1}=X_n$.\,---\,(iv)~Put $X_0=0$ in (i); then apply \eq{power_0case} and the homogeneity of the $B_{n,r}$.
\end{proof}

\begin{exm}\label{simple_fdb_inversion}
Let $f\in\invfuncs$ and $g=\inv{f}$. Then \proposition{basic_CR} immediately provides a quite simple infinite scheme of inverse relations. We have $F_n\comp G_{\num}=G_n\comp F_{\num}=\Phi_n(\id)=X_n$. For instance, $B_n\comp L_{\num}=L_n\comp B_{\num}=X_n$, where $B_n=\Phi_n(\expm)$ and $L_n=\Phi_n(\logm)$ (see \examples{examples_fdbpoly}). Chou, Hsu and Shiue \cite{chhs2006} discuss this and some more examples, which fit into this scheme.
\end{exm}

\sloppy
\begin{rem}\label{hopf_algebra}
For $n\geq 1$, let $\Phi_n[\invfuncs]$ denote the set of $\Phi_n(g)$, $g\in\invfuncs$. This set forms (together with $\comp$) a non-abelian group. According to \proposition{basic_CR} we have: $\Phi_n[\invfuncs]$ is closed under $\comp$, and its identity is $X_n$, since $G_n\comp X_{\num}=X_n\comp G_{\num}=G_n$. The inverse of $G_n=\Phi_n(g)$ is $\inv{G}_n:=\Phi_n(\inv{g})$, which by \remark{inverse_function}, \eqref{fdbpoly_0case} and \eqref{connection_A_B} can be written somewhat more explicitly as $\inv{G}_n=\sum_{k=1}^n A_{k,1}^{g}(0)B_{n,k}$. 

It should be noticed here that the mapping, which assigns to each \mbox{$g\in\invfuncs$} the infinite sequence $(\Phi_1(g),\Phi_2(g),\Phi_3(g),\ldots)$, is an isomorphism from the group $(\invfuncs,\comp)$ to the direct product $\prod_{n\geq 1}\Phi_n[\invfuncs]$ (in essence, the group of \emph{formal diffeomorphisms} leaving 0 fixed; it figures as a non-commutative Hopf algebra in \cite{brfk2006}).
\end{rem}
\fussy
\begin{lem}[Substitution Lemma]\label{substitution_lemma}
If $f\in\funcs$ and $g\in\funcs_0$, then
\begin{align*}
\text{\emph{(i)}\quad}&B_{n,k}\comp G_{\num}=B_{n,k}(G_1,\ldots,G_{n-k+1})=\sum_{j=k}^n B_{j,k}^{g}(0)B_{n,j},\\
\text{\emph{(ii)}\quad}&F_n\comp G_{\num}=F_n(G_1,\ldots,G_n)=\sum_{j=0}^n F_{j}^{g}(0)B_{n,j}.
\end{align*}
\end{lem}
\begin{proof}
(i). By \proposition{basic_CR} we have 
\begin{align*}
	B_{n,k}\comp G_{\num}&=\Phi_n(\expm_k)\comp\Phi_{\num}(g)=\Phi_n(\expm_k\comp g)\\
		&=\sum_{j=0}^n D^j(\expm_k\comp g)(0)B_{n,j}. 
\end{align*}
The FdB formula \eqref{fdb_formula} yields
\begin{align*}
	D^j(\expm_k\comp g)(0)&=\sum_{i=0}^j D^i(\expm_k)(g(0))B_{j,i}^g(0)\\
		&=\sum_{i=0}^j \kronecker{i}{k}B_{j,i}^g(0)=B_{j,k}^g(0).
\end{align*}
Since $B_{j,k}^g(0)=0$ for $j<k$, this proves (i).

(ii). By \eq{fdbpoly_0case} $F_n=\sum_{k=0}^n D^k(f)(0)B_{n,k}$, hence
{\allowdisplaybreaks
\begin{align*}
	F_n(G_1,\ldots,G_n)&=\sum_{k=0}^n D^k(f)(0)B_{n,k}(G_1,\ldots,G_{n-k+1})\\
	                   &\underset{\text{(i)}}{=}\sum_{k=0}^n D^k(f)(0)\sum_{j=k}^n B_{j,k}^{g}(0)B_{n,j},\\
										 &=\sum_{j=0}^n\bigg(\sum_{k=0}^{j}D^k(f)(0)B_{j,k}^{g}(0)\bigg)B_{n,j}.
\end{align*}
} 
\enlargethispage{1.5ex}

\vspace*{-1ex}
\noindent
The inner sum is equal to $F_j^g(0)$.
\end{proof}

\begin{cor}\label{bell_substitutions_1}
\begin{align*}
	\text{\emph{(i)}\quad}B_{n,k}(B_1,\ldots,B_{n-k+1})=\sum_{j=k}^n s_2(j,k)B_{n,j},\\
	\text{\emph{(ii)}\quad}B_{n,k}(L_1,\ldots,L_{n-k+1})=\sum_{j=k}^n s_1(j,k)B_{n,j}.
\end{align*}
\end{cor}
\begin{proof} 
(i). Put $g=\expm$ in part (i) of \lemma{substitution_lemma}; it follows $B_{j,k}^{\expm}(0)=B_{j,k}\comp 1=s_2(j,k)$. --- (ii). Substitute $\logm$ for $g$; hence by \eq{connection_A_B} $B_{j,k}^{\logm}(0)=A_{j,k}^{\expm}(0)=A_{j,k}\comp 1=s_1(j,k)$. 
\end{proof}

\begin{rem}
Let $b(n)$ denote the $n$th \emph{Bell number} (total number of partitions of an $n$-set). We then have $b(n)=B_n\comp 1$ and by part (i) of \corollary{bell_substitutions_1} the following identity established by Yang \cite[Equation (31)]{yang2008}: 
\begin{center}
$B_{n,k}(b(1),\ldots,b(n-k+1))=\sum_{j=k}^n s_2(n,j)s_2(j,k)$.
\end{center}
\end{rem}

The right-hand side of (i) in \corollary{bell_substitutions_1} can assume yet another form by substituting for $s_2(j,k)$ the explicit expression from part (ii) in \remark{power_unification}. By the homogeneity of the Bell polynomials one easily obtains after a short calculation
\begin{equation}\label{bell_substitutions_2}
	B_{n,k}(B_1,\ldots,B_{n-k+1})=\frac{1}{k!}\sum_{j=1}^k (-1)^{k-j}\binom{k}{j}B_n(j X_1,\ldots,j X_n).
\end{equation}
\begin{prop}\label{product_identity}
	$P_{n,k}(B_1,\ldots,B_n)=B_n(k X_1,\ldots,k X_n)$.
\end{prop}

\begin{proof}
Again using part (i) of \corollary{bell_substitutions_1} yields
\begin{align*}
	P_{n,k}(B_1,\ldots,B_n)&=\sum_{j=0}^n \powerfall{k}{j}B_{n,j}(B_1,\ldots,B_n)=\sum_{j=0}^n \powerfall{k}{j}\sum_{r=j}^n s_2(r,j)B_{n,r}\\
	&=\sum_{r=0}^n\bigg(\sum_{j=0}^r \powerfall{k}{j} s_2(r,j)\bigg)B_{n,r}=\sum_{r=0}^n k^r B_{n,r}\\
	&=\sum_{r=0}^n B_{n,r}(k X_1,k X_2,\ldots)=B_n(k X_1,\ldots,k X_n).\hspace*{2em}\qedhere
\end{align*}
\end{proof}

Let us return to \proposition{basic_CR}, since we are now in a position to prove that also the converse statement holds. We give it a slightly different form.

\begin{thm}[First Composition Rule]\label{first_CR}
	Let $f$ and $g$ be any functions such that $f\comp g\in\funcs$. Then, for all $h\in\funcs:$\quad $h=f\comp g$ $\iff$ $H_n=F_n\comp G_{\num}$.
\end{thm}

\begin{proof}
`$\Rightarrow$': By \proposition{basic_CR}. --- `$\Leftarrow$': First assume the 0-case: $f,h\in\funcs$ and $g\in\funcs_0$. Let $n$ be any non-negative integer and suppose $H_n=F_n\comp G_{\num}$. Part~(ii) of \lemma{substitution_lemma} then yields
\begin{equation*}
	\sum_{k=0}^n D^k(h)(0)B_{n,k} = \sum_{k=0}^n F_k^{g}(0)B_{n,k}.
\end{equation*}
If $n=0$, then $h(0)=f(0)=(f\comp g)(0)$ according to \remark{fdbpoly_index0}. For $n>0$ we obtain from \proposition{msp_2} that the sequence $B_{n,1},B_{n,2},\ldots,B_{n,n}$ is linearly independent in $\const[X_1,\ldots,X_n]$. Hence for every $k$, $1\leq k\leq n$, by the FdB formula \eqref{fdb_formula}
\begin{equation}\label{equate_taylorcoeffs}
	D^k(h)(0)=F_k^{g}(0)=\sum_{j=0}^k D^j(f)(0)B_{k,j}^{g}(0)=D^k(f\comp g)(0).
\end{equation}
This shows that the Taylor coefficients of $h$ agree with those of $f\comp g$. --- The same reasoning works for the 1-case ($f\in\laurentpoly$, $g\in\funcs_1$, $h\in\funcs$), however with $D^k(h)(X_0)=D^k(f\comp g)(X_0)$ for $k=0,1,2,\ldots$ instead of \eqref{equate_taylorcoeffs} at the end. Here already  $k=0$ yields the desired result. 
\end{proof}

Our second main result concerns the mapping $P\mapsto P^{\varphi}$ in its effect on the Stirling polynomials.
\begin{thm}[Second Composition Rule]\label{second_CR}
{\allowdisplaybreaks
\begin{align*}
	\text{\emph{(i)}\quad}B_{n,k}^{f\,\comp\,g}(0)&=\sum_{j=k}^n B_{n,j}^{g}(0)B_{j,k}^{f}(0)\qquad(f,g\in\funcs_0),\\
	\text{\emph{(ii)}\quad}A_{n,k}^{f\,\comp\,g}(0)&=\sum_{j=k}^n A_{n,j}^{f}(0)A_{j,k}^{g}(0)\qquad(f,g\in\invfuncs).
\end{align*}
}
\end{thm}
\begin{rem}
Part (i) of the theorem is essentially due to Jabotinsky \cite{jabo1947,jabo1953}. Note that the right-hand side of (i) can be interpreted as the contravariant product of two instances of the matrix $(B_{n,k})$. Comtet used this idea in order to generalize Fa\`{a} di Bruno's formula so as to apply to fractionary iterates of formal series \cite[p.\,144]{comt1974}. 
\end{rem}

\begin{proof}
(i). The result can be obtained by direct computation:
\begin{align*}
	B_{n,k}^{f\,\comp\,g}(0)&=B_{n,k}(D^1(f\comp g)(0),\ldots,D^{n-k+1}(f\comp g)(0))\\
		&=B_{n,k}(\Phi_1(f)^{g}(0),\ldots,\Phi_{n-k+1}(f)^{g}(0))&\hspace*{-2.2em}\text{(\eq{fdb_formula}  and \eqref{fdbpoly_0case})}\\
		&=B_{n,k}(\Phi_1(f),\ldots,\Phi_{n-k+1}(f))^{g}(0).&\text{(\eq{substitution_rule})}
\end{align*} 
Applying \lemma{substitution_lemma}\,(i) to $B_{n,k}(\Phi_1(f),\ldots,\Phi_{n-k+1}(f))$ yields the assertion.

(ii). We use part (i) and \eq{connection_A_B}. Note the covariant behavior of $A_{n,k}$.
\begin{align*}
	A_{n,k}^{f\,\comp\,g}(0)&=(B_{n,k}^{\inv{f\,\comp\,g}}\comp(f\comp g))(0)=B_{n,k}^{\inv{g}\,\comp\,\inv{f}}(f(g(0)))\\
	&=B_{n,k}^{\inv{g}\,\comp\,\inv{f}}(0)=\sum_{j=k}^n B_{n,j}^{\inv{f}}(0)B_{j,k}^{\inv{g}}(0)\\
	&=\sum_{j=k}^n A_{n,j}^{f}(\inv{f}(0))A_{j,k}^{g}(\inv{g}(0)).
\end{align*} 
Since $\inv{f}(0)=\inv{g}(0)=0$, the result follows. 
\end{proof}

\begin{cor}
	If $H_n=\Phi_n(h)$, $h\in\funcs$, then
\[
	H_n^{f\comp g}(0)=\sum_{k=0}^n B_{n,k}^{g}(0)H_k^{f}(0).
\]	
\end{cor}

\begin{proof}
By a short calculation using \eq{fdbpoly_0case} and \theorem{second_CR}\,(i).  
\end{proof}

The following statement was originally published in \cite[Theorem~5.1]{schr2015} under the title `Inversion Law' and was proven there by induction. The new proof presented below is much more natural.

\begin{cor}\label{orthocompanions_A_B}
	$(A_{n,k})$ and $(B_{n,k})$ are orthogonal companions of each other: $A_{n,k}=\ortho{B}_{n,k}$ and $B_{n,k}=\ortho{A}_{n,k}$.
\end{cor}
\begin{proof}
	Let $\varphi$ be any function from $\invfuncs$. Then 
\[	
	B_{n,k}^{\varphi\,\comp\,\inv{\varphi}}(0)=B_{n,k}^{\id}(0)=B_{n,k}(1,0,\ldots,0)=\kronecker{n}{k}.\tag{*} 
\]	
	On the other hand, by part (i) of \theorem{second_CR} and by \eq{substitution_rule} we have
\[
	B_{n,k}^{\varphi\,\comp\,\inv{\varphi}}(0)=\sum_{j=k}^n A_{n,j}^{\varphi}(\inv{\varphi}(0))B_{j,k}^{\varphi}(0)=\bigg(\sum_{j=k}^n A_{n,j}B_{j,k}\bigg)^{\varphi}(0).\tag{**}
\]	
Now equate the right-hand sides of (*) and (**). Then, finally applying the argument from \remark{identical_polynomials} shows that the orthogonality relation \eq{orthogonality_relation} is satisfied by $A_{n,k}$ and $B_{n,k}$.  
\end{proof}
\label{Dual companions}
\section{Representation by Bell polynomials}
As we have seen in Section~4, composing FdB polynomials again yields FdB polynomials. In this section we will go beyond by dealing with polynomial families (not necessarily FdB) that can be represented as instances of the partial Bell polynomials. A well-known example is the cycle indicator. In addition, several new polynomial families will be introduced and examined in more detail, including `forest polynomials' (generated by tree polynomials) as well as multivariate Lah polynomials, which form a self-orthogonal fami\-ly. Since a (regular) family represented this way always has an orthogonal companion, we can establish a corresponding inverse relation in each case. 

\subsection{B-representability}
A polynomial family $(Q_{n,k})$ is said to be \emph{B-representable}, if there is an infinite sequence of polynomials $H_1,H_2,H_3,\ldots$ such that for all integers $n,k$ with $1\leq k\leq n$
\begin{equation}\label{def_B_representable}
	Q_{n,k}=B_{n,k}\comp H_{\num}=B_{n,k}(H_1,\ldots,H_{n-k+1}).
\end{equation}
The Bell polynomials itself are, of course, B-representable (since $B_{n,k}=B_{n,k}\comp B_{\num,1}$ with $B_{n,1}=X_n$). The same holds for the associate Bell polynomials $\widetilde{B}_{n,k}$ \cite[Corollary 4.5]{schr2015} with $\widetilde{B}_{j,1}=X_j$ for $j\geq 2$, but $\widetilde{B}_{1,1}=0$ (implying that this family is not regular). Non-trivial examples of B-representable families provide $(A_{n,k})$ (see \eq{A_B_representable}), or the Stirling numbers $s_2(n,k)$ (see \eq{unification_B}).\medskip

First, we gather some basic properties.

\begin{prop}\label{B_rep_properties}
	Let $(Q_{n,k})$ be any B-representable family of polynomials. Then we have
	\vspace*{-1.2ex}
	\begin{align*}
		\text{\emph{(i)}}\quad&(Q_{n,k})\text{~is lower triangular}.\\
		\text{\emph{(ii)}}\quad&Q_{n,k}=B_{n,k}(Q_{1,1},\ldots,Q_{n-k+1,1})\quad(1\leq k\leq n)\\
		\text{\emph{(iii)}}\quad&Q_{n,n}=(Q_{1,1})^n\\
		\text{\emph{(iv)}}\quad&\text{If~}(Q_{n,k})\text{~is regular, then~}\ortho{Q}_{n,k}=A_{n,k}(Q_{1,1},\ldots,Q_{n-k+1,1})\text{~exists}\\
		&\text{and is B-representable}.
	\end{align*}
\end{prop}

\begin{proof}
(i) Clear by definition.\,---\,(ii) Taking $k=1$ in \eq{def_B_representable} gives $Q_{n,1}=B_{n,1}(H_1,\ldots,H_n)=H_n$.\,---\,(iii) A special case of (ii) is $Q_{1,1}=H_1$; hence, putting $k=n$ in \eq{def_B_representable} we get $Q_{n,n}=B_{n,n}(H_1)=(H_1)^n=(Q_{1,1})^n$.\,--- (iv)~It follows from regularity and triangularity $Q_{j,j}\neq 0$ for all $j$ (diagonal entries), in particular $Q_{1,1}\neq 0$. Therefore, $A_{n,k}(Q_{1,1},\ldots,Q_{n-k+1,1})$ is well-defined (see \eq{stirling_partitionsum}) and by \corollary{orthocompanions_A_B} equal to $\ortho{Q}_{n,k}$. From this we finally obtain by \eq{A_B_representable}
%
\begin{align*}
B_{n,k}(\ortho{Q}_{1,1},&\ortho{Q}_{2,1},\ldots,\ortho{Q}_{n-k+1,1})\\
	&=B_{n,k}(A_{1,1}(Q_{1,1}),A_{2,1}(Q_{1,1},Q_{2,1},\ldots),\ldots)\\
	&=A_{n,k}(Q_{1,1},Q_{2,1},\ldots,Q_{n-k+1,1})=\ortho{Q}_{n,k}.\qedhere
\end{align*}
\end{proof}

\noindent
Part (iii) of \proposition{B_rep_properties} may serve as a \emph{necessary condition for B-re\-pre\-sen\-ta\-bility}. Since, for example, $P_{1,1}=X_1$ and $P_{2,2}=2X_1^2+2X_2\neq (P_{1,1})^2$, the potential polynomials are not B-representable.
\label{Polynomials not B-representable}

\sloppy
\begin{rem}\label{Bell_equations}
As a precaution, it should be emphasized that criterion (iii) is not sufficient. The crucial point here is that \eq{def_B_representable} comprises an \emph{infinite} number of equations and unknowns $H_1,H_2,H_3,\ldots$. On the other hand: Given any fixed $n\geq 1$, the system consisting of the first $n$ equations from \eqref{def_B_representable} is indeed solvable, if and only if there exists $H_1$ such that $(H_1)^n=Q_{n,n}$. 

To see this, recall that $H_n=Q_{n,1}$ and observe that the remaining $H_2,\ldots,H_{n-1}$ appear only linearly, since \eq{Bells_identity} yields 
\[
\frac{\partial{B_{n,k}}}{\partial{X_{n-k+1}}}(H_1,\ldots,H_{n-k+1})=\binom{n}{k-1}H_1^{k-1} 
\]
for every $k$ with $2\leq k\leq n-1$.

As an example, put $n=3$ and consider the equation system 
\[
	B_{3,k}(H_1,\ldots,H_{4-k})=Q_{3,k}:=1, \quad 1\leq k\leq 3.
\]
If we choose $H_1=1$, then necessarily $H_2=\nicefrac{1}{3}$ and $H_3=1$. But this solution cannot be extended to $n=4$, for choosing in this case $H_1=1$ implies $H_2=\nicefrac{1}{6}$, $H_3=-\nicefrac{11}{48}$, and $H_4=1$ so that there is no infinite sequence $H_1,H_2,H_3,\ldots$ satisfying $B_{n,k}(H_1,\ldots,H_{n-k+1})=1$ \emph{for all} $n,k$ with $1\leq k\leq n$.
\end{rem}

\fussy
We will now characterize the B-representable polynomial families. As to FdB polynomials, the following can be easily inferred from the First Composition Rule (FCR).

\begin{prop}\label{FdB_representable}
	Let $(Q_{n,k})$ be any family of FdB polynomials. Then we have: $(Q_{n,k})$ is B-representable $\iff$ $Q_{n,k}=\Phi_n(\frac{h^k}{k!})$ for some $h\in\funcs$. 
\end{prop}

\begin{proof}
`$\Rightarrow$': Let $f_k$ be functions such that $\Phi_n(f_k)=Q_{n,k}=B_{n,k}\comp Q_{\num,1}$. According to the FCR (\theorem{first_CR}, `$\Leftarrow$') the equation $f_k=\expm_k\comp f_1=f_1^k/k!$ holds, where $f_1\in\funcs$.\,---\,`$\Leftarrow$': Immediately, by the FCR (`$\Rightarrow$') $Q_{n,k}=\Phi_n(\frac{h^k}{k!})=\Phi_n(\expm_k\comp h)=\Phi_n(\expm_k)\comp\Phi_{\num}(h)=B_{n,k}\comp Q_{\num,1}$. 
\end{proof}

In the case of an arbitrary polynomial family $(Q_{n,k})$, the statement to follow provides  a necessary and sufficient condition for B-representability.

\begin{prop}\label{general_B_representable}
	$(Q_{n,k})$ is B-representable, if and only if
\[	
	Q_{n,k}=\sum_{j=1}^{n-k+1}\binom{n-1}{j-1}Q_{j,1}Q_{n-j,k-1}.\quad\tag{*}
\]
\end{prop}

\begin{proof}Necessity (`only if'): Immediate by the fact that $Q_{n,k}:=B_{n,k}$ satisfies (*); see, e.\,g., \cite{knut1992a} and \cite[Proposition 5.5, Remark 5.6]{schr2015}.

\noindent
Sufficiency (by induction): Suppose (*); then, observing 
\[
	Q_{j,1}=B_{j,1}(Q_{1,1},\ldots,Q_{j,1})
\]
and applying the induction hypothesis 
\[
	Q_{n-j,k-1}=B_{n-j,k-1}(Q_{1,1},\ldots,Q_{n-k-j+2,1})\quad (1\leq j\leq n-k+1)
\]
one gets
\begin{align*}
	Q_{n,k}&=\mspace{-4mu}\sum_{j=1}^{n-k+1}\mspace{-2mu}\binom{n-1}{j-1}B_{j,1}(Q_{1,1},Q_{2,1}\ldots)B_{n-j,k-1}(Q_{1,1},Q_{2,1},\ldots)\\
	       &=B_{n,k}(Q_{1,1},\ldots,Q_{n-k+1,1})\quad\text{(since $B_{n,k}$ satisfies (*)).}\qedhere
\end{align*}
\end{proof}

\subsection{Generalized Stirling inversion}
The ordinary Stirling inversion \cite[Proposition~1.4.1b]{stan1986} is based on the well-known orthogonality relation satisfied by the Stirling numbers of the first and second kind: $\sum_{j=k}^n s_1(n,j)s_2(j,k)=\kronecker{n}{k}$ for all $n\geq 0$ and $0\leq k\leq n$. This type of inversion can be generalized considerably by taking advantage of the fact that according to \proposition{B_rep_properties}\,(iv), every regular B-representable family of polynomials $Q_{n,k}$ has an orthogonal companion $\ortho{Q}_{n,k}$, which is also B-representable.

\sloppy
\begin{prop}[Generalized Stirling inversion]\label{generalized_stirling_inversion}
Let $U_0,U_1,U_2,\ldots$ and $V_0,V_1,V_2,\ldots$ be two sequences (of polynomials from an arbitrary overring of $\const[X_1^{-1},X_1,X_2,\ldots]$, say) and let $(Q_{n,k})$ be any regular B-representable family of polynomials. Then the following statements are equivalent:
\begin{align*}
	\text{\emph{(i)}} \quad& U_n=\sum_{k=0}^n Q_{n,k}V_k\quad\text{for all~}n\geq 0,\\
	\text{\emph{(ii)}}\quad& V_n=\sum_{k=0}^n \ortho{Q}_{n,k}U_k\quad\text{for all~}n\geq 0.
\end{align*}
\end{prop}
\fussy
\begin{proof}
By \corollary{orthocompanions_A_B} $\ortho{B}_{n,k}=A_{n,k}$; hence (i) and (ii) are equivalent when $Q_{n,k}$ is replaced by $B_{n,k}$ (and consequently $\ortho{Q}_{n,k}$ by $A_{n,k}$); see \cite[Corollary 5.2]{schr2015}. On the other hand, we have by assumption $Q_{n,k}=B_{n,k}\comp Q_{\num,1}$, and thus by part~(iv) of \proposition{B_rep_properties} also $\ortho{Q}_{n,k}=A_{n,k}\comp Q_{\num,1}$. From this follows the assertion.
\end{proof}

\begin{rem}\label{orthobased_inversion}
The above generalizations by no means exhaust all possibilities of basing inverse relationships on orthogonality. In principle, any invertible infinite matrix $(Q_{n,k})$ can be used for this purpose (for convenience, of lower triangular form). For example, Riordan \cite[p.\,100]{rior1968} compiled a large stock of inverse relations by establishing identities involving the Taylor coefficients of a function $f\in\funcs_1$ and its reciprocal $f^{-1}$. With the concepts developed in Section 3, it is easy to convince oneself that all these relations fall under the orthogonality scheme
\[
	\sum_{j=k}^n\frac{X_{n-j}}{(n-j)!}\frac{\widehat{R}_{j-k}}{(j-k)!}=\kronecker{n}{k}.
\]
\sloppy
For the general case, Milne and Bhatnagar \cite{mibh1998} have found recurrences, which characterize the entries $Q_{n,k}$ of an orthogonality relation. Huang \cite{huan2002}, by  representing functions relative to strictly monotone Schauder bases in $\funcs$, has shown that the inverse relationship (with the orthogonality property) comes about through the interchange of bases.
\end{rem}

\fussy
While the matrix entries $Q_{n,k}$ in the orthogonality relation generally only determine each other implicitly, in the case of B-representable polynomials we immediately have the explicit expression $\ortho{Q}_{n,k}=A_{n,k}\comp Q_{\num,1}$. We can go a step further here and resolve $A_{n,k}$ into lower-order expressions, for example according to \corollary{family_A}, to \theorem{mainresult}, or to the general polynomial version of the famous Schl\"omilch formula for the Stirling numbers $s_1(n,k)$ in terms of $s_2(n,k)$. This generalization has been presumably for the first time established and proven in \cite[Theorem 6.4]{schr2015}. After some few index shifts and applying elementary properties of the binomial coefficients, the latter result can be easily rewritten in the form of an identity of the following Schl\"omilch-Schl\"afli type:
\begin{equation}\label{schloemilch_schlaefli_A_B}
	A_{n,n-k}=(-1)^k\sum_{j=0}^k\binom{k+n}{k-j}\binom{k-n}{k+j}X_1^{-(n+j)}B_{k+j,j}.
\end{equation}
By finally applying \proposition{B_rep_properties} to this, we are easily led to the corresponding generalizations regarding B-representable polynomials.

\sloppy
\begin{thm}[Generalized Schl\"omilch-Schl\"afli identities]\label{schloemilch_schlaefli_general}
For every regular B-representable family of polynomials $Q_{n,k}$ the following holds for all \mbox{$n\geq k\geq 0$}:
\begin{align*}
	\text{\emph{(i)}} \quad \ortho{Q}_{n,n-k}&=(-1)^k\sum_{j=0}^k\binom{k+n}{k-j}\binom{k-n}{k+j}(\ortho{Q}_{1,1})^{n+j}Q_{k+j,j,}\\
	\text{\emph{(ii)}}\quad Q_{n,n-k}&=(-1)^k\sum_{j=0}^k\binom{k+n}{k-j}\binom{k-n}{k+j}(Q_{1,1})^{n+j}\ortho{Q}_{k+j,j}.
\end{align*}
\end{thm}

\fussy
\begin{rem}
Note that $\ortho{Q}_{1,1}=A_{1,1}(Q_{1,1})=Q_{1,1}^{-1}$.
\end{rem}

\begin{rem}
Unification on both sides of \eqref{schloemilch_schlaefli_A_B} immediately yields Schl\"afli's formula for $s_1(n,n-k)$ in terms of Stirling numbers of the second kind (see Quaintance and Gould \cite[Eq. (13.31)]{qugo2016}). Conversely, if we take $B_{n,n-k}$ for $Q_{n,n-k}$ and apply unification to part (ii) of \theorem{schloemilch_schlaefli_general}, we get Gould's formula for $s_2(n,n-k)$ in terms of Stirling numbers of the first kind (see \cite{goul1960} and \cite[Eq. (13.42)]{qugo2016}).
\end{rem}

\subsection{Cycle indicator polynomials}
Let $\ptsall$ be the union of all sets of partition types. We define the mapping $\zeta:\ptsall\longrightarrow\naturals$ by 
\begin{equation}\label{cycle_function}
	\zeta(r_1,r_2,r_3,\ldots):=\frac{(r_1+2r_2+3r_3+\cdots)!}{r_1!r_2!r_3!\cdots 1^{r_1}2^{r_2}3^{r_3}\cdots}.
\end{equation}

The right-hand side of \eqref{cycle_function} is Cauchy's famous expression that counts the permutations having exactly $r_j$ cycles of size $j$ ($j=1,2,3,\ldots$). The corresponding partition polynomial
\[
	Z_{n,k}:=\sum_{\ptsind{n}{k}}\zeta(r_1,r_2,r_3,\ldots)X_1^{r_1}X_2^{r_2}X_3^{r_3}\cdots
\]
could rightly be called \emph{partial cycle indicator}, inasmuch the term \emph{cycle indicator} \cite[p.\,68]{rior1958} (sometimes also \emph{augmented cycle index} \cite[p.\,19]{stan1999}) is reserved for $Z_{n}:=Z_{n,1}+Z_{n,2}+\cdots+Z_{n,n}$.
\label{Table: Partial cycle indicator polynomials}

\sloppy
It is easy to check that $(Z_{n,k})$ is B-representable. We have $Z_{n,n}=X_1^n=Z_{1,1}^n$ (the necessary condition (iii) from \proposition{B_rep_properties}) and \mbox{$Z_{n,1}=(n-1)!X_n$}. Thus, a few lines of a direct calculation (cf. \cite[p.\,247]{comt1974}) result in 
\begin{equation}\label{B_rep_cycleindicator}
	B_{n,k}\comp Z_{\num,1}=B_{n,k}(0!X_1,1!X_2,2!X_3,\ldots)=Z_{n,k}.
\end{equation}
By part (iv) of \proposition{B_rep_properties} we obtain as orthogonal companion $\ortho{Z}_{n,k}=A_{n,k}(0!X_1,1!X_2,2!X_3,\ldots)$. As is well known, unification of $\eqref{B_rep_cycleindicator}$ gives the signless Stirling numbers of the first kind $Z_{n,k}\comp 1=c(n,k)=B_{n,k}(0!,1!,2!,\ldots)$ \cite[p.\,135]{comt1974}. Hence $\ortho{Z}_{n,k}\comp 1=A_{n,k}(0!,1!,2!,\ldots)=(-1)^{n-k}s_2(n,k)$, which might be called \emph{signed} Stirling numbers of the se\-cond kind.

\fussy
\begin{rem}
From \eq{B_rep_cycleindicator} it is only a small step to the so-called \emph{exponential formula} that is based on the idea of interpreting the coefficients of $e^{f(x)}$  combinatorially. Here $f$ is assumed to be any function in $\funcs_0$ and $t_n:=D^n(f)(0)/(n-1)!$ for $n\geq 1$. Then, the $n$th Taylor coefficient of $e^{f(x)}$ is $B^{f}_n(0)=B_n(0!t_1,\ldots,(n-1)!t_n)=Z_n(t_1,\ldots,t_n)$, and we immediately obtain
\[
	\exp\bigg(\sum_{n\geq1}t_n\frac{x^n}{n}\bigg)=\sum_{n\geq 0}Z_n(t_1,\ldots,t_n)\frac{x^n}{n!}.
\]
\sloppy
For a detailed treatment, various combinatorial applications and historical notes on this topic, the reader is referred to Stanley \cite[Section 5.1]{stan1999}.
\end{rem}

\fussy
\subsection{Idempotency polynomials. Forest polynomials}
While the potential polynomials itself are not B-representable, this is actually yet the case with certain closely related families to be investigated in the sequel. 

\noindent
A simple example of this kind follows directly from \corollary{family_B}:
\begin{equation}\label{idempot_poly}
	B_{n,k}(X_0,2X_1,3X_2,\ldots) = \binom{n}{k}\widehat{P}_{n-k,k}.
\end{equation}

\vspace*{-1ex}
\noindent
This identity was established by Comtet in a slightly modified form (with $X_0=1$) \cite[Suppl. no.\,4, p.\,156/7]{comt1974}. Observing part (i) of \remark{power_unification} we immediately obtain by unification 
\begin{equation}\label{idempot_numbers}
	B_{n,k}(1,2,3,\ldots)=\binom{n}{k}k^{n-k},
\end{equation}
which equals the number of idempotent maps from an $n$-set into itself having exactly $k$ cycles (see, e.\,g., \cite{hasc1967}). Comtet's argument is based on considering $B_{n,k}^\psi(0)$ with $\psi(x):=x e^x$ \cite[p.\,135]{comt1974}.

Given this combinatorial meaning, it seems justified calling the expressions on the right-hand side of \eq{idempot_poly} \emph{idempotency polynomials}. The following theorem exhibits them (slightly modified) in the role of an orthogonal companion.

\begin{thm}\label{forests_from_trees}
For all $n,k\in\integers$ with $1\leq k\leq n$ we have
\begin{align*}
	\text{\emph{(i)}\qquad}B_{n,k}(\widehat{T}_1,\ldots,\widehat{T}_{n-k+1})&=\binom{n-1}{k-1}\widehat{P}_{n-k,n},\\
	\text{\emph{(ii)}\qquad}A_{n,k}(\widehat{T}_1,\ldots,\widehat{T}_{n-k+1})&=\binom{n}{k}\widehat{P}_{n-k,-k}.
\end{align*}
\end{thm}

Ahead of the proof three remarks are in order.

\begin{rem}\label{khelifa_identity}
Unification on part (i) of \theorem{forests_from_trees} immediately yields
\[
	B_{n,k}(1^0,2^1,3^2,\ldots)=\binom{n-1}{k-1}n^{n-k}. 
\]
This numerical identity has been formulated and given a lengthy proof by Khelifa and Cherruault \cite{khch2000}. Abbas and Bouroubi \cite[Theorems 3 and 6]{abbo2005} provided a significantly shorter argument together with an extension to binomial sequences.
\end{rem}

\begin{rem}\label{trees_and_forests}
One might say, somewhat jokingly, that the Bell polynomials return forests on receiving trees. Recall that there are $\widehat{T}_n\comp 1=n^{n-1}$ rooted (labeled) trees on $n$ vertices. On the other hand, $\binom{n-1}{k-1}n^{n-k}$ counts the planted forests with $k$ components on $n$ vertices (see, e.\,g., \cite[Proposition 5.3.2]{stan1999}). Motivated by this, we shall refer to the expressions on the right-hand side of (i) as \emph{forest polynomials}, denoted by $W_{n,k}$. Of course we have $W_{j,1}=\widehat{T}_j$.
\end{rem}
\label{Table: Forest polynomials}

\begin{rem}
The forest polynomials form a regular family; so, according to \proposition{B_rep_properties}\,(iv) its orthogonal companion $\ortho{W}_{n,k}$ is the polynomials, which appear in \theorem{forests_from_trees}\,(ii). It differs from the idempotency polynomials only in that the second index is negated. By unification, the orthogonal relationship is  immediately passed on to the corresponding number sequences (the orthogonality of which has been observed by Wang and Wang \cite{wtwa2011}).
\end{rem}

\begin{proof}[\theorem{forests_from_trees}]
The proof can be carried out exclusively using polynomial identities. The key idea is here to express the tree polynomials as 
\[
	\widehat{T}_n=A_{n,1}(\widehat{R}_0,2\widehat{R}_1,3\widehat{R}_2,\ldots),\tag{*}
\]
which follows from \corollary{corollary_stirling}\,(ii) and \corollary{basic_CR_polynomialcase}\,(iii). We then have
\begin{align*}
	B_{n,k}(\widehat{T}_1,\widehat{T}_2,\widehat{T}_3,\ldots)&=B_{n,k}(A_{1,1}(\widehat{R}_0),A_{2,1}(\widehat{R}_0,2\widehat{R}_1),\ldots))\hspace*{-5em}&\text{(by (*))}\\
	&=A_{n,k}(\widehat{R}_0,2\widehat{R}_1,3\widehat{R}_2,\ldots)&\text{(by \eq{A_B_representable})}\\	
	&=\binom{n-1}{k-1}\widehat{P}_{n-k,n},&\hspace*{-1em}\text{(by Cor.\,\ref{family_A}, Cor.\,\ref{basic_CR_polynomialcase}\,(iii))}
\end{align*}
which proves (i). In the same manner, (ii) can be shown by applying (*), \eq{B_A_representable}, \eq{idempot_poly}, and \corollary{basic_CR_polynomialcase}\,(ii):
\[
	A_{n,k}(\widehat{T}_1,\widehat{T}_2,\widehat{T}_3,\ldots)=B_{n,k}(\widehat{R}_0,2\widehat{R}_1,3\widehat{R}_2,\ldots)=\binom{n}{k}\widehat{P}_{n-k,-k}.\qedhere
\]
\end{proof}

\begin{prop}\label{cor_forests_from_trees}
	$A_{n,k}(X_0,2X_1,3X_2,\ldots)=\binom{n-1}{k-1}\widehat{P}_{n-k,-n}$.
\end{prop}

\begin{proof}
To evaluate the left-hand-side, use \corollary{family_A}, and then apply \corollary{multiplication_rule}\,(ii).
\end{proof}

\subsection{Lah polynomials. Involution}
\sloppy
Analogous to the considerations in Section~5.3 we define the mapping \mbox{$\omega:\ptsall\longrightarrow\naturals$} by 
\begin{equation}\label{order}
	\omega(r_1,r_2,r_3,\ldots):=\frac{(r_1+2r_2+3r_3+\cdots)!}{r_1!r_2!r_3!\cdots}.
\end{equation}

\fussy
The right-hand side of \eq{order} counts the number of ways a set of $n=r_1+2r_2+3r_3+\cdots$ objects can be partitioned into linearly ordered subsets, $r_j$ denoting the number of subsets with $j$ elements $(j=1,2,3,\ldots)$. The corresponding partition polynomials
\[
	L^{+}_{n,k}:=\sum_{\ptsind{n}{k}}\omega(r_1,r_2,r_3,\ldots)X_1^{r_1}X_2^{r_2}X_3^{r_3}\cdots
\]
will be called \emph{unsigned Lah polynomials}. A simple computation shows that they form a B-representable family. We have $L^{+}_{n,1}=n!X_n$ and
\begin{equation}\label{B_rep_unsignedlahpoly}
	B_{n,k}\comp L^{+}_{\num,1}=B_{n,k}(1!X_1,2!X_2,3!X_3,\ldots)=L^{+}_{n,k}.
\end{equation}
From this it can easily be derived that $L^{+}_{n,k}\comp 1$ are the unsigned Lah numbers $l^{+}(n,k):=\frac{n!}{k!}\binom{n-1}{k-1}$ \cite[p.\,135]{comt1974}. 

Let us now consider the signed Lah numbers $l(n,k):=(-1)^n l^{+}(n,k)$, which are known to be self-orthogonal in the sense that $\sum_{j=k}^n l(n,j)l(j,k)=\kronecker{n}{k}$; see, e.\,g., \cite[p.\,44]{rior1958} and \cite[Examples 5.2\,(ii)]{schr2015}. It is therefore natural to look for self-orthogonal polynomials $L_{n,k}$ such that the signed Lah numbers can be obtained by unification: $L_{n,k}\comp 1=l(n,k)$.

At first glance, $B_{n,k}(-1!X_1,2!X_2,-3!X_3,\ldots)=(-1)^n L^{+}_{n,k}$ might be supposed to be a suitable candidate; but this fails because these polynomials are not orthogonal companions of their own. It is however a promising (and eventually working) idea to raise to the level of multivariate polynomials the well-known identity expressing the signed Lah numbers by the Stirling numbers of the first and second kind \cite[p.\,44]{rior1958}, which is mirrored in the following

\begin{dfn}[Signed Lah polynomials]\label{def_lahpoly}
	$L_{n,k}:=\sum_{j=k}^n (-1)^j A_{n,j}B_{j,k}$.
\end{dfn}
\label{Signed Lah polynomials}

Using this definition we immediately regain by unification (and observing \eqref{unification_A}, \eqref{unification_B}) the numerical identity 
\[
L_{n,k}\comp 1=\sum_{j=k}^n (-1)^j s_1(n,j)s_2(j,k)=l(n,k), 
\]
which just has been alluded to. We also have $L_{n,k}\in\const[X_1^{-1},X_2,\ldots,X_n]$ and the homogeneity $L_{n,k}(t X_1,t X_2,\ldots)=t^{-(n-k)}L_{n,k}$. The instances of $L_{5,k}$, $1\leq k\leq 5$, may serve as an illustration:
\begin{align*}
L_{5,1}&=-\tfrac{210 X_2^4}{X_1^8}+\tfrac{120 X_3 X_2^2}{X_1^7}-\tfrac{30 X_4
   X_2}{X_1^6},\\
	L_{5,2}&=-\tfrac{270 X_2^3}{X_1^6}+\tfrac{40 X_3 X_2}{X_1^5}-\tfrac{10
   X_4}{X_1^4},\\
	L_{5,3}&=-\tfrac{120 X_2^2}{X_1^4},\quad	L_{5,4}=-\tfrac{20 X_2}{X_1^2},\quad
	L_{5,5}=-1.
\end{align*}
Indeed, the remarkable orthogonality relation $\ortho{L}_{n,k}=L_{n,k}$ also applies.

\begin{prop}\label{orthogonal_lahpoly}
	$\sum_{j=k}^n L_{n,j}L_{j,k}=\kronecker{n}{k}$\qquad$(1\leq k\leq n)$.
\end{prop}

\begin{proof}
The assertion follows from \definition{def_lahpoly} by a direct straightforward computation and applying twice \corollary{orthocompanions_A_B}.
\end{proof}

The main result concerning Lah polynomials provides a characterization of all regular B-representable polynomial families, which are orthogonal companions of their own.

\begin{thm}\label{mainthm_lahpoly}
Let $(Q_{n,k})$ be any regular B-representable family of polynomials. Then,
$\ortho{Q}_{n,k}=Q_{n,k}$ holds, if and only if there exists a family of FdB polynomials $(H_1,H_2,H_3,\ldots)\in\prod_{n\geq 1}\Phi_n[\invfuncs]$ such that $Q_{n,k}=L_{n,k}\comp H_{\num}$. 
\end{thm}

\begin{proof}
Sufficiency (`if'): Immediately from \proposition{orthogonal_lahpoly} by substituting $H_j$ for $X_j$ ($j=1,2,3,\ldots$).

Necessity (`only if'): Let $f$ be any function from $\invfuncs$. We then define the function $\varphi(x):=\sum_{n\geq 1}Q_{n,1}^f(0) \frac{x^n}{n!}$, which is invertible, as $Q_{n,k}$ is regular. With this we obtain
{\allowdisplaybreaks
\begin{align*}
	B_{n,k}^{\varphi}(0)&=B_{n,k}(Q_{1,1}^f(0),\ldots,Q_{n-k+1,1}^f(0))\\
		&=Q_{n,k}^f(0) &\hspace*{-2em}\text{(by \eqref{substitution_rule}, $Q_{n,k}$ B-representable)}\\
		&=(\ortho{Q}_{n,k})^f(0) &\text{(by assumption)}\\
		&=A_{n,k}(Q_{1,1}^f(0),\ldots,Q_{n-k+1,1}^f(0))&\hspace*{-5em}\text{(by \eqref{substitution_rule}, Prop.~\ref{B_rep_properties}\,(iv))}\\
		&=A_{n,k}^{\varphi}(0)=B_{n,k}^{\inv{\varphi}}(0)&\hspace*{-5em}\text{(by \eqref{connection_A_B}, $\varphi(0)=0$)}.
\intertext{%
By the Identity \lemma{identity_lemma} we have $\varphi=\inv{\varphi}$. As an involutory function, $\varphi$ can be written in the form $\varphi=\widetilde{g}\comp\inv{g}$, where $g$ is an appropriate function from $\invfuncs$, and $\widetilde{g}(x)=g(-x)$ (see, e.\,g.,\cite{mcca1980}). It follows
}
	Q_{n,k}^f(0)&=B_{n,k}^{\varphi}(0)=B_{n,k}^{\widetilde{g}\,\comp\,\inv{g}}(0)\\
      &=\sum_{j=k}^n B_{n,j}^{\inv{g}}(0)B_{j,k}^{\widetilde{g}}(0)&\hspace*{-5em}\text{(by Jabotinsky's \theorem{second_CR}\,(i))}\\
			&=\sum_{j=k}^n (-1)^j A_{n,j}^g(0)B_{j,k}^g(0)&\hspace*{-5em}\text{(by \eqref{connection_A_B}, $\inv{g}(0)=0$)}\\
			&=L_{n,k}^g(0)&\hspace*{-5em}\text{(by \definition{def_lahpoly})}.
\intertext{%
Hence, putting $h:=g\comp \inv{f}\in\invfuncs$ and $H_m:=\Phi_m(h)$ we get
}
	Q_{n,k}^f(0)&=L_{n,k}^{h\,\comp\,f}(0)=L_{n,k}\comp D^{\num}(h \comp f)(0)\\
	 &=L_{n,k}\comp\sum_{j=0}^{\num}D^j(h)(0)B_{\num,j}^f(0)&\hspace*{-5em}\text{~(by \eq{fdb_formula})}\\
	 &=L_{n,k}\comp\Phi_{\num}(h)^{f}(0)\\
	 &=L_{n,k}(H_1,H_2,H_3,\ldots)^f(0)\hspace*{-4em}&\text{~(by \eq{substitution_rule})}.
\end{align*}
}
Applying the argument from \remark{identical_polynomials} completes the proof.
\end{proof}

\begin{cor}\label{B_rep_lahpoly}
	The signed Lah polynomials $(L_{n,k})$ are B-representable.
\end{cor}

\begin{proof}
Let $h$ be the function from the proof of \theorem{mainthm_lahpoly}. Then, for $j=1,2,3,\ldots$ set $\inv{H}_j:=\Phi_j(\inv{h})$. By the FCR we have $H_j\comp\inv{H}_{\num}=X_j$. The assertion now readily follows by replacing each $X_j$ in $L_{n,k}(H_1,H_2,\ldots)=B_{n,k}(Q_{1,1},Q_{2,1},\ldots)$ with $\inv{H}_j$.
\end{proof}

The Lah polynomials allow us to characterize involutory functions.

\begin{prop}\label{involutory_functions}
A function $f\in\invfuncs$ is involutory, if and only if there exists $g\in\invfuncs$ such that $f(x)=\sum_{n\geq 1}L_{n,1}^g(0)\frac{x^n}{n!}$.
\end{prop}

\begin{proof}
Necessity: As in the proof of \theorem{mainthm_lahpoly} we can write $f=\widetilde{g}\comp\inv{g}$ for some $g\in\invfuncs$. Again it follows that $D^n(f)(0)=L_{n,1}^g(0)$.

\noindent
Sufficiency: By the assumption and by \corollary{B_rep_lahpoly} 
\begin{align*}
	B_{n,k}^{f}(0)&=B_{n,k}(L_{1,1}^g(0),\ldots,L_{n-k+1,1}^g(0))=L_{n,k}^g(0).
\end{align*}
Now, applying the FdB formula \eqref{fdb_formula} and \proposition{orthogonal_lahpoly} we obtain for every $n\geq 1$
\[
	D^n(f\comp f)(0)=\sum_{k=1}^n D^k(f)(0)B_{n,k}^{f}(0)=\sum_{k=1}^n L_{n,k}^g(0)L_{k,1}^g(0)=\kronecker{n}{1},
\]
that is, we have $f\comp f=\id$. 
\end{proof}
\label{Involutory functions}

From the above reasoning it can be seen immediately that, given any involution $f=\widetilde{g}\comp\inv{g}$, $g\in\invfuncs$, the corresponding $f$-polynomial $\Phi_n(f)$ takes the form
\begin{equation}\label{involution_poly}
		J_{g,n}:=\sum_{k=1}^n L_{k,1}^g(0)B_{n,k}.
\end{equation}
According to the FCR, the \emph{involution polynomials} \eqref{involution_poly} are self-inverse, that is, they satisfy the relation $J_{g,n}\comp J_{g,\num}=X_n$.

\subsection{Comtet's polynomials}
\sloppy
We now return to Comtet's attempt \cite{comt1973} (already mentioned in Section~2.1), to determine the class of polynomials associated with the higher-order Lie operator $(\theta D)^n$, $\theta\in\funcs$. A statement concerning expansion and recurrence, analogous to Propositions \ref{msp_1} and \ref{msp_2}, serves as the starting point.

\begin{prop}\label{comtet_poly}
There exist polynomials $C_{n,k}\in\const[X_0,\ldots,X_{n-k}]$ such that
\[
	(\theta D)^n=\sum_{k=0}^{n}C_{n,k}^{\,\theta}\cdot D^{k}.
\]
The family $(C_{n,k})$ is triangular, regular, and uniquely determined by the differential recurrence
\[
	C_{n+1,k}=X_0\left(C_{n,k-1}+\sum_{j=0}^{n-k}X_{j+1}\pderiv{C_{n,k}}{X_j}\right),\quad C_{n,0}=\kronecker{n}{0}.
\]
\end{prop}

\fussy
\begin{proof}
By a straightforward inductive argument as applied in the proof of \proposition{msp_1}; cf. \cite[Propositions 3.1 and 3.5]{schr2015}.
\end{proof}

From this the following representation can be inferred:
\begin{equation}\label{comtetpoly_diophantine}
	C_{n,k}=\sum_{\ptsind{2n-k}{n}} \gamma_{n,k}(r_0,\ldots,r_{n-k})X_0^{r_0}\cdots X_{n-k}^{r_{n-k}},
\end{equation}
the sum ranging over all non-negative integral values of $r_0$ to $r_{n-k}$ such that $r_0+\cdots+r_{n-k}=n$ and $r_0+2r_1+3r_2+\cdots=2n-k$. The coefficients $\gamma_{n,k}(r_0,\ldots,r_{n-k})$ turn out to be positive integers. In \cite[Section 5]{comt1973} Comtet has tabulated $(C_{n,k})_{1\leq k\leq n}$ up to $n=7$ and claimed (without proof) that $C_{n,k}\comp 1=c(n,k)$ and $C_{n,k}(1,1,0,\ldots,0)=s_2(n,k)$. His main result [ibid., Equation (8), p.\,166)] provides the following expression in diophantine form:
\begin{equation}\label{comtetpoly_explicit}
	C_{n,k}=\frac{X_0}{k!}\cdot\mspace{-14mu}\sum_{\rho(n,1)=n-k}k\cdot\prod_{j=1}^{n-1}\frac{k+\rho(n,j)-j}{r_{j}!}X_{r_j},
\end{equation}
where $\rho(n,j)$ denotes the sum $r_{1}+\cdots+r_{n-j}$ (with non-negative integers $r_1,r_2,r_3,\ldots$). This formula appears only to a modest extent suitable for computational purposes. For example, although $C_{6,2}$ consists of only 5 monomials, a total of 70 solutions of the equation $\rho(6,1)=4$ has to be checked in order to finally obtain $C_{6,2}=31 X_0^2 X_1^4+146  X_0^3 X_1^2 X_2+34 X_0^4 X_2^2+57 X_0^4 X_1 X_3+6 X_0^5 X_4$.
\label{Diophantine solutions}

In the following it will be shown that Comtet's polynomial family $(C_{n,k})$ can be smoothly integrated into our algebraic framework developed so far. In particular, it turns out that $C_{n,k}$ can be represented by the Stirling polynomials of the first and second kind.

\begin{thm}\label{A_rep_comtetpoly}
	$C_{n,k}=A_{n,k}(\widehat{R}_0,\ldots,\widehat{R}_{n-k})$\qquad $(0\leq k\leq n)$.
\end{thm}

\begin{proof}
Let $\varphi$ be any function such that $D(\varphi)\in\funcs_1$. According to Propositions~\ref{comtet_poly} and \ref{msp_1} we have
\[
	(\theta D)^n=\sum_{k=0}^n C_{n,k}^{\,\theta}D^k\quad\text{and}\quad(D(\varphi)^{-1}D)^n=\sum_{k=0}^n A_{n,k}^{\varphi}D^k.\tag{*}
\]
Choose $\theta=D(\varphi)^{-1}$, so that both expansions agree. Hence, by \proposition{reciprocal_function} we obtain for every $j\geq 1$
\[
	D^j(\varphi)=D^{j-1}(\theta^{-1})=\widehat{R}_{j-1}^{\,\theta}
\]
and by equating the coefficients of $D^k$ in (*) 
\[
	C_{n,k}^{\,\theta}=A_{n,k}^{\varphi}=A_{n,k}\comp D^{\num}(\varphi)=A_{n,k}(\widehat{R}_0^{\,\theta},\ldots,\widehat{R}_{n-k}^{\,\theta}).
\]
Now replace $D^j(\theta)$ by $X_j$ ($j=0,1,2,\ldots$) on both sides of the equation. This completes the proof.
\end{proof}

\begin{cor}\label{B_rep_comtetpoly}
	$(C_{n,k})$ and $(\ortho{C}_{n,k})$ are B-representable.
\end{cor}

\begin{proof}
\theorem{A_rep_comtetpoly} yields $C_{n,k}=A_{n,k}\comp \widehat{R}_{\num}$, whence by \eq{A_B_representable} and \eqref{composition_associative_law} $C_{n,k}=B_{n,k}\comp(A_{\num,1}\comp\widehat{R}_{\num})$, that is, $C_{n,k}$ is B-representable.\,---\,From \proposition{B_rep_properties}\,(iv) follows $\ortho{C}_{n,k}=B_{n,k}(\widehat{R}_0,\ldots,\widehat{R}_{n-k})$, hence $\ortho{C}_{n,k}$ is B-represen\-table.
\end{proof}

\begin{prop}\label{comtetpoly_special_values}
	\begin{align*}
		\text{\emph{(i)}}\quad&C_{n,k}(1,\ldots,1)=c(n,k)\\
		\text{\emph{(ii)}}\quad&\ortho{C}_{n,k}(1,\ldots,1)=(-1)^{n-k}s_2(n,k)\\
		\text{\emph{(iii)}}\quad&C_{n,k}(1,1,0,\ldots,0)=s_2(n,k)
	\end{align*}
\end{prop}

\begin{proof}
(i) Observing $\widehat{R}_j\comp 1=P_{j,-1}\comp 1=(-1)^j$ we obtain by \theorem{A_rep_comtetpoly} and \eq{unification_A}
{\allowdisplaybreaks
\begin{align*}
	C_{n,k}(1,\ldots,1)&=A_{n,k}(1,-1,\ldots,(-1)^{n-k})\\
	      &=(-1)^{n-k}A_{n,k}(1,\ldots,1)\\
				&=(-1)^{n-k}s_1(n,k)=c(n,k).
\intertext{%
(ii) $\ortho{C}_{n,k}(1,\ldots,1)=B_{n,k}(1,-1,1,-1,\ldots)=(-1)^{n-k}B_{n.k}(1,\ldots,1)$.\newline
(iii) Set $\varphi:=1+\id$; then $C_{n,k}^{\,\varphi}(0)=C_{n,k}(1,1,0,\ldots,0)$, and by \proposition{reciprocal_function} for every $j\geq 0$
}
	\widehat{R}_j^{\,\varphi}(0)&=D^j(\varphi^{-1})(0)=D^j((1+\id)^{-1})(0)\\
			&=(-1)^j j!=s_1(j+1,1)=A_{j+1,1}(1,\ldots,1),
\intertext{%
whence by \eq{B_A_representable} and \eq{unification_B}
}
  C_{n,k}(1,1,0,\ldots,0)&=C_{n,k}^{\,\varphi}(0)=A_{n,k}(\widehat{R}_0^{\,\varphi}(0),\ldots,\widehat{R}_{n-k}^{\,\varphi}(0))\\
	&=A_{n,k}(A_{1,1}(1),\ldots,A_{n-k+1,1}(1,\ldots,1))\\
	&=B_{n,k}(1,\ldots,1)=s_2(n,k).\qedhere
\end{align*}
}
\end{proof}

\sloppy
\begin{rem}
The question of how the coefficients $\gamma_{n,k}(r_0,\ldots,r_{n-k})$ of \eqref{comtetpoly_diophantine} are made up in detail appears to be rather tricky and must remain open for the time being. Attempting a direct evaluation of $A_{n,k}(\widehat{R}_0,\ldots,\widehat{R}_{n-k})$ in general form quickly leads to a piling up of more and more cumbersome expressions. While the coefficients of $A_{n,k}$ and $B_{n,k}$ are products of simple combinatorial terms (see \eq{bell_partitionsum} and \eq{stirling_partitionsum}), one might doubt whether this also applies to the coefficients of $C_{n,k}$. For example, consider
\begin{align*}
C_{n,n-4} &= \binom{n}{5}\frac{15n^3-150n^2+485n-502}{48} X_0^{n-4}X_1^4\\
	&+ \binom{n}{5}\frac{15n^2-85n+116}{6} X_0^{n-3} X_1^2 X_2
	+\binom{n}{5}\frac{5n-13}{3} X_0^{n-2}X_2^2\\
	&+ \binom{n}{5}\frac{5n-11}{2} X_0^{n-2}X_1 X_3 + \binom{n}{5}X_0^{n-1}X_4.
\end{align*}	
\fussy
Here, the first two coefficients $\gamma_{n,n-4}(n-4,4)$ and $\gamma_{n,n-4}(n-3,2,1)$, regarded as polynomials in $n$, cannot be written as products of linear factors over the field $\const$. However, replacing each $X_j$ by $\widehat{R}_j$ $(0\leq j\leq 4)$ and applying \corollary{basic_CR_polynomialcase}\,(iii) resembles a magic wand that turns $C_{n,n-4}\comp\widehat{R}_{\num}$ into
\begin{align*}
A_{n,n-4}(&X_0,\ldots,X_4) = 105\binom{n+3}{8}X_0^{-n-4}X_1^4	\\
                          &- 105\binom{n+2}{7}X_0^{-n-3} X_1^2 X_2 + 10\binom{n+1}{6}X_0^{-n-2}X_2^2\\
													&\hspace*{5em}+ 15\binom{n+1}{6}X_0^{-n-2}X_1 X_3 - \binom{n}{5}X_0^{-n-1}X_4.
\end{align*}	
Also for the simpler cases $A_{n,n-k}$ with $k=0,1,2,3$ one gets similarly closed product representations of the coefficients as here (see Todorov \cite[Equations (14) to (16)]{todo1985}). This could be seen as an indication that focusing on the iterations of $\theta D$ ultimately turns out to be a less well-posed problem. The remedy is, of course, simply the choice $\theta=D(\varphi)^{-1}$ (see \eq{todorovs_choice}).
\end{rem}
\label{Table: Comtet's polynomials}
\section{Applications to binomial sequences}
This section is to demonstrate the succinct way the classical topic of binomial sequences can be treated within the conceptual frame of the preceeding sections. Some new results will be proved.
\subsection{Definition and representation}
Let $f_n=f_n(t)\in\const[t]$ $(n=0,1,2,3\ldots)$ be a sequence of polynomials with $\deg f_n=n$. Then, $f_0,f_1,f_2,\ldots$ is said to be \emph{binomial}, or \emph{of binomial type}, if for every $n\geq 0$
\begin{equation}\label{def_binseq}
	f_n(s+t)=\sum_{k=0}^n\binom{n}{k}f_{n-k}(s)f_k(t).
\end{equation}
Note that clearly $f_0=1$ and $f'_n(0)\neq 0$.

The sequences $t^n$ and $(t)_n$ are binomial (cf. Stanley \cite[Exercise 5.37]{stan1999} for more examples). Knuth \cite{knut1992a} also deals with binomial sequences $f_n$, but in their guise of \emph{convolution polynomials} $f_n/n!$. As is well-known, binomial sequences are closely related to the exponential polynomials. This is reflected in the following two statements.
\begin{prop}\label{rep_binseq}
Let $f_0,f_1,f_2,\ldots$ be a sequence of polynomials from $\const[t]$. Then the following holds:
\begin{align*}
  \text{\emph{(i)}\quad}(f_n)&\text{~binomial~}\iff\exists\,\varphi\in\invfuncs:f_n(t)=[\tfrac{x^n}{n!}]e^{t\varphi(x)}\\
	\text{\emph{(ii)}\quad}(f_n)&\text{~binomial~}\iff f_n(t)=B_n(tf'_1(0),\ldots,tf'_n(0))
\end{align*}
\end{prop}
\begin{proof}
(i) `$\Leftarrow$': Let $\varphi$ be any invertible function and denote by $f_n(t)$ the $n$th Taylor coefficient  of $e^{t\varphi}$; then it is easy to check that $f_n$ satisfies the binomial property \eqref{def_binseq}.\,---\,`$\Rightarrow$': Conversely, given $f_n$ binomial, then $g(t,x):=\sum_{n\geq 0}f_n(t)(x^n/n!)$ satisfies $g(s+t,x)=g(s,x)g(t,x)$. Writing $g(t,x)=\sum_{n\geq 0}g_n(x)(t^n/n!)$, one gets $g_n(x)=g_1(x)^n$ for all $n\geq 0$ (by an inductive argument), and hence $g(t,x)=\sum_{n\geq 0}(g_1(x)t)^n/n!=e^{tg_1(x)}$. On the other hand we have $g_1(x)=\pderiv{}{t}g(t,x)|_{t=0}=\sum_{n\geq 1}f'_n(0)(x^n/n!)$ and therefore $g'_1(0)=f'_1(0)\neq 0$. Thus, taking $\varphi:=g_1\in\invfuncs$ delivers $[x^n/n!]e^{t\varphi}=[x^n/n!]g(t,x)=f_n(t)$.

\sloppy
(ii) According to (i) every binomial sequence can be written $f_n(t)=D^n(e^{t\varphi})(0)$. We have $\varphi(x)=\sum_{n\geq 1}f'_n(0)(x^n/n!)$ (as in the proof of (i)), whence $D^j(\varphi)(0)=f'_j(0)$. Now the FdB formula \eqref{fdb_formula} yields $D^n(e^{t\varphi})(0)=\sum_{k=0}^n t^k B_{n,k}^{\varphi}(0)=B_n\comp tD^{\num}(\varphi)(0)=B_n(tf'_{1}(0),\ldots,tf'_{n}(0))$ (see also \cite[Proposition~7.3\,(i)]{schr2015}).
\end{proof}

\fussy
\subsection{Some results on substitution}
Several authors have dealt with evaluating the Bell polynomials in cases the variables have been replaced by special values from combinatorial number families and, more generally, also by binomial sequences (see, e.\,g., \cite{comt1974,feng2020,miho2008b,miho2010,wtwa2009,yang2008}). Keeping in mind that, according to \proposition{rep_binseq}\,(ii), the latter are instances of the exponential polynomials, it is not surprising that statements concerning the substitution of the $B_n$ into polynomials can be readily transferred to the substitution of binomial sequences.

The subsequent list contains some of the more interesting results together with comments or short proofs. Suppose that $(f_n)$ is a binomial sequence. Then the following statements hold:
\begin{align}
	\label{log_binseq}
	&L_n(f_1(t),\ldots,f_n(t))=tf'_n(0),\\
	\label{subst_potential}
	&P_{n,k}(f_1(t),\ldots,f_n(t))=f_n(kt),\\
	\label{subst_reciprocal}
	&R_n(f_1(t),\ldots,f_n(t))=f_n(-t),\\
	\label{subst_treepoly}
	&T_n(f_1(t),\ldots,f_n(t))=f_{n-1}(nt).
\end{align}
\begin{proof}
\eq{log_binseq} follows directly from \proposition{rep_binseq}\,(ii) by observing the inverse relationship between $L_n$ and $B_n$ (see \remark{simple_fdb_inversion}).
\eq{subst_potential} is derived from \proposition{product_identity} by replacing each $X_j$ with $tf'_j(0)$. Equations \eqref{subst_reciprocal} and \eqref{subst_treepoly} are special cases of \eqref{subst_potential}; recall that $R_n=P_{n,-1}$ and $T_n=P_{n-1,n}$.
\end{proof}

In the same way, without any calculation one gets from \eqref{factorial_potentialpoly} and \eqref{subst_potential}  
\begin{align}
	\label{subst_factorial}
	&F_{n,k}(f_1(t),\ldots,f_n(t))=\sum_{j=0}^k s_1(k,j)f_n(jt),\\
\intertext{furthermore from \proposition{comtet_logpoly} }	
	\label{subst_logarithmic}
	&L_n(f_1(t),\ldots,f_n(t))=\sum_{k=1}^n(-1)^{k-1}\frac{1}{k}\binom{n}{k}f_n(kt)\\
\intertext{and likewise from \eq{bell_substitutions_2}}	
	&B_{n,k}(f_1(t),\ldots,f_{n-k+1}(t))=\frac{1}{k!}\sum_{j=0}^k(-1)^{k-j}\binom{k}{j}f_n(jt).\label{bertrand_binseq}
\end{align}
\begin{rem}
\eq{bertrand_binseq} has been proven by Yang \cite[Theorem 2]{yang2008}. Mihoubi, by using the methods of Umbral Calculus, derived a slightly more general version \cite[Proposition 2]{miho2008b}. Equations \eqref{subst_potential}--\eqref{subst_logarithmic} are new.
\end{rem}

\noindent
From \eq{idempot_poly} one immediately obtains by \eqref{subst_potential}
\begin{align}\label{subst_idempot}
	&B_{n,k}(f_0(t),2f_1(t),3f_2(t),\ldots)=\binom{n}{k}f_{n-k}(kt).\\
\intertext{%
Then, \proposition{cor_forests_from_trees} yields the corresponding orthogonal companion
}\label{subst_forest_0}
	&A_{n,k}(f_0(t),2f_1(t),3f_2(t),\ldots)=\binom{n-1}{k-1}f_{n-k}(-nt).
\end{align}
Furthermore, by substituting \eqref{subst_treepoly} into the equations of \theorem{forests_from_trees} we readily get
\vspace*{-2ex}
\begin{align}\label{subst_forest_1}
	&B_{n,k}(f_0(t),f_1(2t),f_2(3t),\ldots)=\binom{n-1}{k-1}f_{n-k}(nt).\\
\intertext{%
together with the orthogonal companion
}\label{subst_forest_2}
	&A_{n,k}(f_0(t),f_1(2t),f_2(3t),\ldots)=\binom{n}{k}f_{n-k}(-kt).
\end{align}
\begin{rem}
\eq{subst_idempot} has been proved by Yang \cite[Theorem 1]{yang2008}. Abbas and Bouroubi \cite[Theorem 3]{abbo2005} have shown \eqref{subst_forest_1} for the special case $t=1$, which just represents the binomial variant of the Khelifa/Cherruault identity mentioned in \remark{khelifa_identity}. The identities \eqref{subst_forest_0} and \eqref{subst_forest_2} are new.
\end{rem}

\subsection{A binomial sequence related to trees}
Let $\tau$ denote the exponential generating function for labeled rooted trees, that is, $\tau(x)=\sum_{n\geq 1}n^{n-1}x^n/n!$ (see \remark{power_unification}\,(iii)). Knuth and Pittel \cite{knpi1989} have introduced univariate expressions $t_n(y)$, $n\geq 0$, called `tree polynomials' and defined by
\begin{equation*}\label{def_univtreepoly}
	t_n(y):=\left[\frac{x^n}{n!}\right]\frac{1}{(1-\tau(x))^y}.
\end{equation*}
Defining the functions $g,\varphi\in\invfuncs$ by $g(x):=x(1-x)^{-1}$ and \mbox{$\varphi:=\logm\comp g\comp\tau$}, we obtain $(1-\tau(x))^{-y}=e^{y\varphi(x)}$. From this it follows by \proposition{rep_binseq}\,(i) that 
\[
t_0(y)=1,\quad t_1(y)=y,\quad t_2(y)=3y+y^2,\quad t_3(y)=17y+9y^2+y^3,\ldots 
\]
is actually a binomial sequence of polynomials from $\integers[y]$. If we write $t_n(y)=\sum_{k=0}^n t_{n,k}y^k$, then the coefficient $t_{n,k}$ counts the total number of mappings of $\left\{1,2,\ldots,n\right\}$ into itself having exactly $k$ different cycles.

Knuth and Pittel derived some integral formulas and a $\Gamma$-function representation of the $t_n(y)$ (see ibid., equations (2.12), (2.13)), inasmuch the focus of their paper is on the asymptotic behavior of these polynomials. Supplementary to this, we will prove here a new explicit representation built up exclusively by means of elementary combinatorial operations.

\begin{thm}\label{knuth_pittel_poly}
	\begin{align*}
\text{\emph{(i)}\qquad}&t_{n,k}=\mspace{-2mu}\sum_{k\leq j\leq i\leq n}\mspace{-4mu}s_1(j,k)\frac{n^{n-i}i!}{j!}\binom{i-1}{j-1}\binom{n-1}{i-1};\\ 
\text{\emph{(ii)}\qquad}&t_n(y)=B_n(yt_{1,1},\ldots,yt_{n,1}),\text{~where for~}r\geq 1\\
                       &t_{r,1}=\mspace{-2mu}\sum_{1\leq j\leq i\leq r}\mspace{-4mu}(-1)^{j-1}\frac{r^{r-i}i!}{j}\binom{i-1}{j-1}\binom{r-1}{i-1}.
	\end{align*}
\end{thm}

\begin{proof}
(i) With the terms introduced above we have
\[
	\sum_{n\geq 0}t_n(y)\frac{x^n}{n!}=e^{y\varphi(x)}=\sum_{n\geq 0}\bigg(\sum_{k=0}^n B_{n,k}^{\varphi}(0)y^k\bigg)\frac{x^n}{n!}
\]
\sloppy
(cf. \cite[Proposition 7.3\,(i)]{schr2015}), whence $t_n(y)=\sum_{k=0}^n B_{n,k}^{\varphi}(0)y^k$ and by Jabotinsky's \theorem{second_CR}\,(i)
\[
	t_{n,k}=B_{n,k}^{\varphi}(0)=B_{n,k}^{\logm\,\comp\,(g\,\comp\,\tau)}(0)=\sum_{j=k}^n B_{n,j}^{g\,\comp\,\tau}(0)B_{j,k}^{\logm}(0).\tag{*}
\]
\fussy
Now observe that $B_{n,k}^{\logm}(0)=B_{n,k}^{\inv{\expm}}(0)=A_{n,k}^{\expm}(0)=s_1(n,k)$; thus again applying Jabotinsky's formula, we obtain from (*)
\[
	t_{n,k}=\sum_{j=k}^n s_1(j,k)\sum_{i=j}^n B_{n,i}^{\tau}(0)B_{i,j}^g(0).\tag{**}
\]
\sloppy
Evaluating both substitutions gives 
\[
	B_{n,i}^{\tau}(0)=B_{n,i}(1^0,2^1,3^2,\ldots)=\binom{n-1}{i-1}n^{n-i}
\]
by applying \theorem{forests_from_trees}\,(i) (forest numbers, cf. \remark{khelifa_identity}), and 
\[
	B_{i,j}^{g}(0)=B_{i,j}(1!,2!,3!,\ldots)=l^{+}(i,j)=\frac{i!}{j!}\binom{i-1}{j-1}
\]
(unsigned Lah numbers, cf. \eqref{B_rep_unsignedlahpoly}). Finally, we substitute these results into (**) thus arriving at the asserted statement. 

\fussy
(ii) Recall that $s_1(j,1)=(-1)^{j-1}(j-1)!$. Thus $t_{r,1}$ is obtained from (i) by taking $n=r$ and $k=1$. Furthermore $t'_{r}(0)=t_{r,1}$, whence the assertion follows according to \proposition{rep_binseq}\,(ii). 
\end{proof}

\subsection{Coupling binomial sequences}
In their seminal paper \cite{muro1970} Mullin and Rota developed a general theory of binomial sequences. They introduced a family of shift-invariant linear differential operators (called `delta operators'}) on the vector space of polynomials such that each sequence of binomial type can be associated uniquely to a specific (`basic') delta operator. On that basis, they were able to describe the exact form of the connection between two given binomial sequences (see also Aigner \cite[Chapter III]{aign1979}).

In the following it will be demonstrated how the main result of these investigations can also (alternatively) be established by using only properties of the Bell polynomials.

\begin{thm}[Mullin \& Rota 1970]\label{thm_mullin_rota}
Let $(f_n)$ and $(g_n)$ be any two binomial sequences. Then there exist unique constants $c_{n,k}\in\const$, with $c_{n,k}=0$ for $n<k$, such that $f_n(t)=\sum_{k=0}^n c_{n,k}g_k(t)$ for every $n\geq 0$, and the sequence $(h_n)$ defined by $h_n(t):=\sum_{k=0}^n c_{n,k}t^k$ is of binomial type.
\end{thm}

\begin{proof}
According to \proposition{rep_binseq} there are functions $\varphi,\psi\in\invfuncs$ such that $f_n$ and $g_n$ can be written in the form
\begin{equation}\label{two_binseqs}
g_n(t)=\sum_{k=0}^n t^k B_{n,k}^{\varphi}(0)\quad\text{and}\quad f_n(t)=\sum_{k=0}^n t^k B_{n,k}^{\psi}(0).
\end{equation}
Since every sequence of binomial type consists of polynomials linearly independent in $\const[t]$, it follows that 
\begin{align}\label{connecting_coeffs}
f_n(t)&=\sum_{k=0}^n c_{n,k}g_k(t)\\
\intertext{%
with uniquely determined connecting coefficients $c_{n,k}\in\const$. Now substituting the right-hand sides of \eqref{two_binseqs} into \eqref{connecting_coeffs} we obtain
}
	\sum_{k=0}^n t^k B_{n,k}^{\psi}(0)&=\sum_{k=0}^n c_{n,k}\sum_{j=0}^k t^j B_{k,j}^{\varphi}(0)\notag\\
			 &=\sum_{k=0}^n t^k\sum_{j=k}^n c_{n,j} B_{j,k}^{\varphi}(0)\notag\\
\intertext{%
after rearranging the double series. Hence, by equating the coefficients of $t^k$
}
\label{connecting_the_bellpolys_1}
	B_{n,k}^{\psi}(0)&=\sum_{j=k}^n c_{n,j} B_{j,k}^{\varphi}(0).\\
\intertext{%
On the other hand, applying Jabotinsky's formula (\theorem{second_CR}\,(i)) to the `decomposition' $\psi=\varphi\comp(\inv{\varphi}\comp\psi)$ yields
}	\label{connecting_the_bellpolys_2}
	B_{n,k}^{\psi}(0)&=\sum_{j=k}^n B_{n,j}^{\inv{\varphi}\,\comp\,\psi}(0)B_{j,k}^{\varphi}(0).
\end{align}
\label{Rearranging double series}
We now set $V_{n,j}:=c_{n,j}-B_{n,j}^{\inv{\varphi}\,\comp\,\psi}(0)$. Then, from \eqref{connecting_the_bellpolys_1} and \eqref{connecting_the_bellpolys_2}
\begin{equation}\label{vanishing_coeffs}
	\sum_{j=k}^n V_{n,j}B_{j,k}(\varphi_1,\ldots,\varphi_{j-k+1})=0\qquad(0\leq k\leq n),
\end{equation}
where $\varphi_1,\varphi_2,\varphi_3,\ldots$ are the Taylor coefficients of $\varphi$. We take for $k$ the values $n,n-1,n-2,\ldots$ in that order to show $V_{n,j}=0$ for $j=n,n-1,\ldots,0$. In the case $k=n$ \eqref{vanishing_coeffs} becomes $V_{n,n}B_{n,n}(\varphi_1)=V_{n,n}\varphi_1^n=0$, hence $V_{n,n}=0$ because of $\varphi_1\neq 0$. For $k=n-1$ we similarly get $0=V_{n,n-1}B_{n-1,n-1}(\varphi_1)+V_{n,n}B_{n,n-1}(\varphi_1)=V_{n,n-1}\varphi_1^{n-1}$, that is, $V_{n,n-1}=0$, and so on until finally $V_{n,0}=0$. All in all, \eqref{vanishing_coeffs} implies
\[
c_{n,j}=B_{n,j}^{\inv{\varphi}\,\comp\,\psi}(0)=B_{n,j}(a_1,\ldots,a_{n-j+1})\quad(0\leq j\leq n)
\]
with $a_j=D^j(\inv{\varphi}\comp\psi)(0)$. Clearly we now have $c_{n,j}=0$ for $n<j$, and furthermore
\[
 h_n(t)=\sum_{k=0}^n c_{n,k}t^k=\sum_{k=0}^n t^k B_{n,k}(a_1,\ldots,a_{n-k+1})=B_n(ta_1,\ldots,ta_n),
\] 
where $h'_j(0)=a_j$. Hence $(h_n)$ is binomial by \proposition{rep_binseq}\,(ii). 
\end{proof}

\sloppy
From the above proof we learn that for any two given binomial sequences their connecting coefficients are B-representable. The question here is whether one of the sequences can be of binomial type while the other is not. The (negative) answer is provided by the following \emph{both-or-none statement}.

\fussy
\begin{thm}\label{both_or_none_thm}
Let $(f_n)$ and $(g_n)$ be any sequences of polynomials from $\const[t]$ and $a_1\,(\neq 0),a_2,a_3,\ldots$ a sequence of constants from $\const$. Suppose that
\[
f_n(t)=\sum_{k=0}^n g_k(t)B_{n,k}(a_1,\ldots,a_{n-k+1})
\]
holds for all $n\geq 0$. Then we have: $g_n \text{~binomial~} \iff f_n \text{~binomial~}$.
\end{thm}

\begin{proof}
\sloppy
1. Assume $g_n(t)$ to be of binomial type; then, as in the proof of \theorem{thm_mullin_rota}, $g_n(t)=\sum_{k=0}^n t^k B_{n,k}^{\varphi}(0)$ with $\varphi\in\invfuncs$. We set $\psi(x):=\sum_{n\geq 1}a_n(x^n/n!)$, which turns out to be invertible because of $\psi'(0)=a_1\neq 0$. From this we obtain
{\allowdisplaybreaks
\begin{align*}
	f_n(t)&=\sum_{k=0}^n\bigg(\sum_{j=0}^k t^j B_{k,j}^{\varphi}(0)\bigg)B_{n,k}^{\psi}(0)\\
	      &=\sum_{k=0}^n t^k\bigg(\sum_{j=k}^n B_{n,j}^{\psi}(0)B_{j,k}^{\varphi}(0)\bigg)&\text{(by rearranging the double series)}\\
				&=\sum_{k=0}^n t^k B_{n,k}^{\varphi\,\comp\,\psi}(0).\quad&\text{(by Jabotinsky's formula)}
\end{align*}
} 
It follows by \proposition{rep_binseq} that $f_n(t)$ is of binomial type.

2. Conversely, suppose $f_n(t)$ is binomial. This case will be reduced to the situation of part 1. With the same denotations we have $f_n(t)=\sum_{k=0}^n g_k(t)B_{n,k}^{\psi}(0)$, $\psi\in\invfuncs$. To this now, the generalized Stirling inversion (\proposition{generalized_stirling_inversion}) can be applied thus yielding
\begin{align*}
	g_n(t)&=\sum_{k=0}^n f_k(t)A_{n,k}^{\psi}(0)\\
	      &=\sum_{k=0}^n f_k(t)B_{n,k}^{\inv{\psi}}(0)=\sum_{k=0}^n f_k(t)B_{n,k}(\inv{a}_1,\ldots,\inv{a}_{n-k+1}),
\end{align*}
where $\inv{a}_j=D^j(\inv{\psi})(0)$. Since $\inv{a}_1=\nicefrac{1}{a_1}\neq 0$, we see from the statement of part~1 that $(g_n)$ is binomial.
\end{proof}

\fussy
With any binomial sequence it is now easy to create a new one by linearly combining its polynomials with instances of the partial Bell polynomials as connecting coefficients. The following special case has been established by Yang \cite[Lemma 2, p.\,53]{yang2008}.
\begin{cor}\label{yang_lemma}
A sequence $(f_n)$ is binomial, if and only if for every $n\geq 0$
\[
	f_n(t)=\sum_{k=0}^n(t)_k B_{n,k}(f_1(1),\ldots,f_{n-k+1}(1)).
\]
\end{cor}
\begin{proof}
Since we have
\[
	(t)_n=\sum_{k=0}^n t^k s_1(n,k)=\sum_{k=0}^n t^k B_{n,k}^{\logm}(0)=\left[\frac{x^n}{n!}\right]e^{t\logm(x)},
\]
$(t)_n$ turns out to be of binomial type (by \proposition{rep_binseq}\,(i)). Suppose now
\[
	f_n(t)=\sum_{k=0}^n(t)_k B_{n,k}(a_1,\ldots,a_{n-k+1})\tag{*}
\]
with $a_1\neq 0$. One clearly has $f_j(1)=a_j$, and $f_n(t)$ is binomial by \theorem{both_or_none_thm}.\,---\,Conversely, assume $f_n(t)$ is binomial. The Mullin-Rota \theorem{thm_mullin_rota} then yields connecting coefficients $B_{n,k}(a_1,\ldots,a_{n-k+1})$ such that equation (*) is satisfied.
\end{proof}
\section{Lagrange inversion polynomials}
\subsection{Some preliminaries}
After having encountered various inverse relationships ba\-sed on orthogonality in previous sections, we now turn to the type of \emph{compositional inversion} already mentioned in \example{simple_fdb_inversion} and \remark{hopf_algebra}. First we want to free ourselves from the assumption made there, according to which the polynomials in question must be FdB, i.\,e., elements of $\Phi_n[\invfuncs]$. Therefore, we will consider sequences of arbitrary polynomials $(U_n)$ and $(V_n)$ such that for every $n\geq 0$
\begin{equation}\label{compositional_inversion}
	U_n\comp V_{\num}=X_n\quad\text{and}\quad V_n\comp U_{\num}=X_n.
\end{equation}

However, the idea that a polynomial sequence remains (uniquely) associated to a function should still be adhered to. In order to roughly sketch the intended situation here, let us assume that $f,g\in\funcs$ are somehow cha\-racterized by sequences of constants $c_0,c_1,c_2,\ldots$ and $d_0,d_1,d_2,\ldots$, respectively. Next suppose there is a sequence of polynomials $(U_n)$ such that $d_n=U_n(c_0,\ldots,c_n)$ for every $n\geq 0$. Then $U_n$ will be called \emph{conversion polynomial} (of $f$ with respect to $g$). Reversely, if we additionally have conversion polynomials $V_n$ of $g$ (w.r.t. $f$), then the pair $(U_n,V_n)$ satisfies \eqref{compositional_inversion}. In the special case $g=\inv{f}$, the corresponding conversion polynomials will henceforth be called (\emph{generalized}) \emph{Lagrange inversion polynomials}.

\sloppy
The classical Lagrange inversion is a special case that arises when $f$ and $\inv{f}$ are characterized by their respective Taylor coefficients $f_1,f_2,\ldots$ and $\inv{f}_1,\inv{f}_2,\ldots$. While a major part of the extensive literature on the subject is devoted to its analytical and combinatorial aspects (e.\,g. \cite{krat1988,gess2016}), explicit formulas for the coefficients of $\inv{f}$ in the form of polynomial expressions
\[
	\inv{f}_n=\Lambda_n(f_1,\ldots,f_n)
\]
\fussy
have been studied at times, for instance \cite{whit1951}, \cite[p.\,412]{mofe1953}, \cite[Section 2.5]{gess2016}. Comtet \cite[p.\,151]{comt1974} replaced these determinantal and diophantine representations with the following elegant formula (the prehistory of which is sketched in \cite{john2002a}):
\begin{equation}\label{comtet_formula}
		\Lambda_n=\sum_{k=0}^{n-1}(-1)^k X_1^{-(n+k)}B_{n-1+k,k}(0,X_2,\ldots,X_n).
\end{equation}
\eq{comtet_formula} is a special case of \theorem{mainresult} ($=$~Theorem~6.1 in \cite{schr2015}), from which immediately follows that $\Lambda_n=A_{n,1}$. Alternatively, if one wants to avoid the associate Bell polynomials $\widetilde{B}_{n-1+k,k}$ in \eqref{comtet_formula}, the Schl\"omilch-Schl\"afli representation \eqref{schloemilch_schlaefli_A_B} could be used (taking $k=n-1$) to obtain
\begin{equation*}
	\Lambda_n=\sum_{k=0}^{n-1}(-1)^k\binom{2n-1}{n-1-k}X_1^{-(n+k)}B_{n-1+k,k}.
\end{equation*}
Finally, according to \corollary{corollary_stirling}, $\Lambda_n$ can also be expressed by virtue of the reciprocal polynomials and the tree polynomials.

Since the Lagrange inversion polynomial $\Lambda_n$ is the conversion polynomial of $f$ (w.r.t. $\inv{f}$) and of $\inv{f}$ (w.r.t. $f$) as well, it obviously must be self-inverse: $\Lambda_n\comp \Lambda_{\num}=X_n$ (see also \eq{B_A_representable} for $k=1$). 

In addition, we see from Theorems \ref{inverse_power} and \ref{power_of_function} that $B_{n,k}(f_1,\ldots,f_{n-k+1})$ and $A_{n,k}(f_1,\ldots,f_{n-k+1})$ are the Taylor coefficients of $f^k$ and $\inv{f}^k$, respectively. Again we encounter an orthogonality relation, this time through raising a function and its inverse to the $k$th power.

\subsection{A modified inversion problem}
We start with a theorem that gives the solution to a modified inversion problem.
\begin{thm}\label{comtet_thm_F}
Let $s$ be any positive integer and let $c_0,c_1,c_2,\ldots\in\const$ be any sequence of constants with $c_0=1$. Then, $f(x)=\sum_{n\geq 0}\frac{c_n}{n!}x^{s n+1}$ has an inverse of the form $\inv{f}(x)=\sum_{n\geq 0}\frac{d_n}{n!}x^{s n+1}$ with $d_0=1$ and $d_n=\Lambda_n(c_1,\ldots,c_n)$ for every $n\geq 1$, where
\[
	\Lambda_{n}=\sum_{k=1}^n (-1)^k \binom{sn+k}{k-1}(k-1)!B_{n,k}.
\]
\end{thm}

\label{A special inverse relation} 

This statement is Mihoubi's \cite{miho2010} slightly modified (symmetrized) version of a theorem established by Comtet \cite[Theorem F, p.\,151]{comt1974}. The possibly first proof to be found in the literature results from specializing a pair of inverse polynomials investigated by Birmajer, Gil and Weiner \cite[Theorem~4.6, Example~4.9]{bigw2012}. In the following, too, we will regain Comtet's theorem as a corollary of a much more general statement. However, to achieve this, we keep following the idea (outlined in Section 7.1) of switching between certain representations of a function and its inverse using generalized Lagrange inversion polynomials. The one from \theorem{comtet_thm_F}, $\Lambda_n$, is the conversion polynomial of $f$ (characterized by $c_0,c_1,c_2,\ldots$) w.r.t. $\inv{f}$ (characterized by $d_0,d_1,d_2,\ldots$), and \emph{vice versa}, that is, we obviously have $\Lambda_{n}\comp \Lambda_{\num}=X_n$.
\label{Comtet's Theorem F}

In order to obtain a more comprehensive class of inversion polynomials, we will modify the assumptions made in \theorem{comtet_thm_F} as follows: Suppose $f$ is any invertible function and denote by $f_n$ and $\inv{f}_n$, $n\geq 1$, the respective Taylor coefficients of $f$ and $\inv{f}$. Both functions are now to be given the new form
\begin{equation}\label{rewrite_functions}
	\text{(i)~~}f(x)=\sum_{n\geq 0}\frac{c_n}{n!}a(x)\varphi(x)^n\quad\text{~and~}\quad\text{(ii)~~}\inv{f}(x)=\sum_{n\geq 0}\frac{d_n}{n!}b(x)\psi(x)^n,
\end{equation}
where $a,b$ and $\varphi,\psi$ are assumed to be fixed functions (with suitable properties). We will see that the required conversions can be achieved by merely adjusting the sequences $c_0,c_1,c_2,\ldots$ and $d_0,d_1,d_2,\ldots$. In a final step, Lagrange inversion polynomials $U_n,V_n$ with property \eqref{compositional_inversion} will be constructed satisfying $d_n=U_n(c_0,\ldots,c_n)$ and $c_n=V_n(d_0,\ldots,d_n)$.

First we have to fix the conditions that guarantee the existence of the above representations (7.3). For this purpose we define the functions
\[
	c(x)=\sum_{n\geq 0}c_n\frac{x^n}{n!}\qquad\text{and}\qquad d(x)=\sum_{n\geq 0}d_n\frac{x^n}{n!}
\]
so that the series \eqref{rewrite_functions} can equivalently be rewritten as function terms:
\begin{equation}\label{function_terms}
	\text{(i)~~}f=a\cdot(c\comp\varphi)\quad\text{~and~}\quad\text{(ii)~~}\inv{f}=b\cdot(d\comp\psi).
\end{equation}
It is at once clear that $f,\varphi,\psi$ must be invertible functions (what will be tacidly assumed for the remainder of this section). Furthermore, we ge\-nerally suppose that both $c$ and $f$ are neither part of $a$ nor of $\varphi$, and that the corresponding equally holds for $d$ and $\inv{f}$ with respect to $b$ and $\psi$. As usual, the Taylor coefficients of $a$ and $b$ are denoted by $a_0,a_1,a_2,\ldots$ and $b_0,b_1,b_2,\ldots$, respectively.

Let us first turn to the equation (i) in \eqref{function_terms}. 

\begin{prop}\label{conv_c_to_f}
	Suppose that equation \emph{(\ref{function_terms},\,i)} holds. Then, for all $n\geq 0:$ $f_n=\Gamma_n(a,\varphi)(c_0,\ldots,c_n)$, where
\[
	\Gamma_n(a,\varphi):=\sum_{k=0}^n\bigg(\sum_{j=k}^n\binom{n}{j}a_{n-j}B_{j,k}^{\varphi}(0)\bigg)X_k.
\]
\end{prop}
\begin{proof}
According to \definition{def_omega} the conversion polynomial $\Gamma_n(a,\varphi)$ we are looking for is $\Omega_n(f\,|\,c)$. Therefore, we obtain by \eq{omega_product}
\begin{equation*}
	\Gamma_n(a,\varphi)=\Omega_n(a\cdot(c\comp\varphi)\,|\,c)=\sum_{j=0}^n\binom{n}{j}\Omega_{n-j}(a\,|\,c)\Omega_j(c\comp\varphi\,|\,c).\tag{*}
\end{equation*}
Since $c$ does not appear in either $a$ or $\varphi$, we have $\Omega_{n-j}(a\,|\,c)=a_{n-j}$ and observing \eqref{omega_composition_0case} and \eqref{omega_phi}
\begin{equation*}
	\Omega_j(c\comp\varphi\,|\,c)=\sum_{k=0}^j\Omega_k(c\,|\,c)(B_{j,k}\comp\Omega_{\num}(\varphi\,|\,c))=\sum_{k=0}^j B_{j,k}^{\varphi}(0)X_k.
\end{equation*}
Substitution into (*) and rearranging the double series in $\Gamma_n(a,\varphi)$ yields the assertion.
\end{proof}

\begin{rem}
Note that the requirements for \proposition{conv_c_to_f} can be weakend to $f\in\funcs$ and $\varphi\in\funcs_0$. Under these assumptions, the special case $a=1$ is of interest here. It yields 
\[
	\sum_{n\geq 0}c_n\frac{\varphi(x)^n}{n!}=\sum_{n\geq 0}\bigg(\sum_{k=0}^n B_{n,k}^{\varphi}(0)c_k\bigg)\frac{x^n}{n!},
\]
which, by further specializing $c_k=t^k$, leads to the expansion of $e^{t\varphi(x)}$ used in Section 6 (\proposition{rep_binseq}); see also \cite[Proposition 7.3]{schr2015}.
\end{rem}

Applying \remark{invertible_product} to (\ref{function_terms},\,i) we see that $f$ is invertible in just the following two cases, both of which are covered by the linear form $\Gamma_n(a,\varphi)$: 
\begin{enumerate}[(1)]
	\item \quad$a\in\funcs_1$ and $c\comp\varphi\in\invfuncs$,
	\item \quad$a\in\invfuncs$ and $c\comp\varphi\in\funcs_1$.
\end{enumerate}

The inverse transformations that compute the $c_n$'s from the $f_n$'s also turn out to be linear forms. There are two case-specific types that differ from one another.
\begin{prop}\label{conv_f_to_c_case1}
Suppose that \emph{(\ref{function_terms},\,i)} and $a\in\funcs_1$. Then, for every $n\geq 0:$ \mbox{$c_n=\inv{\Gamma}_n^{(1)}(a,\varphi)(0,f_1,\ldots,f_n)$}, where
\[
	\inv{\Gamma}_n^{(1)}(a,\varphi):=\sum_{k=0}^n\bigg(\sum_{j=k}^n\binom{j}{k}\widehat{R}_{j-k}(a_0,\ldots,a_{j-k})A_{n,j}^{\varphi}(0)\bigg)X_k.
\]
\end{prop}
\begin{proof}
In order to obtain $c_n$, we transform (\ref{function_terms},\,i) into
\begin{equation}\label{transformed_equation}
	c=(a^{-1}\cdot f)\comp\inv{\varphi}.
\end{equation}
From this we get the inverse form
{\allowdisplaybreaks
\begin{align*}
	\inv{\Gamma}_n^{(1)}(&a,\varphi)=\Omega_n(c\,|\,f)=\Omega_n((a^{-1}\cdot f)\comp\inv{\varphi}\,|\,f)\\
	&=\sum_{j=0}^n\Omega_j(a^{-1}\cdot f\,|\,f)(B_{n,j}\comp\Omega_{\num}(\inv{\varphi}\,|\,f))&\hspace*{-0.7em}\text{(by \eqref{omega_composition_0case})}\\
	&=\sum_{j=0}^n\sum_{k=0}^j\binom{j}{k}\Omega_{j-k}(a^{-1}\,|\,f)\Omega_k(f\,|\,f)B_{n,j}^{\inv{\varphi}}(0)&\hspace*{-0.7em}\text{(by \eqref{omega_product})}\\
	&\mspace{-6mu}=\sum_{j=0}^n\sum_{k=0}^j\binom{j}{k}\widehat{R}_{j-k}(a_0,\ldots,a_{j-k})A_{n,j}^{\varphi}(0)X_k.&\hspace*{-0.7em}\text{(by Prop.\,\ref{reciprocal_function}, \eqref{connection_A_B})}
\end{align*}
}
The assertion now follows by rearranging the double series as in the proof of \proposition{conv_c_to_f}.
\end{proof}

\begin{prop}\label{conv_f_to_c_case2}
	Suppose that \emph{(\ref{function_terms},\,i)} and $a\in\invfuncs$. Then, for every $n\geq 0:$ \mbox{$c_n=\inv{\Gamma}_n^{(2)}(a,\varphi)(f_1,\ldots,f_{n+1})$}, where
\[
	\inv{\Gamma}_n^{(2)}(a,\varphi):=\sum_{k=0}^n\frac{1}{k+1}\bigg(\sum_{j=k}^n\binom{j}{k}\widehat{R}_{j-k}(\tfrac{a_1}{1},\ldots,\tfrac{a_{j-k+1}}{j-k+1})A_{n,j}^{\varphi}(0)\bigg)X_{k+1}.
\]	
\end{prop}
\begin{proof}
Again we use \eq{transformed_equation}, but now with $a_0=0$ and $a_1\neq 0$. Defining $h\in\funcs_1$ by $h(x):=\sum_{n\geq 0}\frac{a_{n+1}}{n+1}\frac{x^n}{n!}$ we obtain $a(x)=x\cdot h(x)$, and hence $a^{-1}\cdot f=h^{-1}\cdot(f/\id)$. As in the proof of \proposition{conv_f_to_c_case1}, it follows
\begin{align*}
	\inv{\Gamma}_n^{(2)}(a,\varphi)&=\Omega_n(c\,|\,f)=\Omega_n((h^{-1}\cdot\tfrac{f}{\id})\comp\inv{\varphi}\,|\,f)\\
	&=\sum_{j=0}^n\Omega_j(h^{-1}\cdot\tfrac{f}{\id}\,|\,f)A_{n,j}^{\varphi}(0).
\end{align*}
Now we have by \eqref{omega_product} and \proposition{reciprocal_function}
\[
	\Omega_j(h^{-1}\cdot\tfrac{f}{\id}\,|\,f)=\sum_{k=0}^j\binom{j}{k}\widehat{R}_{j-k}(\tfrac{a_1}{1},\ldots,\tfrac{a_{j-k+1}}{j-k+1})\Omega_k(\tfrac{f}{\id}\,|\,f).
\]
\sloppy
Observing $\Omega_k(\tfrac{f}{\id}\,|\,f)=\tfrac{1}{k+1}X_{k+1}$ and rearranging the double series completes the proof.
\end{proof}

\sloppy
For the sake of clarity we will distinguish the cases (1) $a\in\funcs_1$ and (2)~$a\in\invfuncs$ also for $\Gamma_n(a,\varphi)$ in Pro\-po\-sition~\ref{conv_c_to_f}. We  set $\Gamma_n^{(1)}(a,\varphi):=\Gamma_n(a,\varphi)(0,X_1,\ldots,X_n)$, since $a_0\neq 0$ requires $c_0=0$ in (\ref{function_terms},\,i). Other\-wise, we use $\Gamma_n^{(2)}(a,\varphi)$ to denote the result obtained by putting $a_0=0$ in $\Gamma_n(a,\varphi)$. 

\fussy
\vspace*{2ex}
Now we get from Propositions \ref{conv_c_to_f}, \ref{conv_f_to_c_case1}, \ref{conv_f_to_c_case2} the

\begin{cor}\label{inverse_conversions}
For every $n\geq 1$ we have 
\[
	\Gamma_n^{(i)}(a,\varphi)\comp \inv{\Gamma}_{\num}^{(i)}(a,\varphi)=\inv{\Gamma}_n^{(i)}(a,\varphi)\comp\Gamma_{\num}^{(i)}(a,\varphi)=X_n\quad(i=1,2).
\]	
\end{cor}

\sloppy
\begin{cnv}
Similarly as $\Gamma_n(a,\varphi)$ covers both cases (1)~$a\in\funcs_1$ and \mbox{(2)~$a\in\invfuncs$}, we formally write for brevity $\inv{\Gamma}_n(a,\varphi)$, which is to mean either $\inv{\Gamma}_n^{(1)}(a,\varphi)$ or $\inv{\Gamma}_n^{(2)}(a,\varphi)$ depending on which of the two cases actually occurs.
\end{cnv}

\fussy
We are now in a position to construct the desired general type of Lagrange inversion polynomial that fits the situation given in \eq{rewrite_functions}.

\begin{dfn}[Generalized Lagrange inversion polynomial]\label{gen_lagrange_invpoly}
~\newline Given any $\varphi,\psi\in\invfuncs$ and $a,b\in\funcs_1\cup\invfuncs$, we define
\[
	\Lambda_n(a,\varphi\,|\,b,\psi):=\inv{\Gamma}_n(b,\psi)\comp A_{\num,1}\comp\Gamma_{\num}(a,\varphi).
\]
\end{dfn}

\begin{thm}\label{generalized_lagrange_inversion}
Suppose $f,\varphi,\psi\in\invfuncs$ and $a,b\in\funcs_1\cup\invfuncs$. Then, there are uniquely determined sequences of constants $(c_n)$ and $(d_n)$ such that the equations \eqref{rewrite_functions} hold. Moreover, we have $d_n=\Lambda_n(a,\varphi\,|\,b,\psi)(c_0,\ldots,c_n)$ and $c_n=\Lambda_n(b,\psi\,|\,a,\varphi)(d_0,\ldots,d_n)$, or equivalently
\[
	\Lambda_n(a,\varphi\,|\,b,\psi)\comp \Lambda_{\num}(b,\psi\,|\,a,\varphi)=X_n\qquad\text{for every $n\geq 1$}.
\] 
\end{thm}

\begin{proof}
From $f$ we obtain $(c_n)$ by applying $\inv{\Gamma}_n(a,\varphi)$ according to Propositions \ref{conv_f_to_c_case1} and \ref{conv_f_to_c_case2}; and in the same unique way one gets $(d_n)$ by applying $\inv{\Gamma}_n(b,\psi)$ to the Taylor coefficients of $\inv{f}$. 

\sloppy
Upon closer inspection, it becomes clear that in general we have $\Lambda_n(a,\varphi\,|\,b,\psi)\in\const[X_0^{-1},X_1^{-1},X_0,X_1,\ldots,X_{n+1}]$. Depending on the choice of $a,b$ the (Laurent) polynomial $\Lambda_n(a,\varphi\,|\,b,\psi)$ can take four diffe\-rent forms. It may suffice here to consider, for example, the case $a\in\invfuncs$ and $b\in\funcs_1$. Observing $f_0=\inv{f}_0=0$ we obtain
{\allowdisplaybreaks
\begin{align*}
	d_n&=\inv{\Gamma}_n^{(1)}(b,\psi)(0,\inv{f}_1,\ldots,\inv{f}_n)&\text{\hspace*{-2em}(Proposition~\ref{conv_f_to_c_case1})}\\
	&=\inv{\Gamma}_n^{(1)}(b,\psi)(A_{0,1}(0),A_{1,1}(f_1),\ldots,A_{n,1}(f_1,\ldots,f_n))&\text{\hspace*{-2em}(Remark~\ref{inverse_function})}\\
	&=\inv{\Gamma}_n^{(1)}(b,\psi)\comp A_{\num,1}\comp f_{\num}\\
	&=\inv{\Gamma}_n^{(1)}(b,\psi)\comp A_{\num,1}\comp \Gamma_{\num}^{(2)}(a,\varphi)(c_0,\ldots,c_{\num-1},0)&\text{\hspace*{-2em}(Proposition~\ref{conv_c_to_f})}\\
&=\Lambda_n(a,\varphi\,|\,b,\psi)(c_0,\ldots,c_{n-1},0).&\text{\hspace*{-2em}(Definition~\ref{gen_lagrange_invpoly})}
\intertext{%
\sloppy
In the last step, the indeterminate $X_j$ in $\inv{\Gamma}_n^{(1)}(b,\psi)(0,A_{1,1},\ldots,A_{n,1})$ is replaced by $\Gamma_{j}^{(2)}(a,\varphi)(c_0,\ldots,c_{j-1},0)$ for each $j=1,\ldots,n$.\,---\,
In the reverse direction one obtains in a similar way
}
c_n&=\inv{\Gamma}_n^{(2)}(a,\varphi)(f_1,\ldots,f_{n+1})&\text{\hspace*{-2em}(Proposition~\ref{conv_f_to_c_case2})}\\
	&=\inv{\Gamma}_n^{(2)}(a,\varphi)\comp A_{\num,1}\comp \inv{f}_{\num}&\\
	&=\inv{\Gamma}_n^{(2)}(a,\varphi)\comp A_{\num,1}\comp \Gamma_{\num}^{(1)}(b,\psi)(0,d_1,\ldots,d_{\num})&\\
	&=\Lambda_n(b,\psi\,|\,a,\varphi)(0,d_1,\ldots,d_n).&\qedhere
\end{align*}
}
\end{proof}

\fussy
\label{Table: Generalized Lagrange inversion polynomials}
In the general case, the inversion polynomials $\Lambda_n(a,\varphi\,|\,b,\psi)$ turn out to be rather complicated expressions even for small values of $n$. We therefore confine ourselves here to the study of two interesting special cases: $a=b=1$ and $a=b=\id$.

\begin{thm}\label{lagrangepoly_special_1}
Let $\varphi,\psi$ be invertible functions. Then, for every $n\geq 1$
\[
	\Lambda_n(1,\varphi\,|\,1,\psi)=\sum_{k=1}^n\bigg(A_{n,k}^{\psi}(0)\sum_{j=1}^k A_{j,1}^{\varphi}(0)A_{k,j}\bigg).
\]
\end{thm}

\begin{proof}
Since $b_0=1$ and $b_n=0$, $n\geq 1$, we have $\widehat{R}_{j-k}(b_0,b_1,\ldots,b_{j-k})=\kronecker{j}{k}$, and hence $\inv{\Gamma}_n^{(1)}(1,\psi)=\sum_{k=1}^n A_{n,k}^{\psi}(0)X_k$ by \proposition{conv_f_to_c_case1}. Thus it follows from \theorem{generalized_lagrange_inversion}
\begin{equation}\label{lagrangepoly_special_1_eq1}
	\Lambda_n(1,\varphi\,|\,1,\psi)=\sum_{k=1}^n A_{n,k}^{\psi}(0)A_{k,1}(\Gamma_1(1,\varphi),\ldots,\Gamma_k(1,\varphi)).
\end{equation}
Now consider an arbitrary $f\in\invfuncs$ of the form $f=c\comp\varphi$. Then $c_0=0$, $c_1\neq 0$ and $f_0=0$, $f_1\neq 0$, and by \proposition{conv_c_to_f}: $f_n=\Gamma_n(1,\varphi)(0,c_1,\ldots,c_n)$ for $n\geq 0$. Applying \eqref{lagrangepoly_special_1_eq1} to $(c_0,\ldots,c_n)$ thus yields
\begin{equation}\label{lagrangepoly_special_1_eq2}
	\Lambda_n(1,\varphi\,|\,1,\psi)(0,c_1,\ldots,c_n)=\sum_{k=1}^n A_{n,k}^{\psi}(0)A_{k,1}^{f}(0).
\end{equation}
Finally, we obtain by \theorem{second_CR}\,(ii) 
\[
	A_{k,1}^{f}(0)=\sum_{j=k}A_{j,1}^{\varphi}(0)A_{k,j}(c_1,\ldots,c_{k-j+1}).
\]
Substitution of the latter into \eqref{lagrangepoly_special_1_eq2} gives the assertion.
\end{proof}

\sloppy
\begin{rem}
$\Lambda_n(1,\varphi\,|\,1,\psi)$ solves the problem of inverting a function \mbox{$c\comp\varphi\in\invfuncs$} into a function of the prescribed form $d\comp\psi$ by computing the sequence $d_1,d_2,\ldots$ from $c_1,c_2,\ldots$. In the particular case $\varphi=\psi=\id$ this amounts to the classical Lagrange inversion. Indeed, observing $A_{n,k}^{\id}(0)=A_{n,k}(1,0,\ldots,0)=\kronecker{n}{k}$ we get from \theorem{lagrangepoly_special_1}: $\Lambda_n(1,\id\,|\,1,\id)=A_{n,1}$.
\end{rem}

\fussy
Our second case, $a=b=\id$, is a little more involved; its solution at least can be formulated more succintly by introducing the following special polynomial:
\begin{equation*}\label{def_special_invpoly}
	\widehat{I}_{n,k}:=\sum_{j=0}^n(-1)^j(k+j)_{j-1}X_0^{-(k+1+j)}B_{n,j}.
\end{equation*}
We shall also use the instance $I_{n,k}:=\widehat{I}_{n,k}(1,X_1,\ldots,X_n)$ (which, by the way, is FdB, since $I_{n,k}=\Phi_n(\eta_k)$ with $\eta_k(x):=((1+x)^{-(k+1)}-1)/(k+1)$).

\begin{thm}\label{lagrangepoly_special_2}
Let $\varphi,\psi$ be invertible functions. Then, for every $n\geq 0$
\[
	\Lambda_n(\id,\varphi\,|\,\id,\psi)=\sum_{k=0}^n\bigg(A_{n,k}^{\psi}(0)\sum_{j=0}^k B_{k,j}^{\varphi}(0)\,\widehat{I}_{j,k}\bigg).
\]
\end{thm}
\begin{proof}
Suppose the situation \eqref{function_terms} with $a=b=\id$, which implies $a_n=b_n=\kronecker{n}{1}$, $c\in\funcs_1$ and $f_0=0$. Similar to the proof of \theorem{lagrangepoly_special_1}, we get by \proposition{conv_f_to_c_case2}
\begin{equation}\label{lagrangepoly_special_2_eq1}
	\inv{\Gamma}_n^{(2)}(\id,\psi)=\sum_{k=0}^n A_{n,k}^{\psi}(0)\cdot\frac{1}{k+1}X_{k+1},
\end{equation}
hence by \theorem{generalized_lagrange_inversion}
\begin{equation}\label{lagrangepoly_special_2_eq2}
	\Lambda_n(\id,\varphi\,|\,\id,\psi)=\sum_{k=0}^n A_{n,k}^{\psi}(0)\cdot\frac{1}{k+1}A_{k+1,1}(\Gamma_1(\id,\varphi),\ldots,\Gamma_{k+1}(\id,\varphi)).
\end{equation}
We now set $g=c\comp\varphi$ and thus obtain by \proposition{conv_c_to_f} for every $n\geq 1$
\begin{align*}
	f_n=\Gamma_n(\id,\varphi)(c_0,\ldots,c_n)&=\sum_{k=0}^n\bigg(\sum_{j=k}^n\binom{n}{j}\kronecker{(n-j)}{1}B_{j,k}^{\varphi}(0)\bigg)c_k\\
	&=n\sum_{k=0}^{n-1}B_{n-1,k}^{\varphi}(0)c_k\\
	&=n\Omega_{n-1}(g\,|\,c)^{c}(0)\\
	&=ng_{n-1}.
\end{align*}
Combining this result with \eqref{lagrangepoly_special_2_eq2} gives
\begin{equation}\label{lagrangepoly_special_2_eq3}
	\Lambda_n(\id,\varphi\,|\,\id,\psi)(c_0,\ldots,c_n)=\sum_{k=0}^n A_{n,k}^{\psi}(0)\frac{A_{k+1,1}(g_0,2g_1,\ldots,(k+1)g_k)}{k+1}.
\end{equation}
Note that $g_0=c_0$; then by \proposition{cor_forests_from_trees} 
\begin{align*}
	A_{k+1,1}(g_0,2g_1,\ldots,(k+1)g_k)&=\widehat{P}_{k,-(k+1)}(g_0,g_1,\ldots,g_k)\\
		&=\sum_{j=0}^k(-k-1)_j\,c_0^{-k-1-j}B_{k,j}^{g}(0).
\end{align*}

\vspace*{-1ex}
\noindent
Here we can evaluate $B_{k,j}^{g}(0)$ with the help of Jabotinsky's formula (\theorem{second_CR}\,(i)) as follows:\\[-2ex]
\[
	B_{k,j}^{g}(0)=B_{k,j}^{c\,\comp\,\varphi}(0)=\sum_{i=j}^k B_{k,i}^{\varphi}(0)B_{i,j}(c_1,\ldots,c_{i-j+1}).
\]
\vspace*{-0.5ex}
\noindent
With this the fractional expression in \eqref{lagrangepoly_special_2_eq3} becomes
\[
	\frac{A_{k+1,1}(g_0,2g_1,\ldots)}{k+1}\mspace{-2mu}=\mspace{-2mu}\sum_{j=0}^k\sum_{i=j}^k\frac{(-k-1)_j}{k+1}\,c_0^{-k-1-j}B_{k,i}^{\varphi}(0)B_{i,j}(c_1,c_2,\ldots).
\]
\vspace*{-1.2ex}
Rearranging the double series and observing
\[
	\frac{(-k-1)_j}{k+1}=(-1)^j(k+j)_{j-1}
\]

\vspace*{-1.7ex}
\noindent
we obtain
\vspace*{-1ex}
\[
	\frac{A_{k+1,1}(g_0,2g_1,\ldots,(k+1)g_k)}{k+1}=\sum_{i=0}^k B_{k,i}^{\varphi}(0)\widehat{I}_{i,k}(c_0,c_1,\ldots,c_k).
\]
Substitution into \eqref{lagrangepoly_special_2_eq3} finally gives the desired result.
\end{proof}

\sloppy
We conclude this section by showing that \theorem{comtet_thm_F} ($=$ Comtet's Theorem~F) can be obtained as a corollary from \theorem{lagrangepoly_special_2}.

\fussy
\begin{proof}
Considering the assumptions of \theorem{comtet_thm_F}, one might be tempted at first glance to directly evaluate $\Lambda_n(\id,\id^s\,|\,\id,\id^s)$. This of course fails, because $\id^s$ ($s\geq 2$) is not invertible. However, \theorem{lagrangepoly_special_2} remains valid even with the condition $\varphi\in\invfuncs$ weakened to $\varphi\in\funcs_0$. Therefore, as an alternative, we will first deal with calculating $\Lambda_n(\id,\varphi\,|\,\id,\psi)$, where $\varphi=\id^s(\in\funcs_0)$ and $\psi=\id(\in\invfuncs)$. \theorem{lagrangepoly_special_2} yields
\begin{equation}\label{comtet_thm_F_eq1}
 \Lambda_n(\id,\id^s\,|\,\id,\id)=\sum_{k=0}^n\bigg(\kronecker{n}{k}\sum_{j=0}^k B_{k,j}^{\id^s}(0)\,\widehat{I}_{j,k}\bigg)=\sum_{j=0}^n B_{n,j}^{\id^s}(0)\,\widehat{I}_{j,n}.
\end{equation}
Since $D^{r}(\id^s)(0)=\kronecker{s}{r}s!$, we have
\begin{align*}
	B_{n,j}^{\id^s}(0)&=B_{n,j}(0,\ldots,0,s!,0,\ldots,0)&\text{($s!$ at position $s$)}\\
		&=(s!)^j\cdot B_{n,j}(0,\ldots,0,1,0,\ldots,0)&\text{(by homogeneity)}\\
		&=(s!)^j\cdot\frac{(sj)!}{j!(s!)^j}&\text{(if $n=sj$, else 0; by \eqref{bell_partitionsum})}\\
		&=\kronecker{n}{(sj)}\frac{(sj)!}{j!}.
\end{align*}
Thus \eqref{comtet_thm_F_eq1} becomes
\begin{equation}\label{comtet_thm_F_eq2}
	\Lambda_{n}(\id,\id^s\,|\,\id,\id)=\sum_{j=0}^{n}\kronecker{n}{(sj)}\frac{(sj)!}{j!}\,\widehat{I}_{j,n}.
\end{equation}
According to \theorem{generalized_lagrange_inversion} the constants $e_n=\Lambda_{n}(\id,\id^s\,|\,\id,\id)(c_0,c_1,\ldots,c_n)$ are coefficients satisfying $\inv{f}(x)=\sum_{n\geq 0}\frac{e_n}{n!}x^{n+1}$. From \eq{comtet_thm_F_eq2} we immediately see that $e_n$ is unequal to zero, if and only if $n$ is an integral multiple of $s$. Assuming $n=sr$ for some integer $r\geq 0$ and observing $c_0=1$, we get
\[
	e_{sr}=\frac{(sr)!}{r!}\widehat{I}_{r,sr}(c_0,c_1,\ldots,c_r)=\frac{(sr)!}{r!}I_{r,sr}(c_1,\ldots,c_r),
\]
hence
\[
	\inv{f}(x)=\sum_{r\geq 0}\frac{e_{sr}}{(sr)!}\,x^{sr+1}=\sum_{r\geq 0}\frac{d_r}{r!}\,x^{sr+1}
\] 
with 
\begin{align*}
	d_r&=I_{r,sr}(c_1,\ldots,c_r)\\
	   &=\sum_{j=0}^r(-1)^j \binom{sr+j}{j-1}(j-1)!B_{r,j}(c_1,\ldots,c_{r-j+1}).\qedhere
\end{align*}
\end{proof}

\begin{rem}
The polynomials $\widehat{I}_{n,sn}$ and $I_{n,sn}$ are self-inverse, that is, we have $\widehat{I}_{n,sn}\comp\widehat{I}_{\num,s\num}=X_n$ ($n\geq 0$) and $I_{n,sn}\comp I_{\num,s\num}=X_n$ ($n\geq 1$).
\end{rem}
%
\section{Reciprocity theorems}
In this final section we will be occupied by establishing reciprocity laws for several previously studied classes of polynomials. Stanley \cite{stan1974} pointed out that reciprocity (or `duality between two related enumeration problems') is a `rather vague concept' that only becomes clearer through concrete examples. This also applies to the reciprocity statements we are concerned with here. We are usually dealing with two families of polynomials (or sequences of numbers) that arise in some way from certain opposing aspects of a situation, while both families are actually united by their law of reciprocity.

According to Stanley \cite[p.\,15/16]{stan1986} the relationship between the number of $k$-combi\-na\-tions of $n$ elements \emph{with} repetitions and the corresponding number of combinations \emph{without} repetitions is `the simplest instance of a combinatorial reciprocity theorem' (which we shall make use of in the following):
\begin{equation}\label{reciprocity_combinations}
	\binom{n+k-1}{k}=(-1)^k\binom{-n}{k}.
\end{equation}

\begin{rem}\label{binomial_coeffs_def}
\eq{reciprocity_combinations} is valid for integers $k\geq 0$. We will extend it also in the case $k<0$ by setting $\binom{-n}{k}=0$, if $k>-n$, and $\binom{-n}{k}=\binom{-n}{-n-k}$ otherwise. Here again \eqref{reciprocity_combinations} can be applied because of $-n-k\geq 0$.
\end{rem}

Another particularly typical example are the Stirling numbers of the first and second kind. Knuth \cite{knut1992b} has repeatedly emphasized their importance and provided insightful historical comments on their reciprocity. This `beautiful and easily remembered law of duality' can be expressed in different (equivalent) ways, such as
\begin{align}
		\label{reciprocity_cycle_numbers}
		s_2(n,k)&=c(-k,-n)\\
		\label{reciprocity_stirling_numbers}
		\text{or~~}s_2(n,k)&=(-1)^{n-k}s_1(-k,-n)
\end{align}
\sloppy
`implying that there really is only one ``kind'' of Stirling number' [p.\,412, ibid.]. In \cite{knut1992a} Knuth has started with generalizing these relationships to sequences of coefficients belonging to arbitrary pairs of inverse functions.\footnote{See also the enlightening hint from I.\,Gessel that Knuth mentioned in this context.}
\vspace*{2ex}

The first aim of this section is to find a \emph{polynomial analogue} of this reciprocity, that is, a multivariable identity that turns into \eqref{reciprocity_cycle_numbers} or \eqref{reciprocity_stirling_numbers} through unification. From this point of view the resulting numerical identities might appear, so to speak, rather as a shadow of the rich intrinsic structure that is preserved in the corresponding original polynomial equations.

It is clear that $s_2(n,k)$ is to be understood as $B_{n,k}\comp 1$. As for the signless Stirling numbers of the first kind, we know the two options $c(n,k)=Z_{n,k}\comp 1$ (cf. Section~5.3) and $c(n,k)=C_{n,k}\comp 1$ (by \proposition{comtetpoly_special_values}\,(i)). The former must be discarded for compelling reasons the reader will find discussed in \cite[Remark 6.3]{schr2015}. According to \theorem{A_rep_comtetpoly}, the latter amounts to \mbox{$C_{n,k}\comp 1$} $=A_{n,k}((-1)^0,\ldots,(-1)^{n-k})=(-1)^{n-k}s_1(n,k)$, that is, the resulting situation is the same as in \eq{reciprocity_stirling_numbers}. In summary, we therefore re\-cognize the following identity as the best suitable candidate for the desired polynomial version of \eqref{reciprocity_stirling_numbers}: 
\begin{equation}\label{reciprocity_A_B}
	B_{n,k}=(-1)^{n-k}A_{-k,-n}.	
\end{equation}

\noindent
Looking at \eqref{reciprocity_A_B} we face a problem that needs to be solved: \emph{extending the index domain to the integers}. As far as \eqref{reciprocity_stirling_numbers} is concerned, Gould \cite{goul1960} \cite[p.\,417]{knut1992b} was possibly the first to observe that the domain of Stirling numbers can be extended to negative values of $n$ by using the fact that $s_1(n,n-k)$ and $s_2(n,n-k)$ are polynomials in $n$ of degree $2k$. As a consequence, these polynomials can also be defined for arbitrary complex (and \emph{a fortiori} negative integer) values. Details of Gould's method are described in \cite[Section 14.1]{qugo2016}, where also a proof of \eqref{reciprocity_stirling_numbers} is given with the help of the Schl\"omilch-Schl\"afli formula for the signed Stirling numbers of the first kind.

\begin{rem}\label{extended_bellpoly}
One quickly sees that Gould's trick obviously does not work when transferred to the indices of the polynomial families in question. Alternatively, however, the following identity can be used for the same purpose:
\begin{equation}\label{Bell_through_ass_B}
	B_{n,n-k}=\sum_{j=0}^k\binom{n}{k+j}X_1^{n-k-j}\widetilde{B}_{k+j,j}.
\end{equation}
\eq{Bell_through_ass_B} can easily be derived from Eq.\,[3l] in \cite[p.\,136]{comt1974} (cf.~also \cite[Corollary 4.5]{schr2015}) with a few calculations and index shifts. Based on this identity, one can obtain \eq{reciprocity_A_B} as a consequence of Theorem~\ref{mainresult} ($=$ Theorem~6.1 in \cite{schr2015}), and also vice versa, it can be shown that Theorem~\ref{mainresult} follows from \eq{reciprocity_A_B}. The proofs are omitted here, as we are going to take a quite different path.
\end{rem}
\label{Equivalent version of reciprocity}

To enlarge the domain of indices, we propose an easy and more straightforward procedure, which is based on the results obtained in Section 3. First, as customary, we set the value of a void sum to zero. If $n$ is negative, we therefore have according to \eqref{fdbpoly_0case} $\Phi_n(f)=0$ and, in particular, $\widehat{P}_{n,k}=B_{n,k}=0$. Next, let us have a look at the Corollaries \ref{family_B} and \ref{family_A}. In both identities, the partial Bell polynomials $B_{n,k}$ and their orthogonal companions $A_{n,k}$ are expressed by certain instances of $\binom{n}{k}\widehat{P}_{n-k,k}$ and of $\binom{n-1}{k-1}\widehat{P}_{n-k,n}$, respectively. We now have nothing else to do but \emph{redefine} $B_{n,k}$ \emph{and} $A_{n,k}$ \emph{by the right-hand sides of these identities under the assumption $n,k\in\integers$}. That's all!

Given any integer $n$, it is easy to check that $B_{n,k}=0$ if and only if $k>n$ or $n,k$ have unequal signs; the same holds for $A_{n,k}$. The original Stirling polynomials are restricted to indices $n=k=0$ and $1\leq k\leq n$, that is, they are equal to zero for all other values of $n,k$. After their redefinition the extended domain additionally includes $k\leq n\leq-1$. For example, instead of $B_{-3,-5}=0$ (old value), we now get the new value
\[
	B_{-3,-5}=\binom{-3}{-3+5}\widehat{P}_{2,-5}(\tfrac{X_1}{1},\tfrac{X_2}{2},\tfrac{X_3}{3})=45X_1^{-7}X_2^{2}-10X_1^{-6}10X_3,
\]
which is not new at all, since it is equal to $(-1)^{-3+5}A_{5,3}$\,---\,evidently an instance of \eq{reciprocity_A_B}. The following matrix $(B_{n,k})$ with \mbox{$-4\leq n,k\leq 4$} shows, in a neighborhood of $(n,k)=(0,0)$, the family of Stirling polynomials  united by their fundamental reciprocity law:
{\small
\[
\left(
\begin{array}{ccccccccc}
 \frac{1}{X_1^4} & 0 & 0 & 0 & 0 & 0 & 0 & 0 & 0 \\
 \frac{6 X_2}{X_1^5} & \frac{1}{X_1^3} & 0 & 0 & 0 & 0 & 0 & 0 & 0 \\
 \frac{15 X_2^2}{X_1^6}-\frac{4 X_3}{X_1^5} & \frac{3 X_2}{X_1^4} &
   \frac{1}{X_1^2} & 0 & 0 & 0 & 0 & 0 & 0 \\
 \frac{15 X_2^3}{X_1^7}-\frac{10 X_3 X_2}{X_1^6}+\frac{X_4}{X_1^5} & \frac{3
   X_2^2}{X_1^5}-\frac{X_3}{X_1^4} & \frac{X_2}{X_1^3} & \frac{1}{X_1} & 0 & 0
   & 0 & 0 & 0 \\
 0 & 0 & 0 & 0 & 1 & 0 & 0 & 0 & 0 \\
 0 & 0 & 0 & 0 & 0 & X_1 & 0 & 0 & 0 \\
 0 & 0 & 0 & 0 & 0 & X_2 & X_1^2 & 0 & 0 \\
 0 & 0 & 0 & 0 & 0 & X_3 & 3 X_1 X_2 & X_1^3 & 0 \\
 0 & 0 & 0 & 0 & 0 & X_4 & 3 X_2^2+4 X_1 X_3 & 6 X_1^2 X_2 & X_1^4 \\
\end{array}
\right)
\]
}
Let us now turn to the proof of the reciprocity law \eqref{reciprocity_A_B}.
\begin{proof}
If $n,k$ have different signs, then \eqref{reciprocity_A_B} holds because of $B_{n,k}=0$ and $A_{-k,-n}=0$. For $n=k=0$ both sides of \eqref{reciprocity_A_B} are equal to 1. Suppose now that $n,k$ are either both positive or both negative integers. In the case $k>n$ we then have $\widehat{P}_{n-k,\pm k}(\cdots)=0$ because of $n-k<0$ and therefore $B_{n,k}=0=A_{-k,-n}$ (according to the redefinition by the right-hand side expressions in \corollary{family_A} and \corollary{family_B}, respectively). Thus the two cases $1\leq k\leq n$ and $k\leq n\leq-1$ remain. Assuming the latter we obtain $n-k\geq 0$, and hence
{\allowdisplaybreaks
\begin{align*}
	&B_{n,k}=\binom{n}{n-k}\widehat{P}_{n-k,k}(\tfrac{X_1}{1},\tfrac{X_2}{2},\ldots)&\hspace*{-7em}\text{(\corollary{family_B}, \remark{binomial_coeffs_def})}\\
	       &~=(-1)^{n-k}\binom{-k-1}{n-k}\widehat{P}_{n-k,k}(\tfrac{X_1}{1},\tfrac{X_2}{2},\ldots)&\hspace*{-7em}\text{(\eq{reciprocity_combinations})}\\
				 &~=(-1)^{n-k}\binom{-k-1}{-n-1}\widehat{P}_{n-k,-k}(\widehat{R}_0(\tfrac{X_1}{1}),\widehat{R}_1(\tfrac{X_1}{1},\tfrac{X_2}{2}),\ldots)&\hspace*{-7em}\text{(Cor.\,\ref{multiplication_rule}\,(ii))}\\
				&~=(-1)^{n-k}A_{-k,-n}.&\hspace*{-7em}\text{(\corollary{family_A})}
\end{align*}
}
The case $1\leq k\leq n$ can be done in practically the same way and may be left to the reader.
\end{proof}

Of course, \eq{reciprocity_stirling_numbers} immediately follows from \eq{reciprocity_A_B} through unification. On the other hand, it is easy to generalize \eq{reciprocity_A_B} to a reciprocity theorem for B-representable polynomials.

\begin{thm}[General Reciprocity Law]\label{reciprocity_B_rep}
Let $(Q_{n,k})$ be any regular B-representable family of polynomials. Then for all $n,k\in\integers$
\[
	\ortho{Q}_{n,k}=(-1)^{n-k}Q_{-k,-n} .
\]
\end{thm}

\begin{proof}
The above redefinitions of $B_{n,k}$ and $A_{n,k}$ can obviously be regarded as applying also to $Q_{n,k}=B_{n,k}\comp Q_{\num,1}$ and $\ortho{Q}_{n,k}=A_{n,k}\comp Q_{\num,1}$, respectively. Thus \eq{reciprocity_A_B} immediately yields
\begin{equation*}
		\ortho{Q}_{n,k}=A_{n,k}\comp Q_{\num,1}=(-1)^{n-k}B_{-k,-n}\comp Q_{\num,1}=(-1)^{n-k}Q_{-k,-n}.\qedhere
\end{equation*}
\end{proof}
\begin{exms}\label{examples_reciprocity}
In Section 5 we examined a handful of regular polynomial fami\-lies, all of which obey this law of reciprocity: $Z_{n,k}$ (cycle indicators), $W_{n,k}$ (forest polynomials), $L_{n,k}$ (signed Lah polynomials, and unsigned: $L^{+}_{n,k}$) as well as $C_{n,k}$ (Comtet's polynomials). Because of $\ortho{L}_{n,k}=L_{n,k}$ there is also a case of self-reciprocity: $L_{n,k}=(-1)^{n-k}L_{-k,-n}$.
\end{exms}

Looking back for a moment at the proof of \eq{reciprocity_A_B} and at the statements of Corollaries \ref{family_A} and \ref{family_B}, it quickly becomes clear that $\widehat{P}_{n,k}$ is the real hero of the story. In the remainder of this section, we shall deepen that  impression and deal with some other interesting reciprocity properties of the potential polynomials. 

\begin{prop}\label{reciprocity_potpoly_1}
Let $n,k$ be integers with $k\leq n\neq 0$. Then
\[
	\widehat{P}_{n-k,k}(\tfrac{X_1}{1},\ldots,\tfrac{X_{n-k+1}}{n-k+1})=\frac{k}{n}\cdot\widehat{P}_{n-k,-n}(\tfrac{A_{1,1}}{1},\ldots,\tfrac{A_{n-k+1,1}}{n-k+1}).
\]
\end{prop}
\begin{proof}
Using the redefinition of the Stirling polynomials via Corollaries \ref{family_B} and \ref{family_A} we have
{\allowdisplaybreaks
\begin{alignat*}{4}
	\widehat{P}_{n-k,k}(\tfrac{X_1}{1},\ldots,&\tfrac{X_{n-k+1}}{n-k+1})=\binom{n}{k}^{-1}\mspace{-18mu}\cdot B_{n,k}\\
	&=\frac{k!}{n!}(n-k)!A_{n,k}(A_{1,1},\ldots,A_{n-k+1,1})&\hspace*{-9em}\text{(by \eqref{B_A_representable})}&\\
	&=\frac{k}{n}\binom{n-1}{k-1}^{-1}\mspace{-18mu}\cdot A_{n,k}(A_{1,1},\ldots,A_{n-k+1,1})\\
	&=\frac{k}{n}\cdot\widehat{P}_{n-k,n}(\widehat{R}_0(\tfrac{A_{1,1}}{1}),\ldots,\widehat{R}_{n-1}(\tfrac{A_{1,1}}{1},\ldots,\tfrac{A_{n,1}}{n}))\\
	&=\frac{k}{n}\cdot\widehat{P}_{n-k,-n}(\tfrac{A_{1,1}}{1},\ldots,\tfrac{A_{n-k+1,1}}{n-k+1}).&\hspace*{-4.5em}\text{(Cor.\,\ref{multiplication_rule}\,(ii))}\qedhere
\end{alignat*}
}
\end{proof}

\begin{rem}\label{schur_jabotinsky_thm}
\proposition{reciprocity_potpoly_1} may be regarded as a reformulation of a reciprocity law known as the \emph{Schur-Jabotinsky theorem} (see Jabotinsky \cite{jabo1953} and Gessel \cite{gess2016}). To see this we have to temporarily allow (formal) Laurent series over $\const$ as functions. Suppose $\varphi(x)=\sum_{n\geq 1}\varphi_n\tfrac{x^n}{n!}\in\invfuncs$ and $k\in\integers$; then the $k$th powers of $\varphi$ and $\inv{\varphi}$ can be expanded into such series: $\varphi(x)^k=\sum_{n} a_{n,k}x^n$ and $\inv{\varphi}(x)^k=\sum_{n} b_{n,k}x^n$. Since $a_{n,k}=b_{n,k}=0$ for $n<k$, we assume $n\geq k$. In the case $k\geq 0$ the Taylor coefficient $n!a_{n,k}=n![x^n]\varphi(x)^k$ is equal to $k!B_{n,k}^{\varphi}(0)$, whence by \corollary{family_B}
\begin{equation*}
		a_{n,k}=\frac{1}{(n-k)!}\widehat{P}_{n-k,k}(\tfrac{\varphi_{1}}{1},\ldots,\tfrac{\varphi_{n-k+1}}{n-k+1}).
\end{equation*}
This formula represents $a_{n,k}$ also for $n,k\in\integers$ with $n\geq k$ and can thus be applied to $\inv{\varphi}^k$:
\[
	b_{-k,-n}=\frac{1}{(-k+n)!}\widehat{P}_{-k+n,-n}(\tfrac{\inv{\varphi}_{1}}{1},\ldots,\tfrac{\inv{\varphi}_{n-k+1}}{n-k+1}).
\]
\proposition{reciprocity_potpoly_1} now yields the reciprocity law in question in the form given by Gessel [Eq.\,(2.1.11), ibid.]: 
\[a_{n,k}=\frac{k}{n}\,b_{-k,-n}\quad (n\neq 0).\]
\end{rem}

In the proofs of \eq{reciprocity_A_B} and of \proposition{reciprocity_potpoly_1}, the statement (ii) of \corollary{multiplication_rule} has been used. It seems to be the simplest reciprocity law for the potential polynomials and obviously it is also true for the variant without hat:
\begin{equation}\label{basic_reciprocity_potpoly}
	P_{n,-k}=P_{n,k}(R_1,\ldots,R_n).
\end{equation}
Comtet \cite{comt1974} has established a formula that expresses $P_{n,-k}$ as a linear combination of the $P_{n,1},\ldots,P_{n,n}$ (see Theorem C, p.\,142, ibid.). It is stated in the following proposition and given a new and straightforward proof.

\begin{prop}\label{comtet_thm_C_potpoly}
Let $n,k$ be integers with $n\geq 0$ and $(-k)\notin\left\{0,1,\ldots,n\right\}$. Then
\[
		P_{n,-k}=k\binom{n+k}{n}\sum_{j=0}^{n}(-1)^{j}\frac{1}{k+j}\binom{n}{j}P_{n,j}.
\]
\end{prop}

\begin{proof}
By \eq{potential_polynomial_1} we have
\[
	P_{n,-k}=\sum_{j=0}^n(-k)_j B_{n,j}=\sum_{j=0}^n j!\binom{-k}{j} B_{n,j}.
\]
Replacing the binomial term by its `reciprocal' in the sense of \eq{reciprocity_combinations} and applying Bertrand's formula \eqref{bertrand_formula} to $B_{n,j}$ thus yields
\begin{align*}
	P_{n,-k}&=\sum_{j=0}^n j!(-1)^j\binom{k-1+j}{j}\cdot\frac{1}{j!}\sum_{r=0}^n(-1)^{j-r}\binom{j}{r}P_{n,r}\\
	&=\sum_{r=0}^n\sum_{j=0}^n\binom{j}{r}\binom{k-1+j}{j}(-1)^r P_{n,r}.\tag{*}
\end{align*}
Clearly we have $\binom{k-1+j}{j}=\binom{k-1+j}{k-1}=\frac{k}{k+j}\binom{k+j}{k}$ and by an easy inductive argument
\[
	\sum_{j=0}^n\frac{1}{k+j}\binom{k+j}{k}\binom{j}{r}=\frac{1}{k+r}\binom{k+n}{k}\binom{n}{r}.
\]
Substituting this into (*) gives the assertion.
\end{proof}

Comtet proved a slightly stronger version of \proposition{comtet_thm_C_potpoly}, where $k$ can be a complex number. The following theorem shows how the statement may be extended in another direction.

\begin{thm}\label{generalized_thm_C_potpoly}
Let $m,n,k\in\integers$ with $m\geq n\geq 0$ and $(-k)\notin\left\{0,1,\ldots,m\right\}$. Then
\[
		P_{n,-k}=k\binom{m+k}{m}\sum_{j=0}^{m}(-1)^{j}\frac{1}{k+j}\binom{m}{j}P_{n,j}.
\]
\end{thm}

\begin{proof}
We assume $n$ to be a fixed non-negative integer and proceed by induction on $m$. The basis step $m=n$ is already done by \proposition{comtet_thm_C_potpoly}. Let $S_m$ denote the right-hand side of the induction hypothesis. For the inductive step it is then enough to show that the difference $\Delta:=S_{m+1}-S_{m}$ is equal to zero. Applying some elementary properties of the binomial numbers yields
\begin{equation*}
	\Delta=\binom{m+k}{k-1}\sum_{j=0}^{m+1}(-1)^j\binom{m}{j}\bigg(\frac{k}{k+j}P_{n,j}-\frac{k+m+1}{k+j+1}P_{n,j+1}\bigg).\tag{*}
\end{equation*}
We use the abbreviations
\[
	a_j:=(-1)^j\binom{m}{j}\frac{k}{k+j}\quad\text{and}\quad b_j:=(-1)^j\binom{m}{j}\frac{k+m+1}{k+j+1}
\]
and rewrite (*) as
\[
	\Delta=\binom{m+k}{k-1}\bigg(a_0 P_{n,0}+\sum_{j=0}^{m-1}(a_{j+1}-b_j)P_{n,j+1}-b_m P_{n,m+1}\bigg),
\]
where $a_0 P_{n,0}=\kronecker{n}{0}$ and $-b_m P_{n,m+1}=(-1)^{m+1}P_{n,m+1}$. A little calculation gives
\[
	a_{j+1}-b_j=(-1)^{j+1}\binom{m+1}{j+1}.
\]
In summary it results
\begin{align*}
	\binom{m+k}{k-1}^{-1}\Delta&=\kronecker{n}{0}+\sum_{j=0}^{m-1}(-1)^{j+1}\binom{m+1}{j+1}P_{n,j+1}+(-1)^{m+1}P_{n,m+1}\\
	&=\kronecker{n}{0}+\sum_{j=1}^{m+1}(-1)^j\binom{m+1}{j}P_{n,j}\\
	&=\sum_{j=0}^{m+1}(-1)^j\binom{m+1}{j}P_{n,j}.
\end{align*}
By means of the Bertrand formula \eqref{bertrand_formula} one readily verifies that the last sum is equal to $(-1)^{m+1}(m+1)!B_{n,m+1}$, which in fact vanishes for \mbox{$m\geq n$}.
\end{proof}

Comtet's formula for $P_{n,-k}$ in \proposition{comtet_thm_C_potpoly} has a striking resemblance to a well-known binomial transformation attributed to Melzak \cite{melz1951,megp1953}. Let $p(x)$ be any polynomial of degree $\leq n$.  If we now simply write $p(x+k)$ instead of $P_{n,-k}$ and $p(x-j)$ instead of $P_{n,j}$, the result is Melzak's formula
\begin{equation}\label{melzak_formula}
	p(x+k)=k\binom{n+k}{n}\sum_{j=0}^{n}(-1)^{j}\frac{1}{k+j}\binom{n}{j}p(x-j).
\end{equation}

Recently, several authors have presented new studies on this remarkable identity. Quaintance and Gould \cite{qugo2016} devoted Chapter 7 of their monograph to the subject. Boyadzhiev \cite{boya2016b} and Abel \cite{abel2020} provided extensions (concerning the degree of $p(x)$) and new proofs. The former author even traced \eq{melzak_formula} back up to Nielsen's treatise \cite{niel1923} on Bernoulli numbers.

We finally will derive an extended version of \eqref{melzak_formula} from \theorem{generalized_thm_C_potpoly}. This demonstrates that the successful replacement of $P_{n,-k}$ by $p(x+k)$ in Comtet's formula is not an accident. Rather, it turns out in this way that \theorem{generalized_thm_C_potpoly} is the more comprehensive statement.

\begin{thm}\label{thm_melzak}
Let $m,n$ be integers with $m\geq n\geq 0$ and $f_n(x)\in\complex[x]$ any polynomial of degree $n$. Then for all $k\in\integers\setminus\left\{-m,\ldots,-1,0\right\}$
\[
	f_n(x+k)=k\binom{m+k}{m}\sum_{j=0}^m(-1)^j\frac{1}{k+j}\binom{m}{j}f_n(x-j).
\]
\end{thm}

\begin{proof}
We apply $\comp R_{\num}$ on both sides of the equation of \theorem{generalized_thm_C_potpoly}. According to the basic reciprocity law \eqref{basic_reciprocity_potpoly} this gives
\begin{equation}\label{thm_melzak_eq1}
	P_{r,k}=k\binom{m+k}{m}\sum_{j=0}^{m}(-1)^{j}\frac{1}{k+j}\binom{m}{j}P_{r,-j}
\end{equation}
for every non-negative $r\leq n$. Unification then yields
\begin{equation}\label{thm_melzak_eq2}
	k^r=k\binom{m+k}{m}\sum_{j=0}^m(-1)^j\frac{1}{k+j}\binom{m}{j}(-j)^r.
\end{equation}
By assumption we have $f_n(x)=a_{0}+a_{1}x+\cdots+a_{n}x^{n}$, $a_n\neq 0$, and hence from \eqref{thm_melzak_eq2}
\begin{equation}\label{thm_melzak_eq3}
	f_n(k)=k\binom{m+k}{m}\sum_{j=0}^m(-1)^j\frac{1}{k+j}\binom{m}{j}f_n(-j).
\end{equation}
Setting $g_{n,k}(x):=f_n(x+k)-f_n(k)$ we obtain $g_{n,k}(0)=0$ and
\begin{equation}\label{thm_melzak_eq4}
	g_{n,k}(x)=\sum_{r=1}^n a_{r}((x+k)^r-k^r)=\sum_{j=1}^n\bigg(\underbrace{\sum_{r=j}^n\binom{r}{j}a_{r}k^{r-j}}_{(\ast)}\bigg)x^j.
\end{equation}
Abbreviate the inner sum $(\ast)$ to $b_j(k)$.

We will now show that $g_{n,k}$ satisfies the Melzak formula. We prove that there are constants $c_1,\ldots,c_n\in\complex$ such that for all $n\geq 1$
\begin{equation}\label{thm_melzak_eq5}
	P_{n,k}(c_{1}x,\ldots,c_{n}x)=g_{n,k}(x).
\end{equation}
According to \eqref{power_1case} and because of the homogeneity of the partial Bell polynomials, the left-hand side of \eqref{thm_melzak_eq5} can be written as 
\begin{equation}\label{thm_melzak_eq6}
	P_{n,k}(c_{1}x,\ldots,c_{n}x)=\sum_{j=1}^n j!\binom{k}{j}B_{n,j}(c_1,\ldots,c_{n-j+1})x^j.
\end{equation}
Equating the coefficients of $x^j$ in \eqref{thm_melzak_eq4} and \eqref{thm_melzak_eq6} then yields the following equation system for the constants $c_1,\ldots,c_n$:
\begin{equation}\label{thm_melzak_eq7}
	B_{n,j}(c_1,\ldots,c_{n-j+1})=\frac{b_j(k)}{j!\binom{k}{j}}\quad(1\leq j\leq n).
\end{equation}
Recall that according to \remark{Bell_equations} the system \eqref{thm_melzak_eq7} is solvable if and only if $c_1\in\complex$ exists such that 
\begin{equation}\label{thm_melzak_eq8}
	c_1^n=\frac{b_n(k)}{n!}\binom{k}{n}^{\negmedspace-1}.
\end{equation}
Therefore, we choose $c_1$ to be any of the $n$th roots of the right-hand side of \eqref{thm_melzak_eq8}. Then the remaining constants $c_2,\ldots,c_n$ are uniquely determined (each of them depending on $k$, for example, $c_n=b_1(k)/k$). We now put $r=n$ and replace in \eqref{thm_melzak_eq1} each $X_j$ by $c_{j}x$, $1\leq j\leq n$. Since \eq{thm_melzak_eq5} is true for the constants $c_1,\ldots,c_n$, we are led to
\begin{equation}\label{thm_melzak_eq9}
	g_{n,k}(x)=k\binom{m+k}{m}\sum_{j=0}^m(-1)^j\frac{1}{k+j}\binom{m}{j}g_{n,-j}(x).
\end{equation}
Because of $f_n(x+k)=f_n(k)+g_{n,k}(x)$ the assertion follows from \eq{thm_melzak_eq3} and \eq{thm_melzak_eq9}.
\end{proof}
\appendix

\end{document}